\documentclass[a4paper,11pt]{amsart} 

\usepackage{amsmath} 
\usepackage{amssymb} 
\usepackage{amsthm} 
\usepackage{enumerate} 
\usepackage{mathrsfs} 

\usepackage{tikz-cd} 
\usepackage[all,arc]{xy}
\usepackage{graphicx}
\usepackage{color}
\usetikzlibrary{arrows}

\newtheorem{introtheorem}{Theorem}
\newtheorem{theorem}{Theorem}[section]
\newtheorem{introcitation}{Citation}
\newtheorem{Citation}[theorem]{Citation}

\newtheorem{corollary}[theorem]{Corollary}

\newtheorem{proposition}[theorem]{Proposition}
\newtheorem{lemma}[theorem]{Lemma}
\newtheorem{conjecture}[theorem]{Conjecture}

\newtheorem{introquestion}{Question}

\theoremstyle{definition}

\newtheorem{definition}[theorem]{Definition}

\newtheorem{example}{Example}
\newtheorem{introexample}{Example}

\newtheorem{notation}[theorem]{Notation}

\theoremstyle{remark}
\newtheorem{remark}[theorem]{Remark}

\newcommand{\F}{\mathbb{F}}
\newcommand{\N}{\mathbb{N}}
\newcommand{\Q}{\mathbb{Q}}

\newcommand{\R}{\mathbb{R}}
\newcommand{\Z}{\mathbb{Z}}

\newcommand{\defeq}{\mathrel{\mathop{:}}=}
\newcommand{\eqdef}{=\mathrel{\mathop{:}}}
\newcommand{\fac}{\leqslant}
\newcommand{\pinf}{\partial_\infty}

\newcommand{\hgt}{\operatorname{ht}}
\newcommand{\SigmaStd}{\operatorname{\underline{\Sigma}}}

\newcommand{\vstd}{\operatorname{\underline{v}}}
\newcommand{\thick}{\operatorname{th}}

\newcommand{\abs}[1]{\lvert #1 \rvert}

\newcommand\Set[2]{\{\,#1\mid#2\,\}}

\DeclareMathOperator{\Aut}{Aut}

\DeclareMathOperator{\BN}{BN}

\DeclareMathOperator{\CAT}{CAT}
\DeclareMathOperator{\Ch}{Ch}

\DeclareMathOperator{\conv}{conv}

\DeclareMathOperator{\dist}{dist}
\DeclareMathOperator{\End}{End}

\DeclareMathOperator{\GL}{GL}
\DeclareMathOperator{\Hom}{Hom}

\DeclareMathOperator{\hor}{hor}
\DeclareMathOperator{\id}{id}
\DeclareMathOperator{\Isom}{Isom}
\DeclareMathOperator{\lk}{lk}
\DeclareMathOperator{\Opp}{Opp}
\DeclareMathOperator{\op}{op}
\DeclareMathOperator{\PGL}{PGL}
\DeclareMathOperator{\pr}{pr}
\DeclareMathOperator{\RGD}{RGD}
\DeclareMathOperator{\SL}{SL}
\DeclareMathOperator{\SO}{SO}
\DeclareMathOperator{\SOL}{SOL}
\DeclareMathOperator{\st}{st}
\DeclareMathOperator{\St}{St}

\DeclareMathOperator{\supp}{supp}

\numberwithin{equation}{section}

\title[The $\Sigma$-invariants of $S$-arithmetic subgroups of Borel groups]
{The $\Sigma$-invariants of $S$-arithmetic subgroups of Borel groups}

\author[E.~Schesler]{Eduard Schesler}

\address{Fakult\"{a}t f\"{u}r Mathematik und Informatik, FernUniversit\"{a}t in Hagen, \newline 58084 Hagen, Germany}

\email{eduard.schesler@fernuni-hagen.de}

\thanks{The author was supported by the DFG grant WI 4079/4 within the SPP 2026 Geometry at infinity.}

\date{\today}

\begin{document}

\begin{abstract}
Given a Chevalley group $\mathcal{G}$ of classical type and a Borel subgroup $\mathcal{B} \subseteq \mathcal{G}$, we compute the $\Sigma$-invariants of the $S$-arithmetic groups $\mathcal{B}(\Z[1/N])$, where $N$ is a product of large enough primes.
To this end, we let $\mathcal{B}(\Z[1/N])$ act on a Euclidean building $X$ that is given by the product of Bruhat--Tits buildings $X_p$ associated to $\mathcal{G}$, where $p$ runs over the primes dividing $N$.
In the course of the proof we introduce necessary and sufficient conditions for convex functions on $\CAT(0)$-spaces to be continuous.
We apply these conditions to associate to each simplex at infinity $\tau \subset \pinf X$ its so-called parabolic building $X^{\tau}$, which we study from a geometric point of view.
Moreover, we introduce new techniques in combinatorial Morse theory, which enable us to take advantage of the concept of essential $n$-connectivity rather than actual $n$-connectivity.
Most of our building theoretic results are proven in the general framework of spherical and Euclidean buildings.
For example, we prove that the complex opposite each chamber in a spherical building $\Delta$ contains an apartment, provided $\Delta$ is thick enough and $\Aut(\Delta)$ acts chamber transitively on $\Delta$.
\end{abstract}

\maketitle

\setcounter{tocdepth}{1}

\section*{Introduction}

A group $G$ is said to be of \emph{type $F_n$} if it admits a classifying space with finite $n$-skeleton.
Being of type $F_n$ generalizes the fundamental concepts of finite generation, which is equivalent to type $F_1$, and finite presentability, which is equivalent to type $F_2$.
If $G$ is of type $F_n$ for every $n \in \N_0$, then we say that $G$ is of type $F_{\infty}$.
Regarding the appearance of ``finiteness'' in these properties, they are often referred to as finiteness properties.
Probably one of the most natural questions that was raised in the study of finiteness properties is whether for each $n \in \N_0$ there is a group of type $F_n$ but not of type $F_{n+1}$.
Meanwhile, there are many families of groups known to confirm this question, see e.g. \cite{BestvinaBrady97}, \cite{BuxKoehlWitzel2013}, \cite{SkipperWitzelZaremsky2019} and the examples below.
The first such family of solvable groups that was found consists of the groups
\[
A_n(\Z[1/p]) \defeq
\Small{\begin{pmatrix}
1 & \ast & \cdots & \cdots & \ast \\
0 & \ast & \cdots & \cdots & \ast \\
\vdots & \ddots & \ddots & \ddots & \vdots \\
0 & \dots & 0 & \ast & \ast \\
0 & \dots & \dots & 0 & 1
\end{pmatrix}}
\leq \GL_n(\Z[1/p]),
\]
where $p$ is a prime.
These groups, nowadays known as Abels' groups, where introduced by Abels~\cite{Abels1979} in the case $n = 4$ to provide the first examples of finitely presented solvable groups with non-finitely presented quotients.
Shortly afterwards, it was shown by Bieri~\cite{Bieri1980} that $A_n(\Z[1/p])$ is not of type $F_{n-1}$ for every $n \geq 2$.
Later, Abels and Brown~\cite{AbelsBrown85} resolved the question about finiteness properties of Abels groups by showing that each $A_n(\Z[1/p])$ is of type $F_{n-2}$.
Since then, Abels groups remained a rich source of groups with exotic properties (see e.g.~\cite{Cornulier2007},~\cite{Carrion2013},~\cite{Becker2019}) and several generalizations of them emerged.
Before proceeding to them, let us observe that $A_n(\Z[1/p])$ has an analogue in $\SL_n(\Z[1/p])$ that is given by
\[
\mathcal{A}_n(\Z[1/p]) \defeq
\left\{
\Small{\begin{pmatrix}
a_{11} & \ast & \cdots & \ast \\
0 & \ddots & \ddots & \vdots \\
\vdots & \ddots & \ddots & \ast \\
0 & \dots & 0 & a_{nn}
\end{pmatrix}}
\in \SL_n(\Z[1/p]) \ \Biggl\lvert \ v_p(a_{11}) = v_p(a_{nn}) \right\},
\]
where $v_p$ denotes the $p$-adic valuation.
To see that $A_n(\Z[1/p])$ and $\mathcal{A}_n(\Z[1/p])$ are indeed related, one can consider their images in $\PGL_n(\Q)$.
Then it is easy to see that $A_n(\Z[1/p])$ and $\mathcal{A}_n(\Z[1/p])$ are congruent after dividing out appropriate finite subgroups.
In particular, the finiteness properties of these groups coincide.
Since we will always work with Chevalley group schemes, such as $\SL_n$, it will be convenient for us to consider the groups $\mathcal{A}_n(\Z[1/p])$ from now on.
Let $\mathcal{B}_n(\Z[1/p]) \leq \SL_n(\Z[1/p])$ and $B_n(\Z[1/p]) \leq \GL_n(\Z[1/p])$ denote the corresponding subgroups of upper triangular matrices.
Note that we can describe $\mathcal{A}_n(\Z[1/p])$ as the kernel of a \emph{character} of $\mathcal{B}_n(\Z[1/p])$, i.e.\ a homomorphism to $\R$, that is given by
\[
\widetilde{\chi} \colon \mathcal{B}_n(\Z[1/p]) \rightarrow \R,\ (a_{ij}) \mapsto v_p(a_{11}) - v_p(a_{nn}).
\]
Similarly, there is a character of $B_n(\Z[1/p])$ whose kernel contains $A_n(\Z[1/p])$ as a subgroup of finite index.
In this article we study generalizations of $\mathcal{A}_n(\Z[1/p])$ by replacing
\begin{itemize}
\item $\widetilde{\chi}$ with an arbitrary character $\chi \in \Hom(\mathcal{B}_n(\Z[1/p]),\R)$,
\item $\Z[1/p]$ with the ring $\Z[1/N]$ for arbitrary $N \in \N$,
\item $\mathcal{B}_n$ with Borel subgroup schemes in general Chevalley group schemes.
\end{itemize}
The second type of generalization was considered by Lyul'ko~\cite{Lyulko86} who showed that $A_n(\Z[1/N])$ is finitely presented for all $n,N \in \N$ with $n \geq 4$.
In the case of certain $S$-arithmetic rings in positive characteristic, these generalizations were studied, among other things, in~\cite{CornulierTessera2013},~\cite{Rego2021}, and~\cite{Bux04} respectively for the second, the latter two, and all three types of replacements.
By varying the characters of $B_n(\Z[1/p])$, Witzel~\cite{Witzel13} provided examples of groups whose finiteness properties can differ arbitrarily much from their corresponding Bredon finiteness properties (see~\cite{FluchWitzel12} for the relevant background).
In fact, each of the characters in~\cite{Witzel13} has a kernel with the same (ordinary) finiteness properties as $A_n(\Z[1/p])$.
In view of the following result, which is a consequence of the $\SL_n$-case in Theorem~\ref{thm:B}, we see that the kernel of a character $\chi \in \Hom(\mathcal{B}_n(\Z[1/p]),\R)$ can have arbitrary finiteness properties.
To formulate it, let us fix some terminology. 
Let $\mathcal{U}_n(\Z[1/N]) \leq \mathcal{B}_n(\Z[1/N])$ denote the subgroup of unipotent matrices.
For every $1 \leq k < n$ and every prime $p \in \N$ dividing $N$, we define the character
\[
\chi_{k,p} \colon \mathcal{B}_n(\Z[1/N]) \rightarrow \R, (a_{i,j}) \mapsto v_p(a_{k+1,k+1})-v_p(a_{k,k}).
\]
Let $\Delta_{N,n}^{(0)}$ denote the union of these characters, which can be easily seen to be a basis of $\Hom(\mathcal{B}_n(\Z[1/N]),\R)$.
Let $C_{N,n}$ be the set of non-trivial characters of the form $\sum \limits_{\chi \in \Delta_{N,n}^{(0)}} \lambda_{\chi} \cdot \chi$, where all coefficients $\lambda_{\chi}$ are non-negative.
For each $k \in \N_0$ we write $C_{N,n}^{(k)} \subseteq C_{N,n}$ to denote the subset of characters $\sum \limits_{\chi \in \Delta_{N,n}^{(0)}} \lambda_{\chi} \cdot \chi$ with at most $k$ non-zero coefficients.

\begin{introtheorem}\label{introtheorem:Fn-SLn}
Let $k,n,N \in \N$ and let $H \leq \mathcal{B}_n(\Z[1/N])$ be a subgroup.
Suppose that $H$ contains $\mathcal{U}_n(\Z[1/N])$.
Then the following hold.
\begin{enumerate}
\item $H$ is of type $F_{\infty}$ if and only if $\chi(H) \neq 0$ for every $\chi \in C_{N,n}$.
\item If $H$ is of type $F_k$, then $\chi(H) \neq 0$ for every $\chi \in C_{N,n}^{(k)}$.
\item Suppose that every prime factor $p$ of $N$ satisfies $p \geq 2^{n-2}$.
Then $H$ is of type $F_k$ if and only if $\chi(H) \neq 0$ for every $\chi \in C_{N,n}^{(k)}$.
\end{enumerate}
\end{introtheorem}

As a consequence of Theorem~\ref{introtheorem:Fn-SLn}, we obtain an abundance of new examples of solvable groups with interesting finiteness properties.
Since each character $\chi$ of $\mathcal{B}_n(\Z[1/N])$ vanishes on $\mathcal{U}_n(\Z[1/N])$, we can apply Theorem~\ref{introtheorem:Fn-SLn} on the kernel of $\chi$ to obtain such examples.
Suppose that $\chi$ maps onto $\Z$.
Then every character that vanishes on $\ker(\chi)$ is given by $\lambda \cdot \chi$ for some $\lambda \in \R$.
Hence $\ker(\chi)$ is of type $F_{\infty}$ if and only if not all coefficients of $\chi$ with respect to $\Delta_{N,n}^{(0)}$ have the same sign.
Similarly, this case allows us to deduce that $\ker(\chi)$ is of type $F_k$ if $\chi$ has more than $k$ non-vanishing coefficients with respect to $\Delta_{N,n}^{(0)}$ and $p \geq 2^{n-2}$ for every prime factor $p$ of $N$.
Let us look at some examples.

\begin{introexample}\label{exa:1}
For every prime $p$, the group
\[
H_1 \defeq \left\{
\begin{pmatrix}
p^k & \ast & \ast \\
0 & 1 & \ast \\
0 & 0 & p^{-k} \\
\end{pmatrix}
\in \mathcal{B}_3(\Z[1/p]) \mid k \in \Z
\right\}
\]
is of type $F_{\infty}$.
To see this, we consider the character $\chi \defeq \chi_{1,p} - \chi_{2,p}$ of $\mathcal{B}_3(\Z[1/p])$ and a matrix $(a_{i,j}) \in \ker(\chi)$.
Then $v_p(a_{11}) - 2v_p(a_{22}) + v_p(a_{33}) = 0$.
Together with
\[
v_p(a_{11}) + v_p(a_{22}) + v_p(a_{33})
= v_p(\det((a_{i,j})))
= 0,
\]
this shows that a matrix $(a_{i,j}) \in \mathcal{B}_3(\Z[1/p])$ lies in $\ker(\chi)$ if and only if $v_p(a_{11}) = - v_p(a_{33})$ and $v_p(a_{22}) = 0$.
Since $\Z[1/p]^{\times} = \Set{\pm p^{k}}{k \in \Z}$, we see that $H_1$ is a finite index subgroup of $\ker(\chi)$, and is therefore of type $F_{\infty}$.
\end{introexample}

\begin{introexample}\label{exa:4}
Let $p,q$ be distinct primes.
The group
\[
\left\{
\begin{pmatrix}
p^{k_1}q^{\ell_1} & \ast & \ast \\
0 & p^{k_2}q^{\ell_2} & \ast \\
0 & 0 & p^{k_3}q^{\ell_3} \\
\end{pmatrix}
\in \mathcal{B}_3(\Z[1/pq]) \mid k_1+\ell_1+2 \ell_2 = 3 \ell_3+k_3
\right\},
\]
denoted by $H_2$, is of type $F_3$ but not $F_4$.
To see this, consider the character $\chi \defeq \chi_{1,p}+\chi_{2,p}+\chi_{1,q}+3\chi_{2,q}$, whose kernel is easily seen to contain $H_2$ as a finite index subgroup.
Now the claim follows since $\chi$ has exactly $4$ non-vanishing coefficients, all of which have the same sign.
\end{introexample}

\begin{introexample}\label{exa:3}
Let $k,n,N \in \N$ and let $S \subset \N$ be the set of prime factors of $N$.
Let $\chi$ denote the sum of all characters in $\Delta_{N,n}^{(0)}$.
By applying $\chi$ to a matrix $(a_{ij}) \in \mathcal{B}_n(\Z[1/N])$ we obtain
\[
\chi((a_{ij}))
= \sum \limits_{p \in S} \sum \limits_{k=1}^{n-1} \chi_{k,p}((a_{ij}))
= \sum \limits_{p \in S} v_p(a_{11}) - \sum \limits_{p \in S} v_p(a_{nn}).
\]
From this we see that the group
\[
H_3 \defeq \left\{
\SMALL{\begin{pmatrix}
\prod \limits_{p \in S} p^{e_p} & \ast & \cdots & \ast \\
0 & \ast & \ddots & \vdots \\
\vdots & \ddots & \ast & \ast \\
0 & \dots & 0 & \prod \limits_{p \in S} p^{f_p}
\end{pmatrix}}
\in \mathcal{B}_n(\Z[1/N]) \ \Biggl\lvert \ \sum \limits_{p \in S} e_p = \sum \limits_{p \in S} f_p \right\}
\]
is a finite index subgroup of $\ker(\chi)$.
In view of Theorem~\ref{introtheorem:Fn-SLn}, we can therefore deduce that $H_3$ is not of type $F_{\abs{S} \cdot (n-1)}$.
If moreover $p \geq 2^{n-2}$ for every $p \in S$, then $H_3$ is of type $F_{\abs{S} \cdot (n-1)-1}$.
\end{introexample}

\subsection*{$\Sigma$-invariants}

In order to prove Theorem~\ref{introtheorem:Fn-SLn}, we compute the so-called $\Sigma$-invariants of $\mathcal{B}_n(\Z[1/N])$.
These invariants had their origin in the study of metabelian groups.
In fact, the first definition of $\Sigma$-invariants, which was introduced by Bieri and Strebel in~\cite{BieriStrebel80}, was only applicable for metabelian groups.
After a series of generalizations (see e.g.~\cite{BieriNeumannStrebel87},~\cite{Brown87}, and \cite{Meigniez90}),
the version of $\Sigma$-invariants that will be studied in this article was defined by Bieri and Renz in~\cite{BieriRenz88}.
Unlike most invariants in group theory, the $\Sigma$-invariants of a finitely generated group $G$ are not algebraic structures themselves but rather appear as geometric subsets on a sphere associated to $G$.
This sphere, known as the \emph{character sphere} of $G$, is given by
\[
S(G) \defeq (\Hom(G,\R) \backslash \{0\})/ \hspace{-1.5mm} \sim,
\]
where $\chi \sim \psi$ if there is some $\lambda > 0$ with $\lambda \cdot \chi = \psi$.
If $G$ is of type $F_n$ for some $n \in \N \cup \{\infty\}$, then the $n$th $\Sigma$-invariant of $G$, denoted by $\Sigma^n(G)$, can be defined as a certain subset of $S(G)$.
Together these subsets form a decreasing sequence
\[
S(G) \supseteq \Sigma^1(G) \supseteq \Sigma^2(G) \supseteq \dots \supseteq \Sigma^{\infty}(G).
\]
To make this more precise, suppose that $G$ is of type $F_n$ and let $\chi$ be a non-trivial character of $G$.
Then we can choose a contractible cell complex $X$ on which $G$ acts freely such that $G \backslash X$ has finite $n$-skeleton.
Moreover, it is shown in~\cite[Konstruktion II.2.2]{Renz88} that there is a continuous function $h \colon X \rightarrow \R$ with $h(g(x)) = \chi(g) + h(x)$ for all $g \in G, x \in X$.
Using this terminology, $\chi$ represents a class in $\Sigma^n(G)$ if and only if the system $(h^{-1}([r,\infty)))_{r \in \R}$ is essentially $(n-1)$-connected (see Definition~\ref{def:essentially-connected}).
One of the main motivations in studying $\Sigma$-invariants is their relation to finiteness properties of groups.
This can be demonstrated by the following result, which was proven by Renz in his thesis~\cite[Satz C]{Renz88}.

\begin{introcitation}\label{introcitation:relation-sigma-fn-intro}
Let $G$ be a group of type $F_n$ and let $[G,G] \subseteq H \subseteq G$ be a subgroup.
Then $H \text{ is of type } F_n \Leftrightarrow \Set{\chi \in S(G)}{\chi(H)=0} \subseteq \Sigma^n(G)$.
\end{introcitation}

In particular, this result tells us that the $\Sigma$-invariants of a group $G$ contain the information about the finiteness properties of all subgroups $[G,G] \subseteq H \subseteq G$.
This also indicates that computing the $\Sigma$-invariants of a group is very difficult in general.
In fact, there are not many examples of groups available of which all $\Sigma$-invariants are known.
In the case of right-angled Artin groups a full computation of its $\Sigma$-invariants was achieved by Meier, Meinert, and VanWyk~\cite{MeierMeinertVanWyk98} and reproved by Bux and Gonzalez with more geometric methods~\cite{BuxGonzalez99}.
For Thompsons group $F$, all $\Sigma$-invariants were computed by Bieri, Geoghegan, and Kochloukova in~\cite{BieriGeogheganKochloukova2010}.
The main goal of this article is to determine the $\Sigma$-invariants of the $S$-arithmetic groups $\mathcal{B}(\Z[1/N])$, where $\mathcal{B}$ is a Borel subgroup scheme in a Chevalley group scheme.
For $S$-arithmetic rings $\mathcal{O}_S$ in positive characteristic, partial computations of the $\Sigma$-invariants of $\mathcal{B}(\mathcal{O}_S)$ where obtained by Bux~\cite{Bux04}.
In the case of Chevalley groups of classical type, we completely determine the $\Sigma$-invariants of $\mathcal{B}(\Z[1/N])$ if the prime factors of $N$ are large enough.
Let us formulate our main result for the groups $\mathcal{B}_{n}(\Z[1/N])$, see Theorem~\ref{thm:B} for the general case.
To this end, we consider the images of $C_{N,n}$ and $C_{N,n}^{(k)}$ in $S(\mathcal{B}_{n}(\Z[1/N]))$, which we denote by $\Delta_{N,n}$, respectively $\Delta_{N,n}^{(k-1)}$.

\begin{introtheorem}\label{introthm:B-intro}
Let $n,N \in \N$ and let $S \subseteq \N$ be the set of prime factors of $N$.
The $\Sigma$-invariants of $\Gamma \defeq \mathcal{B}_n(\Z[1/N])$ satisfy the following.
\begin{enumerate}
\item $\Sigma^{\infty}(\Gamma) = S(\Gamma) \backslash \Delta_{N,n}$.
\item $\Sigma^k(\Gamma) \subseteq S(\Gamma) \backslash \Delta_{N,n}^{(k-1)}$ for every $k \in \N$.
\item If every $p \in S$ satisfies $p \geq 2^{n-2}$, then $\Sigma^k(\Gamma) = S(\Gamma) \backslash \Delta_{N,n}^{(k-1)}$ for every $k \in \N$.
\end{enumerate}
\end{introtheorem}

This partially confirms a conjecture of my PhD-supervisor Stefan Witzel.
In the setting of Theorem~\ref{introthm:B-intro} his conjecture says that for each $k \in \N$ the equality \[
\Sigma^k(\Gamma) = S(\Gamma) \backslash \Delta_{N,n}^{(k-1)},
\]
holds without any restriction on the prime factors of $N$, see Conjecture~\ref{conjecture:large-factors} for the general Conjecture.

\subsection*{Euclidean buildings}

To prove Theorem~\ref{introthm:B-intro}, we consider the action of $\mathcal{B}_n(\Z[1/N])$ on the product of Bruhat--Tits buildings $X_p$ corresponding to $\SL_n(\Q_p)$, where $p$ is a prime factor of $N$.
Let $X$ denote this product.
The cocompactness of this action and the fact that all cell stabilizers are of type $F_{\infty}$ will allows us to translate the problem of determining the $\Sigma$-invariants of $\mathcal{B}_n(\Z[1/N])$ into a natural geometric problem in $X$.
To state this problem, we consider an apartment $A$ in $X$, a chamber at infinity $\sigma \subseteq \pinf A$, and the corresponding retraction $\rho \defeq \rho_{\sigma,A} \colon X \rightarrow A$ (see Subsection~\ref{subsec:height-on-X} for a construction of $\rho$).
Since $X$ is a Euclidean building, we can view $A$ as a real vector space.
In particular, we can define the dual space of $A$, which we denote by $A^{\ast}$.
By precomposing the linear forms in $A^{\ast}$ with $\rho$, we obtain the vector space
$X^{\ast}_{\sigma} = \Set{\alpha \circ \rho}{\alpha \in A^{\ast}}$.
We think of each $h \in X^{\ast}_{\sigma}$ as a height function on $X$.
Using this point of view, we consider the filtration of $X$ that is given by the superlevelsets $X_{h \geq r} \defeq h^{-1}([r,\infty))$ for every $r \in \R$.
Now the geometric problem in $X$ mentioned above can be stated as follows.

\begin{introquestion}\label{que:essential-connectivity}
Given height function $h \in X^{\ast}_{\sigma}$ and an integer $k \geq 0$, is the system $(X_{h \geq r})_{r \in \R}$ essentially $k$-connected?
\end{introquestion}

Most part of this article is devoted to answer Question~\ref{que:essential-connectivity} in some generality.
In the $\SL_n$-case, we can formulate an ad-hoc version of our main result on Question~\ref{que:essential-connectivity} as follows, see Theorem~\ref{thm:A-new} for the general case.

\begin{introtheorem}\label{introtheorem:C}
Let $\alpha \in A^{\ast} \setminus \{0\}$, let $\eta \defeq \pinf(\alpha^{-1}([0,\infty))) \subseteq \pinf A$, and let $k$ denote the dimension of $\eta \cap \overline{\sigma}$.
Then $h \defeq \alpha \circ \rho \in X_{\sigma}^{\ast}$ satisfies the following.
\begin{enumerate}
\item If $k = \dim(\sigma)$, then $(X_{h \geq r})_{r \in \R}$ is essentially contractible.
\item If $k < \dim(\sigma)$, then $(X_{h \geq r})_{r \in \R}$ is not essentially $(\dim(X)-k-2)$-connected.
\item Suppose that each prime factor $p$ of $N$ satisfies $p \geq 2^{n-2}$ and that $k < \dim(\sigma)$.
Then $(X_{h \geq r})_{r \in \R}$ is essentially $(\dim(X)-k-3)$-connected.
\end{enumerate}
\end{introtheorem}

In the course of the proof of Theorem~\ref{introtheorem:C} we will always work in the general framework of arbitrary spherical and Euclidean buildings.
In fact, our main result on Question~\ref{que:essential-connectivity} can be stated in a purely building theoretic context, see Corollary~\ref{cor:A}.
As in Theorem~\ref{introtheorem:Fn-SLn} and Theorem~\ref{introthm:B-intro}, there is some evidence that Theorem~\ref{introtheorem:C} holds without any restriction on the primes dividing $N$.
An interesting special case arises if $h = \alpha \circ \rho \in X_{\sigma}^{\ast}$ is a Busemann function on $X$.
It is easily seen that $h$ is a Busemann function if and only if $\alpha$ is a Busemann function on $A$ that is centered at a point $\xi \in \overline{\sigma}$.
In this case $(X_{h \leq r})_{r \in \R}
$ is a system of horoball complements, whose essential connectivity properties where determined by Bux and Wortmann in~\cite{BuxWortman11} under the minor restriction that $X$ is a thick building.
A key ingredient for this result was provided by Schulz~\cite{Schulz13}.
He showed that, given a thick spherical building $\Delta$ and a point $x \in \Delta$, the set of points of distance more than $\pi / 2$ from $x$ is $(\dim(\Delta)-1)$-connected and not contractible.
Such spaces, which we shall call Schulz complexes, arise as descending links for certain Morse functions that were constructed in~\cite{BuxWortman11} and later in~\cite{BuxKoehlWitzel2013}.
In the case where $h \in X_{\sigma}^{\ast}$ is not a Busemann function, the natural analogue of a Schulz complex is given by the preimage of an open hemisphere under the retraction corresponding to an apartment and a chamber in a spherical building.
Such generalized Schulz complexes turn out to be dramatically more difficult than Schulz complexes.
Even worse, the natural generalization of the Morse function in~\cite{BuxWortman11} to the present context fails to produce the right essential connectivity properties, no matter how connected these generalized Schulz complexes are.
To overcome these difficulties, we will not work with a classical Morse lemma, which would typically produce a result on the connectivity properties of each superlevelset $X_{h \leq r}$.
Instead, we introduce techniques that allow us to take advantage of the concept of \emph{essential} connectivity properties rather than connectivity properties themselves.

\subsection*{The paper is organized as follows.}
In Section~\ref{sec:backgroud-topology} we recall some basic notions on polysimplicial complexes.
The necessary background on $\Sigma$-invariants is given in Section~\ref{sec:background-sigma-invariants}.
In Section~\ref{sec:coxeter-complex} we recall some well-known constructions related to spherical and Euclidean buildings.
Moreover, we introduce sufficient conditions for a spherical building $\Delta$ to ensure that the complex opposite each chamber in $\Delta$ contains an apartment.
In Section~\ref{sec:deconstr} we collect some auxiliary results on Euclidean Coxeter complexes that will be used in Section~\ref{sec:pos-dir} to prove that certain systems of superlevelsets $(X_{h \geq r})_{r \in \R}$ are essentially $(\dim(X)-2)$-connected, provided $h$ is chosen generically.
Using the same sort of genericity, we will prove in Section~\ref{sec:neg-dir} that $(X_{h \geq r})_{r \in \R}$ is not essentially $(\dim(X)-1)$-connected.
In Section~\ref{sec:convex-on-cat0} we introduce necessary and sufficient conditions for convex functions on $\CAT(0)$-spaces to be continuous.
These conditions will be used in Section~\ref{sec:parabolic} to review the concept of a parabolic building $X^{\tau}$ associated to a Euclidean building $X$ and a simplex at infinity $\tau \subseteq \pinf X$.
We take the opportunity to develop a geometric description of $X^{\tau}$ that allows us to apply methods from discrete Morse theory on $X^{\tau}$.
As an application, we reduce the general question on essential connectivity of $(X_{h \geq r})_{r \in \R}$ to the generic case from Sections~\ref{sec:pos-dir} and~\ref{sec:neg-dir}.
In Section~\ref{sec:the-geom-main} we prove a building theoretic version of our main result on $\Sigma$-invariants.
Section~\ref{sec:chevalley-groups} is devoted to give a short construction of the Chevalley group $\mathcal{G}(\Phi,\rho,K)$ associated to a root system $\Phi$, a representation $\rho$ of the corresponding complex semisimple Lie-Algebra $\mathcal{L}$, and a field $K$.
Moreover, we review the construction of the Bruhat--Tits buildings associated to $\mathcal{G}(\Phi,\rho,K)$ in the case where $K$ possesses a discrete valuation.
This will be done from a geometric point of view that allows us to formulate our main result about the $\Sigma$-invariants of $S$-arithmetic subgroups of Borel groups, which we finally prove in Section~\ref{sec:final}.

\subsection*{Acknowledgments}
This article grew out of my PhD thesis, which I wrote under the supervision of Stefan Witzel.
I would like to thank him for his guidance, understanding, and his endless patience.
I am grateful to Kevin Wortman for agreeing to act as a referee for my thesis, for the warm hospitality during my stay in Salt Lake City and all the discussions we had.
Many thanks go to Benjamin Br\"uck, Jonas Flechsig, Christoph Hilmes, Dawid Kielak, Thomas Lessmann, and Yuri Santos Rego for proofreading parts of my thesis.

The author was supported by the DFG grant WI 4079/4 within the SPP 2026 Geometry at infinity.

\section{Background on topology}\label{sec:backgroud-topology}

\subsection{Polytopes and cell complexes}

We quickly recall some basic aspects of the theory of polyhedral cell complexes.
The details of the constructions can be found in~\cite[I.7]{BridsonHaefliger99}.
For the rest of this section we fix a metric cell complex $X$ whose cells are isometric to \emph{Euclidean polysimplices}, i.e.\ products of simplices in some Euclidean space.
Further, we will assume that, up to isometry, there are only finitely many cells in $X$.
In order to deal with polysimplices, it is sometimes useful to note that they are simple polytopes.

\begin{definition}\label{def:simple-polytope}
A $d$-dimensional convex polytope $C$ is called \emph{simple} if every vertex in $C$ is contained in exactly $d$ \emph{facets}, i.e.\ faces of codimension $1$.
\end{definition}

From this characterization of simple polytopes it follows that products of simple polytopes are again simple.
In particular we see that polysimplices are simple.
In view of this, the characterization of simple polytopes given in~\cite[Proposition 2.16]{Ziegler95} tells us the following.

\begin{lemma}\label{lem:faces-in-polysimplices}
Let $A$ be a face of a polysimplex $C$ of codimension $k$.
Then there are precisely $k$ facets of $C$ that contain $A$.
In particular, there is a unique set of facets of $C$ that intersects in $A$.
\end{lemma}

\begin{notation}\label{not:open-cells}
Unless otherwise stated, the term \emph{cell} will always be used to denote the relative interior of its ambient closed polytope.
Nevertheless, we will say that a cell $A$ is a \emph{face} of a cell $B$ if $A$ is contained in $\overline{B}$.
In this case $B$ is said to be a \emph{coface} of $A$.
\end{notation}

One advantage of working with relatively open cells is that every point $x \in X$ is contained in a unique cell.

\begin{definition}\label{def:star}
Let $Y \subseteq X$ be a subcomplex, let $x \in Y$, and let $A \subseteq Y$ be the cell containing $x$.
The \emph{relative star} of $A$ in $Y$, denoted by $\st_Y(A)$, is the union of cofaces $B \subseteq Y$ of $A$.
The \emph{relative star} of $x$ in $Y$ is defined by $\st_Y(x) = \st_Y(A)$.
We define the \emph{relative link} of $x$ in $Y$, denoted by $\lk_Y(x)$, as the set of \emph{directions} at $x$, i.e.\ germs of geodesics emanating from $x$, that point into $Y$.
The \emph{relative link} of $A$ in $Y$ is the subspace $\lk_Y(A) \subseteq \lk_Y(\mathring{A})$ of directions that are orthogonal to $A$, where $\mathring{A}$ denotes the barycenter of $A$.
If $Y$ coincides with $X$, then we will often omit the subscript that indicates the ambient space and just speak about the stars/links of $A$ and $x$.
\end{definition}

\begin{notation}\label{not:boundary-simplex}
Let $A$ be a cell in $X$.
The \emph{boundary} $\partial A$ of $A$ is the complex of proper faces of $A$.
In particular, we have $\partial(v) = \emptyset$ for each vertex $v$ of $X$.
We will use the convention ${\dim(\emptyset) = -1}$.
\end{notation}

By assumption, each cell in $X$ isometrically embeds as a convex subspace of some Euclidean vector space.
This allows us to speak about angles between geodesic germs that emanate from the same point and to define a metric on the space of these germs.
Details of this construction can be found in~\cite[I.7.38]{BridsonHaefliger99}.
Using this topology, it is easy to see that the topological boundary of the relative star of a cell $A$ in $Y$ decomposes as $\partial(\st_Y(A)) \cong \partial A \ast \lk_Y(A)$, where $\ast$ denotes the join.
This decomposition will be especially useful when we combine it with the following result, which tells us that connectivity properties of cell complexes behave well under taking joins (see e.g.~\cite[Lemma 2.3]{Milnor56}).

\begin{lemma}\label{lem:connected-joins}
Let $Y$ and $Z$ be two cell complexes.
If $Y$ is $m$-connected and $Z$ is $n$-connected then their join $Y \ast Z$ is $(m+n+2)$-connected.
\end{lemma}

\begin{lemma}\label{lem:gluing-for-inclusion}
Let $n \in \N_0$ and let $Z$ be a cell complex that can be written as a union of subcomplexes $Z = X \cup \bigcup_{i \in I} Y_i$, where $I$ is an index set.
Assume that
\begin{enumerate}
\item each $Y_i$ is contractible,\label{item:contractible}
\item $Y_i \cap Y_j \subseteq X$, and that\label{item:intersection-in-X}
\item $Y_i \cap X$ is $(n-1)$-connected.
\label{item:intersection-nice}
\end{enumerate}
Then the pair $(Z,X)$ is $n$-connected. The same holds if ``$n$-connected'' is replaced by ``$n$-acyclic''.
\end{lemma}

\begin{proof}
By definition we have to show that for each $0 \leq k \leq n$ every map $(B^{k},S^{k-1}) \rightarrow (Z,X)$ is homotopic relative $S^{k-1}$ to a map whose image lies in $X$.
Thus for $k = 0$ it suffices to check that each point $p \in Z$ can be connected by a path to a point in $X$.
But this is clear since each $Y_i$ is path-connected by \eqref{item:contractible} and its intersection with $X$ is non-empty by~\eqref{item:intersection-in-X}. 
Note that this allows us to restrict ourselves to the case where $X$ and $Z$ are path-connected.
For $k=1$ the claim follows from the van Kampen theorem.
In view of Hurewicz's theorem it remains to show that the relative homology groups $\widetilde{H}_k(Z,X)$ vanish for $2 \leq k < n$.
Since taking colimits commutes with the homology functor (see e.g.~\cite[Theorem 14.6]{May99}) it follows from assumption~\eqref{item:intersection-in-X} that it is sufficient to consider the case where $I=\{i\}$ is a singleton.
We write $Y \defeq Y_i$ and consider the part
\[
0 = \widetilde{H}_k(Y) \to \widetilde{H}_{k}(Y,Y \cap X) \to \widetilde{H}_{k-1}(Y \cap X) \to \widetilde{H}_{k-1}(Y) = 0
\]
of the long exact sequence for the pair $(Y,Y \cap X)$. 
By~\eqref{item:intersection-in-X} we see that $\widetilde{H}_k(Y,Y\cap X) \cong \widetilde{H}_{k-1}(Y \cap X) = 0$ for $k < n$.
Since $Z = X \cup Y$ it remains to observe that excision gives us $\widetilde{H}_k(Z,X) \cong \widetilde{H}_k(Y,Y\cap X)$ for $k < n$.
\end{proof}

The following simple homological observation will help us to deduce negative connectivity properties of certain spaces.

\begin{lemma}\label{lem:unique-bounding-disc}
Let $X$ be a contractible cell complex of dimension $d$.
Let $z \in Z_{d-1}(X;\F_2)$ be a cycle of dimension $d-1$.
Then there is a unique $d$-chain $b \in C_{d}(X;\F_2)$ with $\partial b = z$.
\end{lemma}
\begin{proof}
First, observe that $Z_{d}(X;\F_2)=0$ since there are no cells of dimension $d+1$ and $\widetilde{H}_{d}(X;\F_2)=0$.
Suppose there are $d$-chains $B,B'$ such that $\partial B = \partial B' = Z$.
Then $\partial (B - B') = 0$ and therefore $B - B' \in Z_{d}(X;\F_2) = 0$.
\end{proof}

\begin{notation}\label{not:supported-subcomplex}
Let $X$ be a cell complex and let $M \subseteq X$ be an arbitrary subset.
The largest subcomplex of $X$ contained in $M$ will be denoted by $X(M)$.
We will say that $X(M)$ is the \emph{subcomplex of $X$ supported by~$M$}.
\end{notation}

\section{Background on $\Sigma$-invariants}\label{sec:background-sigma-invariants}

Unlike most invariants in group theory, the $\Sigma$-invariants of a group are not algebraic structures themselves.
Instead, they live on the so-called character sphere of a group, which consists of equivalence classes of non-trivial characters of the group.

\begin{notation}\label{not:homothety-classes-general}
Given a finite dimensional real vector space $V$, we write $S(V) \defeq (V \setminus \{0\}) / \hspace{-1.5mm} \sim$ to denote the space of positive homothety classes of non-trivial elements of $V$.
That is, for every $v,w \in V \setminus \{0\}$, we define $v \sim w$ if there is a real number $r > 0$ with $v = r w$.
\end{notation}

Note that $S(V)$, endowed with the natural topology, is a sphere of dimension $\dim(V)-1$.

\begin{definition}\label{def:charsphere}
Let $G$ be a finitely generated group.
A group homomorphism $\chi \colon G \rightarrow \R$ is called a \emph{character} of $G$.
The \emph{character sphere} of $G$ is defined as $S(G) \defeq S(\Hom(G,\R))$.
\end{definition}

\begin{definition}\label{def:essentially-connected}
Let $(X_{\lambda})_{\lambda \in \Lambda}$ be a directed system of cell complexes for a poset $(\Lambda,\leq)$ and let $X_{\alpha} \xrightarrow{f_{\alpha,\beta}} X_{\beta}$ be continuous maps for $\alpha \leq \beta$.
The system $(X_{\lambda})_{\lambda \in \Lambda}$ is \emph{essentially $n$-connected} for some $n \in \N_0$ if for every index $\alpha \in \Lambda$ there is an index $\beta \in \Lambda$ with $\alpha \leq \beta$ such that the induced maps
\[
\pi_k(f_{\alpha,\beta}) \colon \pi_k(X_{\alpha},x) \rightarrow \pi_k(X_{\beta},x)
\]
are trivial for every $x \in X_{\alpha}$ and every $0 \leq k \leq n$.
Analogously, we say that the system $(X_{\lambda})_{\lambda \in \Lambda}$ is \emph{essentially $n$-acyclic} for some $n \in \N_0$ if for every $\alpha \in \Lambda$ there is a $\beta \geq \alpha$ such that $\widetilde{H}_k(f_{\alpha,\beta})=0$ for every $0 \leq k \leq n$.
\end{definition}

The notion of essential connectivity appears naturally if one has to deal with group actions that are not cocompact.

\begin{definition}\label{def:type-fn}
A group $G$ is said to be of type $F_n$ if it acts freely on a contractible cell complex $X$ such that the quotient of the $n$-skeleton of $X$ by the group action is compact.
\end{definition}

In the following, we will often suppress the class of a character by just writing $\chi \in S(G)$.
This will not cause any problems since all properties of characters that we are going to look at are invariant under scaling with positive real numbers.
To define the $\Sigma$-invariants of a group $G$, we have to extend its characters equivariantly to an appropriate cell complex on which $G$ acts.

\begin{definition}\label{def:ext-char-and-superlvl}
Let $G$ be a group acting on a topological space $X$.
Let further $\chi$ be a character of $G$.
A continuous function $h \colon X \rightarrow \R$ is called a \emph{height function associated to $\chi$} if it is equivariant with respect to the action of $G$ on $\R$ via $\chi$, i.e.
\[
h(g(x)) = \chi(g) + h(x)
\]
for all $g \in G, x \in X$.
A \emph{superlevelset} in $X$ with respect to $h$ is a subset of the form $X_{h \geq r} \defeq h^{-1}([r,\infty))$ for $r \in \R$.
Analogously we define $X_{h \leq r}, X_{h=r}$ etc.
\end{definition}

If the action is free, then it is always possible to find height functions for characters (see~\cite[Konstruktion II.2.2]{Renz88}).

\begin{proposition}\label{prop:ex-of-height-fct}
Let $G$ be a group.
Suppose that $G$ acts freely on a contractible cell complex $X$ such that $G \backslash X$ has finite $n$-skeleton.
Then for every character $\chi \colon G \rightarrow \R$ there is a height function $h \colon X \rightarrow \R$ associated to $\chi$.
\end{proposition}

We are now ready to define $\Sigma$-invariants.

\begin{definition}\label{def:sigma-invariants}
Let $G$ be a group that acts freely on a contractible cell complex $X$ such that the quotient of the $n$-skeleton of $X$ by the group action is compact.
For every character $\chi$ of $G$ let $h_{\chi}$ be a height function associated to $\chi$.
The $n$th $\Sigma$-invariant of $G$, denoted by $\Sigma^n(G)$, is defined to be the subset of the character sphere that consists of characters $\chi$ such that the system $(X_{h_{\chi} \geq r})_{r \in \R}$ is essentially $(n-1)$-connected.
\end{definition}

Note that $\Sigma^n(G)$ is only defined for groups of type $F_n$.
It can be shown (see~\cite[Bemerkungen 3.5]{Renz88}) that the definition of $\Sigma$-invariants does not depend on the choices that have been made.
The next theorem, which is a special case of~\cite[Theorem 12.1]{BieriGeoghegan03}, tells us that the assumption of the freeness of the action can be considerably weakened.

\begin{theorem}\label{thm:sigmas-and-stabilizers}
Let $G$ be a group that acts on a contractible cell complex $X$ such that the quotient of the $n$-skeleton of $X$ by the group action is compact.
Suppose that the stabilizer of each $p$-cell is of type $F_{n-p}$ for $p \leq n-1$.
Let $\chi$ be a non-trivial character of $G$.
Suppose further that there is a height function $h \colon X \rightarrow \R$ associated to $\chi$.
Then $\chi$ lies in $\Sigma^{n}(G)$ if and only if the system $(X_{h \geq r})_{r \in \R}$ is essentially $(n-1)$-connected.
\end{theorem}

One of the most important properties of the $\Sigma$-invariants of a group $G$ is that they completely determine the finiteness properties of all subgroups $H \subseteq G$ that contain $[G,G]$.
The following result of Bieri and Renz~\cite[Satz C]{Renz88} makes this more precise.

\begin{theorem}\label{thm:relation-sigma-fn}
Let $G$ be a group of type $F_n$ and let $[G,G] \subseteq H \subseteq G$ be a subgroup.
Then $H \text{ is of type } F_n \Leftrightarrow \Set{\chi \in S(G)}{\chi(H)=0} \subseteq \Sigma^n(G)$.
\end{theorem}

\section{Background on buildings}\label{sec:coxeter-complex}

\subsection{Basic definitions}

The spaces we are going to look at will mainly be subcomplexes of spherical and Euclidean buildings.
Before defining them, let us recall that a spherical (resp. Euclidean) Coxeter group $W$ is generated by the set of reflections corresponding to a certain system $\mathcal{H}$ of linear (resp. affine) hyperplanes in a finite dimensional Euclidean vector space $V$.
We will think of the spherical (resp. Euclidean) Coxeter complex $\Sigma$ of $W$ as the standard sphere $S \subseteq V$ (resp. $V$) whose cell structure is given by intersecting $S$ (resp. $V$) with the hyperplanes in $\mathcal{H}$.
We will always assume that the action of a spherical Coxeter group $W$ on $\Sigma$ is essential, i.e.\ that $W$ does not fix a non-trivial subspace of $\Sigma$.
This ensures that the cell structure of $\Sigma$ is simplicial.
Similarly, we assume that a Euclidean Coxeter group acts essentially on the boundary at infinity $\pinf \Sigma$ (see Subsection~\ref{subsec:build-at-inf}) of $\Sigma$, which ensures that the cells in $\Sigma$ are polysimplicial.

\medskip

The following definition of a building is a slight variation of~\cite[Definition 4.1]{AbramenkoBrown08} in that we do not require the building to be simplicial.
This allows us to work with buildings whose apartments are products of irreducible Euclidean Coxeter complexes.

\begin{definition}\label{def:building}
A building is a cell complex $\Delta$ that can be expressed as the union of subcomplexes $\Sigma$ (called apartments) satisfying the following axioms:
\begin{enumerate}
\item[(B0)] Each apartment $\Sigma$ is a Coxeter complex.
\item[(B1)] For every two cells $A,B \subseteq \Delta$, there is an apartment $\Sigma$ containing both of them.\label{item:building-B1}
\item[(B2)] If $\Sigma_1$ and $\Sigma_2$ are two apartments containing two cells $A$ and $B$, then there is an isomorphism $\Sigma_1 \rightarrow \Sigma_2$ fixing $A$ and $B$ pointwise.
\end{enumerate}
The building $\Delta$ is called spherical (respectively Euclidean) if its apartments are spherical (respectively Euclidean) Coxeter complexes.
\end{definition}

Using (B1), we can define a $\CAT(1)$-metric on spherical building (see~\cite[Example 12.39]{AbramenkoBrown08}) and a $\CAT(0)$-metric on a Euclidean building (see~\cite[Theorem 11.16]{AbramenkoBrown08}) as follows.

\begin{theorem}\label{thm:buildings-are-CAT0}
Let $\Delta$ be a spherical (resp. Euclidean) building and let $d \colon \Delta \times \Delta \rightarrow \R$ be the function given by $(x,y) \mapsto d_{\Sigma}(x,y)$, where $\Sigma$ is any apartment containing $x$ and $y$, and $d_{\Sigma}$ is the metric induced by the spherical (resp. Euclidean) structure on $\Sigma$.
Then $d$ is a $\CAT(1)$-metric if $\Delta$ is spherical, and a $\CAT(0)$-metric if $\Delta$ is Euclidean.
\end{theorem}

\begin{definition}\label{def:chamber-panel}
Let $\Delta$ be a building.
A cell $A \subseteq \Delta$ of maximal dimension is called a \emph{chamber}.
A cell $A \subseteq \Delta$ of codimension 1 is called a \emph{panel}.
For every subcomplex $X \leq \Delta$, we write $\Ch(X)$ to denote the set of chambers in $X$.
\end{definition}

\begin{definition}
Let $\Delta$ be an arbitrary building.
A finite sequence of chambers $\Gamma = E_1, \ldots, E_n$ in $\Delta$ is called a \emph{gallery} from $E_1$ to $E_n$ if every two consecutive chambers $E_i,E_{i+1}$ share a common panel.
In this case we write $\Gamma = E_1 \vert \ldots \vert E_n$ and call $n-1$ the \emph{length} of $\Gamma$.
For every two chambers $C$, $D$ in $\Delta$, we define the \emph{gallery distance} between $C$ and $D$, denoted by $\dist_{\Delta}(C,D)$, to be the minimal length of a gallery from $C$ to $D$.
A gallery $\Gamma = C \vert \ldots \vert D$ of length $\dist_{\Delta}(C,D)$ is called minimal.
\end{definition}

Given a spherical or Euclidean buildings $\Delta$, there are at least to two useful notions of convexity.
On the one hand, we say that a subset of $X \subseteq \Delta$ is metrically convex if every geodesic connecting two points of $X$ stays in $X$.
On the other hand, there is a combinatorial notion of convexity given by the following definition, which applies to arbitrary buildings.

\begin{definition}
Let $\Delta$ be a building and let $X \subseteq \Delta$ be a subcomplex such that every maximal cell in $X$ is a chamber.
We say that $X$ is convex if for all $C,D \in \Ch(\Delta)$ every minimal gallery from $C$ to $D$ stays in $X$.
Given a subset $S \subseteq \Delta$, we define the (combinatorial) convex hull $\conv(S) \subseteq \Delta$ to be the smallest convex subcomplex containing $S$.
\end{definition}

An important feature of spherical buildings is that there is a notion of opposition.
We say that two points $x,y$ in a spherical building $\Delta$ are \emph{opposite} to each other if they are antipodal in some (hence every) apartment $\Sigma \subseteq \Delta$ containing them.
Similarly, we say that two cells $A,B \subseteq \Delta$ are opposite to each other if they are antipodal in some apartment containing them.

\begin{definition}\label{def:projections}
Let $\Delta$ be a Euclidean building and let $A,B \subseteq \Delta$ be two cells.
Given any two points $a \in A$, $b \in B$ we define the \emph{projection} of $A$ to $B$, denoted by $\pr_B(A)$, as the unique cell that contains an initial part of the open geodesic $(a,b)$ from $a$ to $b$.
The same definition applies if $\Delta$ is spherical, unless $A$ and $B$ are opposite, in which case there is no projection.
\end{definition}

\begin{lemma}\label{lem:existence-aps-sph-build}
Let $\Delta$ be a spherical building and let $\Sigma$ be an apartment of $\Delta$.
For each simplex $A \subseteq \Sigma$ and every opposite simplex $B$ of $A$ there is an apartment $\Sigma'$ containing $B$ and the star $\st_{\Sigma}(A)$ of $A$.
\end{lemma}
\begin{proof}
Let $A$ and $B$ be a pair of opposite simplices in $\Delta$.
Let ${C \subseteq \st_{\Sigma}(A)}$ be a chamber and let $D \subseteq \st_{\Sigma}(A)$ be the opposite chamber of $C$ in $\st_{\Sigma}(A)$.
Let $\Sigma'$ be an (in fact the unique) apartment containing $C$ and the projection chamber $\pr_B(D)$.
From~\cite[Proposition 4.69]{AbramenkoBrown08} it follows that $\pr_A(\pr_B(D)) = D$.
Since apartments are closed under taking projections we get $D \subseteq \Sigma'$.
Now the claim follows since the convex hull $\conv(C,D)$ coincides with $\st_{\Sigma}(A)$.
\end{proof}

In the following we will mainly be interested in thick buildings.
These are defined as follows.

\begin{definition}\label{def:thickness}
Let $\Delta$ be a building.
The \emph{thickness} of $\Delta$, denoted by $\thick(\Delta)$, is the maximal $t \in \N \cup \{\infty\}$ such that for every panel $P \subseteq \Delta$ there are at least $t$ chambers having $P$ as a face.
If $\thick(\Delta) \geq 3$, then we just say that $\Delta$ is \emph{thick}.
\end{definition}

\subsection{Buildings arising from BN-pairs}\label{subsec:from-BN-to-building}

The buildings we are interested in are associated to $\BN$-pairs.
Let us quickly recall what these are.

\begin{definition}\label{def:BN-pair}
Let $G$ be a group and let $B$, $N$ be subgroups of $G$ such that $G$ is generated by $B$ and $N$.
The pair $(B,N)$ is called a \emph{$\BN$-pair} of $G$ if $T \defeq B \cap N$ is normal in $N$, and the quotient $W \defeq N / T$ admits a set of generators $S$ such that the following conditions are satisfied.
\begin{itemize}
\item[$(\BN1)$] $sBw \subseteq BswB \cup BwB$ for every $s \in S$ and every $w \in W$.
\item[$(\BN2)$] $sBs^{-1} \nleq B$ for every $s \in S$.
\end{itemize}
In this case the tuple $(G,B,N,S)$ is called a \emph{Tits system}.
\end{definition}

As the notation suggest one can show that the pair $(W,S)$ coming from a Tits system $(G,B,N,S)$ is a Coxeter system (see~\cite[Theorem 6.56]{AbramenkoBrown08}).
Let us next describe how $(G,B,N,S)$ gives rise to a simplicial building $\Delta(B,N)$.
For each subset $J \subseteq S$ let $W_J$ denote the group generated by $J$.
From~\cite[Proposition 6.27]{AbramenkoBrown08}) we know that the union of double cosets $P_J = \bigcup \limits_{w \in W_J} B\widetilde{w}B$ is a group, where $\widetilde{w} \in N$ is any representative of $w$.
The groups $P_J$ are called \emph{standard parabolic subgroups}.
For every $s \in S$ we write $P_{\hat{s}} \defeq P_{S \backslash \{s\}}$.

\begin{definition}\label{def:from-BN-to-building}
Let $(B,N)$ be a $\BN$-pair of a group $G$ and let $(G,B,N,S)$ be a corresponding Tits system.
We define $\Delta(B,N)$ as the abstract simplicial complex whose simplices are given by the sets of the form $\Set{gP_{\hat{s}}}{s \in J}$, where $J \subseteq S$ and $g \in G$ are arbitrary.
Moreover we define $\Sigma(B,N)$ as the subcomplex of $\Delta(B,N)$ that consists of simplices $\Set{nP_{\hat{s}}}{s \in J}$ with $n \in N$.
\end{definition}

\begin{remark}\label{rem:independence-of-S}
It can be shown that every group $B \subseteq P \subseteq G$ is already of the form $P_J$ for some subset $J \subseteq S$ (see~\cite[Theorem 6.43]{AbramenkoBrown08}).
In particular, this tells us that $\Delta(B,N)$ and $\Sigma(B,N)$ do not depend on the choice of $S$.
\end{remark}

The next result summarizes some properties of $\Delta(B,N)$ (see e.g.~\cite[Exercise 6.54]{AbramenkoBrown08} and its solution starting on page 708).

\begin{Citation}\label{cit:from-BN-to-building}
Let $(B,N)$ be a $\BN$-pair of a group $G$.
Then $\Delta(B,N)$ is a building and $\Sigma(B,N)$ is an apartment in $\Delta(B,N)$.
\end{Citation}

\subsection{The spherical building at infinity}\label{subsec:build-at-inf} 

In the following it will be crucial for us to work with the spherical building at infinity of a Euclidean building $X$.
As a set, the spherical building at infinity of $X$ is given by the visual boundary $\pinf X$ of $X$.
In order to define the cell structure on $\pinf X$, we first have to introduce the notion of a special vertex.

\begin{definition}\label{def:special-vertex}
Let $\Sigma$ be a Euclidean Coxeter complex and let $\mathcal{H}$ be the corresponding set of hyperplanes.
A vertex $v \in \Sigma$ is called \emph{special} if for every hyperplane $H \in \mathcal{H}$ there is a parallel hyperplane $H' \in \mathcal{H}$ that contains $v$.
\end{definition}

The existence of a special vertex in a Euclidean Coxeter complex $\Sigma$ is guaranteed by~\cite[Proposition 10.17]{AbramenkoBrown08} together with~\cite[Proposition 10.19]{AbramenkoBrown08}.
In what follows, they will often serve as the origin of an apartment.

\begin{definition}\label{def:simpl-at-infinity-in-coxeter-complex}
Let $\Sigma$ be a Euclidean Coxeter complex and let $v \in \Sigma$ be a special vertex.
For every proper coface $A$ of $v$, let $S^{\Sigma}_v(A) \subseteq \pinf \Sigma$ be the subset of points $\xi$ such that the open ray $(v,\xi) \defeq [v,\xi) \backslash \{v\}$ has an initial segment in $A$.
Let $K^{\Sigma}_v(A) \subseteq \Sigma$ denote the union of all rays $(v,\xi)$ with $\xi \in S^{\Sigma}_v(A)$.
\end{definition}

Consider now a Euclidean building $X$ and let $S_{\infty}(X)$ be the set of subsets $S^{\Sigma}_v(A) \subseteq \pinf X$, where $\Sigma \subseteq X$ is an apartment, $v \in \Sigma$ is a special vertex, and $A$ is a proper coface of $v$.
Note that $S_{\infty}(X)$ can be endowed with a natual incidence structure by declaring $S^{\Sigma}_v(A)$ and $S^{\Sigma}_v(B)$ to be incident if $A$ and $B$ are incident.
Moreover, each of the spaces $S^{\Sigma}_v(A)$ has a natural metric given by the angle at $v$.
It is well-known that this structure defines a spherical building (see~\cite[Section 11]{AbramenkoBrown08}).

\begin{definition}\label{def:build-at-inft}
Let $X$ be a Euclidean building.
The \emph{spherical building at infinity} of $X$ is defined as the visual boundary $\pinf X$ endowed with the stucture of open cells given by $S_{\infty}(X)$.
\end{definition}

The following definition provides us with a more flexible version of the subsets $K^{\Sigma}_v(A)$ in that it allows $v$ to be an arbitrary point rather than a special vertex.

\begin{definition}\label{def:cones-induced-by-infty}
Let $X$ be a Euclidean building and let $\sigma \subseteq \pinf X$ be a simplex at infinity.
For each point $p \in X$ we define
\[
K_p(\sigma) = \bigcup \limits_{\xi \in \sigma} (p, \xi).
\]
If $\sigma$ is a chamber, we say that $K_p(\sigma)$ is a \emph{sector}.
\end{definition}

It will be often important for us to know whether two sectors correspond to the same chamber at infinity.
The following well-known result provides us with a useful criterion to decide this (see~\cite[Proposition 11.77]{AbramenkoBrown08}).

\begin{proposition}\label{prop:common-subsector}
Let $K_{p_1}(\sigma)$, $K_{p_2}(\tau)$ be two sectors in a Euclidean building $X$.
Then $\sigma = \tau$ if and only if $K_{p_1}(\sigma) \cap K_{p_2}(\tau)$ contains a sector.
\end{proposition}

\subsection{The opposition complex}

In the following, it will be crucial for us to understand the topological properties of the subcomplex of a spherical building that is opposite to a given chamber.

\begin{definition}\label{def:opposite-complex}
Let $\Delta$ be a spherical building and let $C$ be a chamber in $\Delta$.
The subcomplex of $\Delta$ that consists of all cells $A$ that are opposite to some face of $C$ will be denoted by $\Opp_{\Delta}(C)$.
\end{definition}

Recall that a $d$-dimensional cell complex $X$ is called \emph{spherical} if it is $(d-1)$-connected.
If a spherical building $\Delta$ is non-exceptional and thick enough, it is a well-known result of Abramenko~\cite[Theorem B]{Abramenko96} that $\Opp_{\Delta}(C)$ is spherical.

\begin{theorem}\label{thm:connectivity-of-opp}
Let $\Delta$ be an arbitrary building of type $A_{n+1}$, $C_{n+1}$ or $D_{n+1}$, but not an exceptional $C_3$ building.
Assume that $\thick(\Delta) \geq 2^{n}+1$ in the $A_{n+1}$ case, respectively $\thick(\Delta) \geq 2^{2n+1}+1$ in the other two cases.
Then $\Opp_{\Delta}(C)$ is spherical but not contractible for every chamber $C$ in $\Delta$.
\end{theorem}

\begin{remark}\label{rem:type-B-and-C}
Recall that the Coxeter groups that appear as the Weyl groups of the root systems of type $B_n$ and $C_n$ coincide.
In the theory of buildings it is a common convention to speak of buildings of type $C_n$ in this case.
\end{remark}

If the type of a spherical building $\Delta$ is arbitrary, then we can still show that $\Opp_{\Delta}(C)$ is not contractible, provided $\Delta$ is thick enough.
In fact, the following result endows us with a criterion to deduce that $\Opp_{\Delta}(C)$ contains an apartment, and hence has non-vanishing topdimensional homology.

\begin{proposition}\label{prop:existence-of-opposite-app}
Let $\Delta$ be a spherical building and let $\Sigma$ be an apartment in $\Delta$.
Suppose that the thickness of $\Delta$ is larger than the number of chambers in $\Sigma$.
Then there is a chamber $C$ in $\Delta$ such that $\Sigma$ is contained in $\Opp_{\Delta}(C)$.
If moreover $\Aut(\Delta)$ acts transitively on the chambers in $\Delta$, then each chamber $C$ in $\Delta$ contains an apartment in $\Opp_{\Delta}(C)$.
\end{proposition}
\begin{proof}
For each chamber $D \in \Ch(\Delta)$ let $N_D \defeq \sum \limits_{E \in \Ch(\Sigma)} \dist(D,E)$.
Let $C$ be a chamber such that $N_C$ is maximal.
Suppose that $\Sigma$ does not lie in $\Opp_{\Delta}(C)$.
Then there is a chamber $D \in \Ch(\Sigma)$ that is not opposite to $C$.
In this case~\cite[Lemma 4.106]{AbramenkoBrown08} tells us that there is a panel $P$ of $C$ such that $\pr_P(D) = C$.
Moreover we know from~\cite[Proposition 4.95]{AbramenkoBrown08} that each chamber $D' \in \Ch(\Delta)$ satisfies the so-called gate property
\[
\dist(D',C) = \dist(D',\pr_P(D')) + \dist(\pr_P(D'),C).
\]
From our assumption on the thickness of $\Delta$ we can therefore deduce that there is a chamber $C' \in \Ch(\st_{\Delta}(P)) \setminus \Set{\pr_P(D')}{D' \in \Ch(\Sigma)}$ that satisfies $\dist(C',D) > \dist(C,D)$ and $\dist(C',D') \geq \dist(C,D')$ for every $D' \in \Ch(\Sigma)$.
By this implies $N_{C'} > N_C$, which contradicts our choice of $C'$.

Suppose now that $\Aut(\Delta)$ acts transitively on $\Ch(\Delta)$.
Let $D$ be an arbitrary chamber in $\Delta$.
From the first claim we know that there is some chamber $C \in \Ch(\Delta)$ such that $\Opp_{\Delta}(C)$ contains an apartment $\Sigma$.
Let $\alpha \in \Aut(\Delta)$ be such that $\alpha(C) = D$.
Then $\alpha(\Sigma)$ is an apartment that is opposite to $D$, which gives us the second claim.
\end{proof}

Note that Proposition~\ref{prop:existence-of-opposite-app} implies in particular that for every spherical building $\Delta$ of infinite thickness and every chamber $C \subseteq \Delta$ the opposition complex $\Opp_{\Delta}(C)$ contains an apartment of $\Delta$.
This will be applied in Section~\ref{sec:neg-dir} on the boundary at infinity of a thick Euclidean building.

\section{Deconstructing subcomplexes of Coxeter complexes}\label{sec:deconstr}
Throughout this section we fix a $d$-dimensional Euclidean Coxeter complex $\Sigma$.
Let $\Sigma = \prod \limits_{i=1}^{s} \Sigma_i$ be the decomposition of $\Sigma$ into its irreducible factors.
We fix a special vertex $v \in \Sigma$. 
This allows us to view $\Sigma$ as a vector space with origin $v$.
Let $\sigma \subseteq \pinf \Sigma$ be a chamber at infinity and let $E \subseteq \st_{\Sigma}(v)$ be the corresponding chamber in $K_{v}(\sigma)$.
Let $\{P_1, \ldots, P_d\}$ be the set of panels of $E$ that contain $v$.
For each panel $P_i$ let $\alpha_i$ be a linear form on $\Sigma$ that satisfies $\alpha_i(P_i) = 0$ and $\alpha_i(E) > 0$.
For convenience, we choose $\alpha_i$ such that the parallel walls of $\alpha_i^{-1}(0)$ in $\Sigma$ are given by $W_{i,k} \defeq \alpha_i^{-1}(k)$, $k \in \Z$.
For each $i$ we consider the unique vertex $\xi_i$ of $\sigma$ for which the ray $[v, \xi_i)$ does not lie in $W_{i,0}$.
The Busemann function associated to $\xi_i$ and $v$ will be denoted by $\beta_i$.
It is an easy exercise to see that $\beta_i \colon \Sigma \rightarrow \R$ is the linear form characterized by $\beta_i([v,\xi_i)(t))=t$ and $\beta_i(\xi_i(1)^{\perp})=0$ (see~\cite[II.8.24.(1)]{BridsonHaefliger99}).

\medskip

In the next sections we will study height functions on Euclidean buildings that are induced by linear forms on a fixed apartment.
In this section we prepare this study by introducing auxiliary combinatorial properties of the superlevelsets in $\Sigma$ that come from certain linear forms on $\Sigma$.
To this end we fix a non-trivial linear form $h \colon \Sigma \rightarrow \R$ such that the composition $h \circ \xi \colon [0,\infty) \rightarrow \R$ is strictly decreasing for each $\xi \in \overline{\sigma}$.

\begin{remark}\label{rem:general_position}
Let $\eta = \pinf(h^{-1}((-\infty,r]))$, which does not depend on $r \in \R$.
The condition that $h \circ \xi$ is strictly decreasing for each $\xi \in \overline{\sigma}$ can also be expressed by saying that $\overline{\sigma} \subseteq \eta^\circ$ where $\eta^\circ$ denotes the interior of $\eta$.
Another equivalent condition is that $\eta \cap \overline{\sigma^{\op}} = \emptyset$, where $\sigma^{\op}$ denotes the opposite chamber of $\sigma$ in $\pinf \Sigma$.
\end{remark}

\begin{definition}\label{def:projection-to-inner-simplices}
Let $A$ be a cell in $\Sigma$ and let $\tau$ be a simplex in $\pinf \Sigma$.
The \emph{projection of $\tau$ to $A$}, denoted by $\pr_A(\tau)$, is the unique cell in $\st_{\Sigma}(A)$ such that for any (interior) point $\xi \in \tau$ and any point $x \in A$, there is an initial segment of $(x,\xi)$ that lies in $\pr_A(\tau)$.
\end{definition}

The following definition specifies the idea of moving towards a chamber at infinity.

\begin{definition}\label{def:moving-towards-infinity}
Let $\tau \subseteq \pinf \Sigma$ be a chamber.
A gallery $\Gamma = C_1 \vert \ldots \vert C_n$ in $\Sigma$ is \emph{$\tau$-minimal} if for every two consecutive chambers $C_i$, $C_{i+1}$ of $\Gamma$, separated by a panel $P$, the condition $\pr_P(\tau) = C_{i+1}$ is satisfied.
In this case we also say that $\Gamma$ is moving towards $\tau$.
\end{definition}

\begin{lemma}\label{lem:seeing-locally-where-sigma-is}
Let $A$ be a cell in $\Sigma$ of codimension at least $1$ and let $\tau$ be a chamber in $\pinf \Sigma$.
If $\Gamma = C_1 \lvert \ldots \rvert C_n$ is a minimal gallery in $\st_{\Sigma}(A)$ that terminates in $\pr_{A}(\tau)$, then $\Gamma$ is $\tau$-minimal.
\end{lemma}
\begin{proof}
Let $1 \leq i \leq n-1$ and let $P$ be the panel separating $C_i$ and $C_{i+1}$.
Let $W$ be the wall spanned by $P$ and let $R$ be the half space bounded by $W$ that contains $C_{i+1}$.
Since $\Gamma$ is minimal and terminating in $\pr_{A}(\tau)$ it follows that $\pr_{A}(\tau)$ lies in $R$.
Let $\xi \in \tau$ and $a \in A$ be arbitrary points.
By definition we have $[a,\xi)((0,\varepsilon)) \subseteq \pr_{A}(\tau)$ for some $\varepsilon > 0$.
In particular we see that $[a,\xi)((0,\varepsilon))$ is contained in $R$.
Hence for every point $x \in P$ the translate $[x,\xi)((0,\varepsilon))$ of $[a,\xi)((0,\varepsilon))$ is also contained in $R$.
It therefore follows that $\pr_P(\tau) = C_{i+1}$, which proves the claim.
\end{proof}

Lemma~\ref{lem:seeing-locally-where-sigma-is} gives us the following characterization of $\pr_{A}(\tau)$.

\begin{corollary}\label{cor:other-char-of-proj}
Let $A$ be a cell in $\Sigma$ and let $C$ be a chamber in $\st_{\Sigma}(A)$.
For each chamber $\tau$ in $\pinf \Sigma$ we have $\pr_A(\tau) = C$ if and only if $\pr_{P}(\tau) = C$ for every panel $A \subseteq P \subset C$.
\end{corollary}

\begin{lemma}\label{lem:upper-face}
Let $C \subseteq \Sigma$ be a chamber and let $I$ be the set of faces $A$ of $C$ that satisfy $\pr_A(\sigma) = C$.
Then $I$ contains a unique, minimal, proper, non-empty face of $C$.
In other words, there a non-empty face $U \subsetneq C$ such that $A \in I$ if and only if $U \fac A \fac C$.
\end{lemma}
\begin{proof}
Let $A,B \in I$ be two cells and let $\mathcal{P}_A$ and $\mathcal{P}_B$ be the sets of panels of $C$ that are cofaces of $A$ respectively $B$.
Corollary~\ref{cor:other-char-of-proj} tells us that $\pr_P(\sigma) = C$ for every $P \in \mathcal{P}_A \cup \mathcal{P}_B$.
On the other hand, we know from Lemma~\ref{lem:faces-in-polysimplices} that $A = \bigcap \limits_{P \in \mathcal{P}_A} P$ and $B = \bigcap \limits_{P \in \mathcal{P}_B} P$ and therefore $A \cap B = \bigcap \limits_{P \in \mathcal{P}_A \cup \mathcal{P}_B} P$.
Thus the uniqueness statement in Lemma~\ref{lem:faces-in-polysimplices} implies that every panel $A \cap B \subseteq P \subset C$ is contained in $\mathcal{P}_A \cup \mathcal{P}_B$ and therefore satisfies $\pr_P(\sigma) = C$.
In view of Corollary~\ref{cor:other-char-of-proj}, it remains to show that $A \cap B$ is not empty.
To see this, let $C = \prod \limits_{i=1}^s C_i$ be the decomposition of $C$ into simplices $C_i \subseteq \Sigma_i$ and let $1 \leq j \leq s$ be a fixed coordinate.
We claim that there is a panel $P$ of $C_j$ such that the corresponding panel
\[
C_1 \times \ldots C_{j-1} \times P \times C_{j+1} \ldots \times C_s
\]
of $C$ is not contained in $\mathcal{P}_A \cup \mathcal{P}_B$.
Indeed, otherwise $[x,\xi)$ would stay in
\[
\Sigma_1 \times \ldots \Sigma_{j-1} \times C_j \times \Sigma_{j+1} \ldots \times \Sigma_s
\]
for every $x \in C$ and every point $\xi$ in the (open) chamber $\sigma$.
But then $[x,\xi)$ is constant in the coordinate $j$, which is only possible if $\xi \in \overline{\sigma}$ is not an interior point of $\sigma$.
On the other hand, the intersection over a set of closed panels of a chamber is empty if and only if the set consists of all panels of the chamber.
Together this shows that $A \cap B$ is non-empty, which proves the claim.
\end{proof}

\begin{definition}\label{def:upper-face}
For each chamber $C$ in $\Sigma$ we define its \emph{upper face}, denoted by $C^{\uparrow}$, as the intersection of all panels $P \subset C$ with $\pr_P(\sigma) = C$ (the face $U$ in Lemma~\ref{lem:upper-face}).
Similarly, we define the \emph{lower face} of $C$, denoted by $C^{\downarrow}$, as the intersection of all panels $P \subset C$ with $\pr_P(\sigma^{\op}) = C$, where $\sigma^{\op} \subseteq \pinf \Sigma$ denotes the chamber opposite to $\sigma$.
\end{definition}

Let us recall the so-called gate property for Coxeter complexes~\cite[Proposition 3.105]{AbramenkoBrown08}.

\begin{proposition}\label{prop:gate-groperty}
Let $C$ be a chamber in $\Sigma$ and let $A \subset \Sigma$ be an arbitrary cell.
Then every chamber $D \subseteq \st_{\Sigma}(A)$ satisfies the equality
\[
d(D,C) = d(D,\pr_{A}(C)) + d(\pr_{A}(C),C).
\]
\end{proposition}

As a consequence of Proposition~\ref{prop:gate-groperty}, we see that there is a minimal gallery from $D$ to $C$ passing through $\pr_{A}(C)$.
Recall from Subsection~\ref{subsec:build-at-inf} that for each point $x \in \Sigma$ and each simplex $\tau \subseteq \pinf \Sigma$ we write
\[
K_x(\tau) = \bigcup \limits_{\xi \in \tau} (x, \xi)
\]
to denote the (relatively open) cone corresponding to $\tau$ with tip in $x$.

\begin{remark}\label{rem:description-of-sectors}
By our choice of $\alpha_1,\ldots,\alpha_d$, every sector $K_x(\sigma)$ can be described as the set of points $y \in \Sigma$ such that $\alpha_i(y) > \alpha_i(x)$ for every $1 \leq i \leq d$.
Similarly, $K_x(\sigma^{\op})$ can be described as the set of points $y \in \Sigma$ with $\alpha_i(y)<\alpha_i(x)$ for every $1 \leq i \leq d$.
\end{remark}

\begin{lemma}\label{lem:truncated-sectors}
For every point $x \in \Sigma$ and every $r \in \R$ the intersection
\[
\overline{K_x(\sigma^{\op})} \cap h^{-1}((-\infty,r])
\]
is compact.
\end{lemma}
\begin{proof}
The polyhedron $\overline{K_x(\sigma^{\op})} \cap h^{-1}((-\infty,r])$ has boundary
\begin{align*}
\pinf(\overline{K_x(\sigma^{\op})} \cap h^{-1}((- \infty,r])) &= \pinf(\overline{K_x(\sigma^{\op})}) \cap \pinf(h^{-1}((- \infty,r]))\\ &= \overline{\sigma^{\op}} \cap \eta = \emptyset\text{,}
\end{align*}

see Remark~\ref{rem:general_position}.
It is therefore compact and the claim follows.
\end{proof}

Note that Lemma~\ref{lem:truncated-sectors} implies in particular that the subcomplex
\[
\Sigma(\overline{K_x(\sigma^{\op})} \cap h^{-1}((-\infty,r]))
\]
of $\Sigma$ that is supported on $\overline{K_x(\sigma^{\op})} \cap h^{-1}((-\infty,r])$ is also compact.

\begin{lemma}\label{lem:sector-covering-compact}
For every bounded subset $Z$ of $\Sigma$ there is a special vertex $w \in \Sigma$ such that $Z$ is contained in $K_w(\sigma^{\op})$.
\end{lemma}
\begin{proof}
Since $Z$ is bounded there is an integer $n$ such that $\alpha_i(z) < n$ for every $z \in Z$ and every $1 \leq i \leq d$.
In view of Remark~\ref{rem:description-of-sectors} it therefore suffices to define $w \in \Sigma$ by $\alpha_i(w) = n$ for all $1 \leq i \leq d$.
\end{proof}

The following lemma provides us with a uniform lower bound for the special vertex in Lemma~\ref{lem:sector-covering-compact} in the case where $Z$ consists of a single point.

\begin{lemma}\label{lem:sector-covering-simplex}
There is a constant $\varepsilon > 0$ such that for every point $x \in \Sigma$ there is a special vertex $w \in \Sigma$ of height $h(w) > h(x) - \varepsilon$ such that $K_w(\sigma^{\op})$ contains $x$.
\end{lemma}
\begin{proof}
Let $C \subseteq \Sigma$ be a chamber with $x \in \overline{C}$ and let $u_1$ be a special vertex of $C$.
We consider the points $z_i =\alpha_i(u_1)$ for every ${1 \leq i \leq d}$.
Let $u_2$ be the special vertex characterized by $\alpha_i(u_2) = z_i + 1$ for every $1 \leq i \leq d$.
Then the subcomplex $\overline{K_{u_1}(\sigma^{\op})} \subseteq \Sigma$ lies in the (open) sector $K_{u_2}(\sigma^{\op})$.
It follows that $\st_{\Sigma}(u_1)$ is contained in $K_{u_2}(\sigma^{\op})$.
In particular we see that $x \in \overline{\st(u_1)} \subseteq \overline{K_{u_2}(\sigma^{\op})}$.
Let us now consider the special vertex $w \in \Sigma$ that is characterized by $\alpha_i(w) = z_i + 2$ for every $1 \leq i \leq d$.
Then $\overline{K_{u_2}(\sigma^{\op})}$ is contained in $K_{w}(\sigma^{\op})$, which proves the second part of the lemma.
To see the first part, let $\delta_1$ be the $h$-distance between $u_1$ and $u_2$ and let $\delta_2$ be the $h$-diameter of the star of a special vertex.
Then it follows from our construction that $w$ satisfies $h(w) \geq h(x) - \varepsilon$ for $\varepsilon = 2\delta_1 + 2\delta_2$.
\end{proof}

\begin{definition}\label{def:sigma-convexity}
A subcomplex $Z \subseteq \Sigma$ is called \emph{$\sigma$-convex} if for every two cells $A, B \subseteq Z$ the following is satisfied.
Every $\sigma$-minimal gallery $\Gamma$ from $\pr_A(\sigma)$ to $\pr_B(\sigma^{\op})$ is contained in $Z$.
\end{definition}

\begin{remark}\label{rem:sigm-con-is-sigm-op-con}
We emphasize that Definition~\ref{def:sigma-convexity} does not require the existence of a $\sigma$-minimal gallery in $Z$.
\end{remark}

\begin{definition}\label{def:non-separating-boundary}
Given a subcomplex $Z$ of $\Sigma$, we write $R(Z)$ to denote the union of cells $A \subseteq Z$ that satisfy $\pr_A(\sigma^{\op}) \nsubseteq Z$.
\end{definition}

\begin{lemma}\label{lem:R-is-a-subcomplex}
Let $Z \subseteq \Sigma$ be a subcomplex.
If $Z$ is $\sigma$-convex, then $R(Z)$ is a subcomplex of $Z$.
\end{lemma}
\begin{proof}
Let $B$ be a cell in $R(Z)$ and let $A$ be a face of $B$.
Let $\Gamma$ be a minimal gallery from $\pr_A(\sigma^{\op}) \eqdef C$ to $\pr_B(\sigma^{\op}) \eqdef D$.
Then $\Gamma$ is contained in $\st_{\Sigma}(A)$ and it follows from~\cite[Lemma 4.96]{AbramenkoBrown08} that $\Gamma$ can be extended to a minimal gallery $\Gamma'$ from $C$ to $\pr_A(\sigma)$.
In this case Lemma~\ref{lem:seeing-locally-where-sigma-is} implies that $\Gamma$ is $\sigma$-minimal.
In order to apply the $\sigma$-convexity of $Z$, we note that $C = \pr_{C^{\uparrow}}(\sigma)$ and $D = \pr_{D^{\downarrow}}(\sigma^{\op})$.
From Lemma~\ref{lem:upper-face} it follows that $B$ is a coface of $D^{\downarrow}$.
In particular, we see that $D^{\downarrow}$ lies in $Z$.
Suppose that $A$ is not a cell of $R(Z)$.
Then $C^{\uparrow} \subseteq C \subseteq Z$ and we may apply the $\sigma$-convexity of $Z$ to deduce that the entire gallery $\Gamma$ is contained in $Z$.
In particular, $D$ is a chamber in $Z$, which contradicts our assumption that $B$ is a cell in $R(Z)$.
Thus $A$ is a cell of $R(Z)$, which proves the claim.
\end{proof}

\begin{definition}\label{def:length-of-chambers}
Let $Z$ be a subcomplex of $\Sigma$.
For each chamber $C$ in $Z$, we define its \emph{$\sigma$-length} in $Z$, denoted by $\ell_Z(C)$, as the length of the longest $\sigma$-minimal gallery in $Z$ that starts in $C$.
If there are arbitrarily long $\sigma$-minimal galleries in $Z$ that in $C$, then we define $\ell_Z(C) = \infty$.
\end{definition}

\begin{lemma}\label{lem:removing-chambers-of-zero-length}
Let $Z$ be a $\sigma$-convex subcomplex of $\Sigma$ and let $C \subseteq Z$ be a chamber with $\ell_{Z}(C)=0$.
Then the following are satisfied.
\begin{enumerate}
\item $\st_{Z}(C^{\downarrow}) \subseteq \overline{C}$,
\item $Z \backslash \st_{Z}(C^{\downarrow})$ is $\sigma$-convex,
\item $R(Z \backslash \st_{Z}(C^{\downarrow})) = R(Z)$.
\end{enumerate}
\end{lemma}
\begin{proof}
To prove the first claim, let $A$ be a cell in $\st_{Z}(C^{\downarrow})$ and let $D \defeq \pr_{A}(\sigma^{\op})$.
Let further $\Gamma$ be a minimal gallery from $D$ to $C$.
Since $A$ is a coface of $C^{\downarrow}$, it follows that $\Gamma$ is contained in $\st_{\Sigma}(C^{\downarrow})$.
On the other hand we have $C = \pr_{C^{\downarrow}}(\sigma^{\op})$ so that we can apply Lemma~\ref{lem:seeing-locally-where-sigma-is} to deduce that $\Gamma$ is $\sigma^{\op}$-minimal.
Since also $C = \pr_{C^{\uparrow}}(\sigma)$, it follows from the $\sigma$-convexity of $Z$ that $\Gamma$ is contained in $Z$.
Now the condition $\ell_Z(C)=0$ implies $C = D$ and therefore $A \subseteq \overline{C}$.
For the second claim, let $A$ and $B$ be two cells in $Z \backslash \st_{Z}(C^{\downarrow})$
such that there is a $\sigma$-minimal gallery $\Gamma$ from $\pr_{A}(\sigma)$ to $\pr_{B}(\sigma^{\op})$.
We have to show that $\Gamma$ lies in $Z \backslash \st_{Z}(C^{\downarrow})$.
By the first claim it thus suffices to prove that $\Gamma$ does not contain $C$.
Suppose that $\Gamma$ contains $C$ and let $\Gamma'$ be the subgallery of $\Gamma$ starting at $C$.
If $C \neq \pr_{B}(\sigma^{\op})$, then the $\sigma$-convexity of $Z$ implies that $\Gamma'$ is contained in $Z$ and therefore $\ell(C) > \ell(\pr_{B}(\sigma^{\op}))$, which is a contradiction to $\ell(C) = 0$.
On the other hand, if $C = \pr_{B}(\sigma^{\op})$,
then $B$ is a coface of $C^{\downarrow}$ by Lemma~\ref{lem:upper-face}.
But this is a contradiction since there are no cofaces of $C^{\downarrow}$ lying in $Z \backslash \st_{Z}(C^{\downarrow})$.
To prove the third claim, let $A$ be a cell in $R(Z)$.
By definition we have $\pr_{A}(\sigma^{\op}) \nsubseteq Z$, which gives us $\pr_{A}(\sigma^{\op}) \nsubseteq Z \backslash \st_{Z}(C^{\downarrow})$.
To prove that $A$ is contained in $R(Z \backslash \st_{Z}(C^{\downarrow}))$, it suffices to show that $A \nsubseteq \st_{Z}(C^{\downarrow})$.
Suppose the opposite.
Then the first claim tells us that $A$ is a coface of $C^{\downarrow}$ lying in $\overline{C}$ so that we can apply Lemma~\ref{lem:upper-face} to deduce that
\[
\pr_{A}(\sigma^{\op}) = \pr_{C^{\downarrow}}(\sigma^{\op}) = C \subseteq Z.
\]
But this contradicts our assumption that $\pr_{A}(\sigma^{\op}) \nsubseteq Z$.
Suppose now that $A$ is a cell in $R(Z \backslash \st_{Z}(C^{\downarrow}))$.
Then $\pr_{A}(\sigma^{\op}) \nsubseteq R(Z \backslash \st_{Z}(C^{\downarrow}))$ and the first claim implies that either $\pr_{A}(\sigma^{\op}) = C$ or $\pr_{A}(\sigma^{\op}) \nsubseteq Z$.
Since we are done otherwise, we can assume that $\pr_{A}(\sigma^{\op}) = C$.
But in this case another application of Lemma~\ref{lem:upper-face} tells us that $A$ is a coface of $C^{\downarrow}$, which cannot lie in $Z \backslash \st_{Z}(C^{\downarrow})$.
\end{proof}

Note that Lemma~\ref{lem:removing-chambers-of-zero-length} can be applied inductively by starting with a $\sigma$-convex complex $Z$ and removing a chamber of length $0$ at each step.
We thus obtain the following way of filtering $\sigma$-convex complexes.

\begin{corollary}\label{cor:deconstr-sigma-conv-complexes}
Let $Z$ be a compact $\sigma$-convex subcomplex of $\Sigma$ and let $n$
be the number of chambers in $Z$.
There is a filtration
\[
Z_0 \lneq Z_1 \lneq \ldots \lneq Z_n = Z
\]
of $Z$ by subcomplexes $Z_i$ such that
\begin{enumerate}
\item $Z_0 = R(Z)$,
\item $Z_{m+1} = Z_{m} \cup \overline{C_{m+1}}$ for some chamber
$C_{m+1} \subseteq Z$ with \linebreak${\ell_{Z_{m+1}}(C_{m+1}) = 0}$,
\item $Z_{m} \cap \overline{C_{m+1}} = Z_{m} \cap \partial(\st_{Z_{m+1}}(C_{m+1}^{\downarrow}))$, and
\item $\st_{Z_{m+1}}(C_{m+1}^{\downarrow}) \subseteq \overline{C_{m+1}}$.
\end{enumerate}
\end{corollary}
\begin{proof}
We prove the corollary by induction on $n$.
If $n=0$, then there is nothing to show.
Suppose that the claim is true if $Z$ has at most $n$ chambers.
Consider now a compact, $\sigma$-convex subcomplex $Z \subseteq \Sigma$ with $n+1$ chambers.
Since the length of a minimal gallery in $Z$ is bounded, it follows that there is a chamber $C_{n+1} \subseteq Z$ with $\ell_{Z}(C_{n+1})=0$.
Let $Z_{n+1} = Z$ and let $Z_n = Z_{n+1} \backslash \st_{Z_{n+1}}(C_{n+1}^{\downarrow})$.
From Lemma~\ref{lem:removing-chambers-of-zero-length} we know that $Z_n$ is $\sigma$-convex, that $R(Z_n) = R(Z_{n+1})$, and that $\st_{Z_{n+1}}(C_{n+1}^{\downarrow}) \subseteq \overline{C_{n+1}}$.
From the latter property we obtain $Z_{n+1} = Z_n \cup \overline{C_{n+1}}$ and $Z_{n} \cap \overline{C_{n+1}} = Z_{n} \cap \partial(\st_{Z_{n+1}}(C_{n+1}^{\downarrow}))$.
By our induction hypothesis there is a sequence
\begin{equation}\label{eq:deconstr-sigma-conv-complexes}
Z_0 \lneq Z_1 \lneq \ldots \lneq Z_n
\end{equation}
of subcomplexes as in the corollary.
Now the claim follows by extending~\eqref{eq:deconstr-sigma-conv-complexes} with $Z_{n+1}$.
\end{proof}

For short reference we note the following easy property of sectors.

\begin{lemma}\label{lem:proj-in-sectors}
Let $w \in \Sigma$ be a special vertex and let $\tau$ be a chamber in $\pinf \Sigma$.
For every cell $A$ in $\overline{K_w(\tau)}$, the projection chamber $\pr_A(\tau)$ lies in $\overline{K_w(\tau)}$.
\end{lemma}
\begin{proof}
This follows directly from the fact that for every  $\xi \in \tau$ and every $x \in \overline{K_w(\tau)}$ the ray $[x,\xi)$ stays in $\overline{K_w(\tau)}$.
\end{proof}

The following corollary follows by applying Lemma~\ref{lem:proj-in-sectors} on a panel $P$ that lies in a wall of a sector $K_w(\tau)$.

\begin{corollary}\label{cor:sigma-min-gal-stay-in-cones}
Let $w \in \Sigma$ be a special vertex and let $\tau$ be a chamber in $\pinf \Sigma$.
Let $\Gamma = E_1 \vert \ldots \vert E_n$ be a $\tau$-minimal gallery in $\Sigma$.
If $E_1$ lies in $K_w(\tau)$, then the whole gallery $\Gamma$ lies in $K_w(\tau)$.
\end{corollary}

\begin{lemma}\label{lem:proj-in-sector-complements}
Let $w \in \Sigma$ be a special vertex and $A$ be a cell in $\Sigma \backslash K_w(\sigma)$.
Then $\pr_A(\sigma^{\op})$ lies in $\Sigma \backslash K_w(\sigma)$.
\end{lemma}
\begin{proof}
Suppose that $\pr_A(\sigma^{\op}) \subseteq K_w(\sigma)$.
Then $A$ is a cell in $\overline{K_w(\sigma)}$ and we obtain
$\pr_A(\sigma) \subseteq \overline{K_w(\sigma)}$ from Lemma~\ref{lem:proj-in-sectors}.
From the convexity of $K_w(\sigma)$ we deduce that
\[
\conv(\pr_A(\sigma),\pr_A(\sigma^{\op})) = \st_{\Sigma}(A)
\]
lies in $K_w(\sigma)$.
Since $\st_{\Sigma}(A)$ is an open neighborhood of $A$ it follows that $A$ lies in the (open) sector $K_w(\sigma)$.
But this contradicts the choice of $A$.
\end{proof}

\begin{proposition}\label{prop:cones-are-sigma-convex}
Let $w \in \Sigma$ be a special vertex.
The closed sectors $\overline{K_w(\sigma)}$, $\overline{K_w(\sigma^{\op})}$ and the complements of the open sectors $\Sigma \backslash K_w(\sigma)$ and $\Sigma \backslash K_w(\sigma^{\op})$ are $\sigma$-convex.
\end{proposition}
\begin{proof}
Since $\overline{K_w(\sigma)}$ and $\overline{K_w(\sigma^{\op})}$ are convex subcomplexes it follows that they are also $\sigma$-convex.
Next we consider the complements.
We start with $\Sigma \backslash K_w(\sigma)$.
Let $A$ and $B$ be cells in $\Sigma \backslash K_w(\sigma)$ and suppose that there is a $\sigma$-minimal gallery $\Gamma = E_1 \lvert \ldots \rvert E_n$ from $E_1 = \pr_A(\sigma)$ to $E_n = \pr_B(\sigma^{\op})$.
By Lemma~\ref{lem:proj-in-sector-complements} the chamber $\pr_B(\sigma^{\op})$
is contained in $\Sigma \backslash K_w(\sigma)$.
Suppose that $\Gamma$ contains a chamber $E_{i_0}$ in $K_w(\sigma)$.
Then the subgallery $\Gamma' \defeq E_{i_0} \lvert \ldots \rvert E_n$ is $\sigma$-minimal
and hence by Lemma~\ref{cor:sigma-min-gal-stay-in-cones} stays in $K_w(\sigma)$.
A contradiction to $E_n = \pr_B(\sigma^{\op}) \subseteq \Sigma \backslash K_w(\sigma)$.

Suppose now that $A$ and $B$ are cells in $\Sigma \backslash K_w(\sigma^{\op})$ and that there is a $\sigma$-minimal gallery $\Gamma = E_1 \lvert \ldots \rvert E_n$ from $E_1 = \pr_A(\sigma)$ to $E_n = \pr_B(\sigma^{\op})$.
By Lemma~\ref{lem:proj-in-sector-complements} the chamber $\pr_A(\sigma)$ is contained in $\Sigma \backslash K_w(\sigma^{\op})$.
Note that the reversed gallery $\Gamma^{\op} = E_n \lvert \ldots \rvert E_1$ is moving towards $\sigma^{\op}$.
Hence if $\Gamma^{\op}$ contains a chamber $E_{i_0} \subseteq K_w(\sigma^{\op})$, then the subgallery $E_{i_0} \lvert \ldots \rvert E_1$ is contained in $K_w(\sigma^{\op})$ by Lemma~\ref{cor:sigma-min-gal-stay-in-cones}.
But this is a contradiction to $E_1 = \pr_A(\sigma) \subseteq \Sigma \backslash K_w(\sigma^{\op})$.
\end{proof}


\begin{lemma}\label{lem:intersec-nonsep-boundary}
Let $Y$ and $Z$ be two subcomplexes of $\Sigma$.
Then
\[
R(Y \cap Z) = Y \cap Z \cap (R(Y) \cup R(Z)).
\]
\end{lemma}
\begin{proof}
Let $A$ be a cell in $\Sigma$.
From the definitions it directly follows that
\begin{align*}
A \subseteq R(Y \cap Z)
&\Leftrightarrow A \subseteq Y \cap Z \text{ and } \pr_A(\sigma^{\op}) \nsubseteq Y \cap Z\\
&\Leftrightarrow A \subseteq Y \cap Z \text{ and } (\pr_A(\sigma^{\op}) \nsubseteq Y \text{ or } \pr_A(\sigma^{\op}) \nsubseteq Z)\\
&\Leftrightarrow A \subseteq Y \cap Z \cap (R(Y) \cup R(Z)),
\end{align*}
which proves the claim.
\end{proof}

\begin{definition}\label{def:lower-complex}
For each $r \in \R$ let $M(r)$ denote the set of special vertices $w \in \Sigma$ of height $h(w) \geq r$.
We define \emph{upper complex} associated to $h$ and $r$ by
\[
U_h(r) = \bigcup \limits_{w \in M(r)} \overline{K_w(\sigma^{\op})}.
\]
Moreover we define the \emph{lower complex} associated to $h$ and $r$, denoted by $L_h(r)$, as the complement of the interior of $U_h(r)$ in $\Sigma$,\ i.e.
\[
L_h(r) = \Sigma \backslash \bigcup \limits_{w \in M(r)} K_w(\sigma^{\op}).
\]
\end{definition}

Let us summarize some properties of $L_h(r)$.

\begin{proposition}\label{prop:intersection-cones}
There is a constant $\varepsilon > 0$ such that for every $r \in \R$
\begin{enumerate}
\item $L_h(r)$ is $\sigma$-convex,
\item $h^{-1}((-\infty,r]) \subseteq L_h(r)$,
\item $L_h(r) \subseteq h^{-1}((-\infty,r+\varepsilon])$, and
\item $R(L_h(r)) \subseteq h^{-1}([r,r+\varepsilon])$.
\end{enumerate}
\end{proposition}
\begin{proof}
Note that we can write the lower complex as an intersection
\[
L_h(r) = \Sigma \backslash \bigcup \limits_{w \in M(r)} K_w(\sigma^{\op}) = \bigcap \limits_{w \in M(r)} \Sigma \backslash K_w(\sigma^{\op})
\]
of sector complements.
Now the first claim follows by Proposition~\ref{prop:cones-are-sigma-convex} and the obvious observation that the intersection of $\sigma$-convex subcomplexes is $\sigma$-convex.
Recall from our choice of $h$ that $w$ is the lowest point in $K_w(\sigma^{\op})$.
As a consequence, it follows that $K_w(\sigma^{\op}) \subseteq h^{-1}((r,\infty))$ for every vertex $w \in M(r)$.
Thus the sublevelset $h^{-1}((-\infty,r])$ is contained in $\Sigma \backslash K_w(\sigma^{\op})$ for every vertex $w \in M(r)$, which shows the second claim.
To prove the third claim, let $\varepsilon > 0$ be the constant from Lemma~\ref{lem:sector-covering-simplex} and let $x \in L_h(r)$ be an arbitrary point.
Then Lemma~\ref{lem:sector-covering-simplex} provides us with a special vertex $w \in \Sigma$ of height $h(w) > h(x)-\varepsilon$ such that $x \in K_w(\sigma^{\op})$.
Thus $h(w) < r$ and we deduce that $h(x) < r+\varepsilon$.
To verify the last claim, let $A \subseteq L_h(r)$ be a cell with $\pr_A(\sigma^{\op}) \nsubseteq L_h(r)$.
Then there is a vertex $w \in M(r)$ such that $\pr_A(\sigma^{\op}) \subseteq K_w(\sigma^{\op})$.
This gives us a chain of inclusions
\[
A \subseteq \overline{K_w(\sigma^{\op})} \subseteq h^{-1}([h(w),\infty)) \subseteq h^{-1}([r,\infty)).
\]
On the other hand, the third claim gives us $A \subseteq L_h(r) \subseteq h^{-1}((-\infty,r+\varepsilon])$, which completes the proof.
\end{proof}

\begin{lemma}\label{lem:intersection-cones}
There is a constant $\varepsilon > 0$ such that for every $r \in \R$ and every special vertex $w \in \Sigma$, the intersection $\overline{K_w(\sigma^{\op})} \cap L_h(r)$ is $\sigma$-convex and $R(\overline{K_w(\sigma^{\op})} \cap L_h(r))$ lies in $h^{-1}([r,r+\varepsilon])$.
\end{lemma}
\begin{proof}
Let $\varepsilon$ be as in Proposition~\ref{prop:intersection-cones}.
The complexes $\overline{K_w(\sigma^{\op})}$ and $L_h(r)$ are $\sigma$-convex by  Proposition~\ref{prop:cones-are-sigma-convex} and Proposition~\ref{prop:intersection-cones}.
Thus their intersection is $\sigma$-convex as well.
From Lemma~\ref{lem:intersec-nonsep-boundary} we know that
\[
R(\overline{K_w(\sigma^{\op})} \cap L_h(r))
= \overline{K_w(\sigma^{\op})} \cap L_h(r) \cap \Big( R(\overline{K_w(\sigma^{\op})}) \cup R(L_h(r))\Big).
\]
On the other hand, it follows from Lemma~\ref{lem:proj-in-sectors} that $R(\overline{K_w(\sigma^{\op})}) = \emptyset$.
Together with claim $(4)$ of Proposition~\ref{prop:intersection-cones} this implies
\[
R(\overline{K_w(\sigma^{\op})} \cap L_h(r)) \subseteq R(L_h(r)) \subseteq h^{-1}([r,r+\varepsilon]),
\]
which proves the lemma.
\end{proof}

\begin{lemma}\label{lem:panels:are-nonisolating}
Let $r \in \R$, let $w \in \Sigma$ be a special vertex, and let
\[
P \subseteq R\Big(\overline{K_w(\sigma^{\op})} \cap L_h(r)\Big) \cap K_w(\sigma^{\op})
\]
be a panel.
Then $\pr_{P}(\sigma)$ lies in $K_w(\sigma^{\op}) \cap L_h(r)$.
\end{lemma}
\begin{proof}
Since $P$ lies in the (open) sector $K_w(\sigma^{\op})$, it follows that $\pr_{P}(\sigma)$ is contained in $K_w(\sigma^{\op})$.
Suppose that $\pr_{P}(\sigma)$ is not contained in $L_h(r)$.
Then there is a vertex $u \in M(r)$ with $\pr_{P}(\sigma) \subseteq K_u(\sigma^{\op})$.
On the other hand, Lemma~\ref{lem:proj-in-sectors} tells us that $\pr_{P}(\sigma^{\op})$ lies in $K_u(\sigma^{\op})$.
Together this gives us $P \subseteq \st_{\Sigma}(P) \subseteq K_u(\sigma^{\op})$.
But this contradicts our assumption that $P$ lies $L_h(r)$, which proves the lemma.
\end{proof}

\begin{proposition}\label{prop:removing-one-sector}
Let $r \in \R$ and let $w \in \Sigma$ be a special vertex.
Then the complex $Z = U_h(r) \cup \overline{K_w(\sigma^{\op})}$ has a filtration
\[
U_h(r) = Z_0 \lneq Z_1 \lneq \ldots \lneq Z_n = Z
\]
by subcomplexes $Z_i$ such that the following hold for each $0 \leq m < n$:
\begin{enumerate}
\item $Z_{m+1} = Z_{m} \cup \overline{C_{m+1}}$ for some chamber
$C_{m+1} \subseteq Z_{m+1}$ that satisfies ${\ell_{Z_{m+1}}(C_{m+1}) = 0}$.
\item $Z_{m} \cap \overline{C_{m+1}} = Z_{m} \cap \partial(\st_{Z_{m+1}}(C_{m+1}^{\downarrow}))$.
\item $\st_{Z_{m+1}}(C_{m+1}^{\downarrow}) \subseteq \overline{C_{m+1}}$.
\end{enumerate}
\end{proposition}
\begin{proof}
By the third claim of Proposition~\ref{prop:intersection-cones} there is some $\varepsilon > 0$ such that $L_h(r) \subseteq h^{-1}((-\infty,r+\varepsilon])$.
Since $\overline{K_w(\sigma^{\op})} \cap {h^{-1}((-\infty,r+\varepsilon])}$ is compact by Lemma~\ref{lem:truncated-sectors}, it follows that there are only finitely many cells in $U_h(r) \cup \overline{K_w(\sigma^{\op})}$ not lying in $U_h(r)$.
We saw in Lemma~\ref{lem:intersection-cones} that $Y = L_h(r) \cap \overline{K_w(\sigma^{\op})}$ is $\sigma$-convex.
Hence Proposition~\ref{cor:deconstr-sigma-conv-complexes} provides us with a filtration
\[
R(Y) = Y_0 \lneq Y_1 \lneq \ldots \lneq Y_n = Y
\]
such that for each $0 \leq m < n$ the following hold
\begin{enumerate}
\item $Y_{m+1} = Y_{m} \cup \overline{D_{m+1}}$ for some chamber
$D_{m+1} \subseteq Y_{m+1}$ that satisfies $\ell_{Y_{m+1}}(D_{m+1}) = 0$,
\item $Y_{m} \cap \overline{D_{m+1}} = Y_{m} \cap \partial(\st_{Y_{m+1}}(D_{m+1}^{\downarrow}))$, and
\item $\st_{Y_{m+1}}(D_{m+1}^{\downarrow}) \subseteq \overline{D_{m+1}}$.
\end{enumerate}
We claim that we obtain the desired filtration be setting $Z_0 = U_h(r)$, $C_m = D_m$, and $Z_{m+1} = Z_{m} \cup \overline{C_{m+1}}$ for every $0 \leq m < n$.
To prove the first claim, it suffices to show that $\ell_{Z_{m+1}}(C_{m+1}) = 0$.
By construction we know that $\ell_{Y_{m+1}}(D_{m+1}) = 0$.
Suppose that $\ell_{Z_{m+1}}(C_{m+1}) > 0$.
Then there is a panel $P$ of $C_{m+1}$ such that $C_{m+1} \neq \pr_P(\sigma) \subseteq Z_{m+1}$.
This shows that $C_{m+1} = \pr_P(\sigma^{\op})$.
Since $\ell_{Y_{m+1}}(D_{m+1}) = 0$ it follows that $\pr_P(\sigma)$ does not lie in $Y_{m+1}$ and hence that $\pr_P(\sigma) \subseteq U_h(r)$.
Thus there is a vertex $u \in M(r)$ such that $\pr_P(\sigma)$ is a chamber of $K_u(\sigma^{\op})$.
In particular this implies $P \subset \overline{K_u(\sigma^{\op})}$ so that we can apply Lemma~\ref{lem:proj-in-sectors} to deduce that $\pr_P(\sigma^{\op}) = C_{m+1}$ lies in $K_u(\sigma^{\op})$.
But this is a contradiction to our assumption that $C_{m+1} \subseteq L_h(r)$.
To prove the third claim, let $A \subseteq \st_{Z_{m+1}}(C_{m+1}^{\downarrow})$ be a cell.
Since $\st_{Y_{m+1}}(C_{m+1}^{\downarrow}) \subseteq \overline{C_{m+1}}$, it suffices to consider the case where $A$ does not lie in $\st_{Y_{m+1}}(C_{m+1}^{\downarrow})$.
Then $A$ is a coface of $C_{m+1}^{\downarrow}$ that lies in $U_h(r)$.
Hence there is a vertex $u \in M(r)$ such that $A \subseteq \overline{K_u(\sigma^{\op})}$.
In particular $C_{m+1}^{\downarrow} \subseteq \overline{K_u(\sigma^{\op})}$ so that Lemma~\ref{lem:proj-in-sectors} implies
\[
C_{m+1} = \pr_{C_{m+1}^{\downarrow}}(\sigma^{\op}) \subseteq K_u(\sigma^{\op}) \subseteq U_h(r).
\]
But this is a contradiction to our assumption that $C_{m+1} \subseteq Y_{m+1} \subseteq L_h(r)$.
Let us verify the second claim.
To this end, we apply the third claim to obtain
\[
\partial(\st_{Z_{m+1}}(C_{m+1}^{\downarrow}))
\subseteq \overline{\st_{Z_{m+1}}(C_{m+1}^{\downarrow})}
\subseteq \overline{C_{m+1}}.
\]
Thus we have $Z_m \cap \partial(\st_{Z_{m+1}}(C_{m+1}^{\downarrow})) \subseteq Z_m \cap \overline{C_{m+1}}$.
To prove the reverse inclusion, let $A$ be a cell in $Z_{m} \cap \overline{C_{m+1}}$.
Since
\begin{align*}
Y_{m} \cap \overline{C_{m+1}} &= Y_{m} \cap \overline{D_{m+1}}\\
&= Y_{m} \cap \partial(\st_{Y_{m+1}}(D_{m+1}^{\downarrow}))\\
&\subseteq Z_{m} \cap \partial(\st_{Y_{m+1}}(C_{m+1}^{\downarrow})),
\end{align*}
it suffices to consider the case when $A$ does not lie in $Y_{m}$.
Then there is a vertex $u \in M(r)$ such that $A$ lies in the (open) sector $K_u(\sigma^{\op})$.
But this implies $C_{m+1} \subseteq \st_{\Sigma}(A) \subseteq K_u(\sigma^{\op})$, which contradicts our choice of $C_{m+1}$ and completes the proof.
\end{proof}

\section{The positive direction in top dimension}\label{sec:pos-dir}

In the previous section we considered filtrations of certain subcomplexes of Euclidean Coxeter complexes.
In this section we will apply this filtration to filter subcomplexes of Euclidean buildings that appear as preimages of retractions from infinity.

\subsection{Height functions on Euclidean buildings}\label{subsec:height-on-X}

Let $X$ be a thick, $d$-dimensional Euclidean building.
We fix an apartment $\SigmaStd$ in $X$ and a chamber $\sigma \subseteq \pinf \SigmaStd$.
As in the previous section, we think of $\SigmaStd$ as a Euclidean vector space whose origin is given by a fixed special vertex $\vstd \in \SigmaStd$.
In particular, this allows us to consider the dual space $\SigmaStd^{\ast} = \Hom(\SigmaStd,\R)$.
Our first goal is to extend the linear forms in $\SigmaStd^{\ast}$ to height functions on $X$.
To this end we consider the set $\mathcal{A}_{\sigma}$ of apartments in $X$ that contain a subsector of $K_{\vstd}(\sigma) \subseteq \SigmaStd$.
From~\cite[Theorem 11.63.(1)]{AbramenkoBrown08} we know $X$ is covered by the apartments in $\mathcal{A}_{\sigma}$.
As a consequence, the following definitions makes sense.

\begin{definition}\label{def:apartments-towards-sigma}
For each $\Sigma \in \mathcal{A}_{\sigma}$, let
$f_{\Sigma} \colon \Sigma \rightarrow \SigmaStd$ denote the isomorphism given by the building axiom (B2).
We define the \emph{retraction from infinity corresponding to $\SigmaStd$ and $\sigma$} as the function $\rho \defeq \rho_{\sigma, \SigmaStd} \colon X \rightarrow \SigmaStd$ that is given by $x \mapsto f_{\Sigma}(x)$, where $\Sigma$ is any apartment in $\mathcal{A}_{\sigma}$ containing $x$.
\end{definition}


\begin{definition}\label{def:retr-height-functions}
Let $X^{\ast}_{\sigma,\vstd} = \Set{\alpha \circ \rho}{\alpha \in \SigmaStd^{\ast}}$ denote the real vector space of functions $f \colon X \rightarrow \R$ that are given by precomposing linear forms $\alpha \in \SigmaStd^{\ast}$ with $\rho$.
\end{definition}

As notation suggests, the space $X^{\ast}_{\sigma,\vstd}$ does not depend on $\SigmaStd$.
More precisely, this means that $X^{\ast}_{\sigma,\vstd} = \Set{\alpha \circ \rho_{\sigma,\Sigma}}{\alpha \in \Sigma^{\ast}}$, where $\Sigma \in \mathcal{A}_{\sigma}$ is any apartment with $\vstd$ as its origin.
Indeed, for every $\alpha \in \SigmaStd^{\ast}$ we have
\[
\alpha \circ \rho_{\sigma,\SigmaStd}
= \alpha \circ f_{\Sigma} \circ \rho_{\sigma,\Sigma}
= \alpha' \circ \rho_{\sigma,\Sigma}
\]
with $\alpha' = \alpha \circ f_{\Sigma} \in \Sigma^{\ast}$, which shows that $X^{\ast}_{\sigma,\vstd}$ is contained in $\Set{\alpha \circ \rho_{\sigma,\Sigma}}{\alpha \in \Sigma^{\ast}}$.
The same argument also shows the reverse inclusion.

\begin{remark}\label{rem:another-char-of-retr-pres}
Note that $X^{\ast}_{\sigma,\vstd}$ can be defined as the space of $\rho$-invariant real valued functions on $X$ that are linear on $\SigmaStd$,\ i.e.
\[
X^{\ast}_{\sigma,\vstd} = \Set{f \colon X \rightarrow \R}{f \circ \rho = f \text{ and } f_{|\SigmaStd} \in \SigmaStd^{\ast}}.
\]
\end{remark}

For the rest of this section we fix a height function $h = \alpha \circ \rho \in X^{\ast}_{\sigma,\vstd}$, where $\alpha \circ [x,\xi) \colon [0,\infty) \rightarrow \R$ is strictly decreasing for every $\xi \in \overline{\sigma}$ and every $x \in \SigmaStd$.
Our goal is to prove that, under certain conditions on $X$, the system of superlevelsets $(X_{h \geq r})_{r \in \R}$ is essentially $(d-2)$-connected.
In the following sections we will see how this result can be used to determine the essential connectivity properties for more general height functions.

\subsection{Deconstructing retraction preimages}

Many of the complexes we are going to study involve the $\rho$-preimages of certain subsets of $\SigmaStd$.
The following lemma provides us with a criterion to ensure that the image of a minimal gallery under $\rho$ stays minimal.

\begin{lemma}\label{lem:rho-inv-sigma-min-gal}
Let $A$ be a cell in $X$ and let $C = \pr_{A}(\sigma)$.
\begin{enumerate}
\item For every chamber $D \subseteq \st_X(A)$ there is an apartment $\Sigma \in \mathcal{A}_{\sigma}$ such that $C,D \subseteq \Sigma$.
\item If $\Gamma = C_1 \vert \ldots \vert C_n \subseteq \st_X(A)$ is a minimal gallery terminating in $C$, then $\rho(\Gamma) = \rho(C_1) \vert \ldots \vert \rho(C_n)$ is a minimal gallery in $\SigmaStd$ that terminates in $\pr_{\rho(A)}(\sigma)$.
\end{enumerate}
\end{lemma}
\begin{proof}
To prove the first claim let $\Sigma \in \mathcal{A}_{\sigma}$ be an apartment containing~$D$.
Let further $x \in A$ and $\xi \in \sigma$ be (interior) points.
Since $\sigma \subseteq \pinf \Sigma$ it follows that the ray $[x,\xi)$ is contained in $\Sigma$.
On the other hand, there is an initial segment of $(x,\xi)$ that lies in $\pr_{A}(\sigma) = C$.
Thus we see that $\Sigma$ contains a point of $C$, which gives us $C \subseteq \Sigma$.
For the second claim, let $\Sigma \in \mathcal{A}_{\sigma}$ be an apartment containing $C_1$ and $C_n = C$.
Since $\Sigma \in \mathcal{A}_{\sigma}$, it follows that the restriction $\rho_{\vert \Sigma} \colon \Sigma \rightarrow \SigmaStd$ is an isomorphism.
In particular, the minimal gallery $\Gamma$ is mapped to the minimal gallery $\rho(\Gamma)$.
It remains to prove that $C = \pr_{A}(\sigma) = \pr_{\rho(A)}(\sigma)$.
Let $\xi \in \sigma$ and let $x \in A$.
It is sufficient to show that the image of $[x,\xi)$ under $\rho$ is a geodesic ray of the form $[\rho(x),\xi)$.
But this follows from our assumption that $\Sigma$ lies in $\mathcal{A}_{\sigma}$, which implies $\xi \in \pinf \Sigma$.
\end{proof}

\begin{notation}\label{not:opposition-in-stars}
Let $A \subseteq X$ be a cell of codimension at least $1$ and let $B,C$ be cofaces of $A$.
We say that $B$ and $C$ are opposite in $\st_X(A)$, denoted by $B \op_{\st_X(A)} C$, if their induced simplices in $\lk_X(A)$ are opposite to each other.
\end{notation}

In order to switch easily between subsets of $\SigmaStd$ and their $\rho$-preimages in $X$, we introduce the following notation.

\begin{notation}\label{not:hat}
For each subset $Z \subseteq \SigmaStd$ we write $\widehat{Z} \defeq \rho^{-1}(Z)$.
\end{notation}

In view of Lemma~\ref{lem:rho-inv-sigma-min-gal}, we see that $\rho$ commutes with taking opposite chambers of $\pr_A(\sigma)$.
More precisely, this can be stated as follows.

\begin{corollary}\label{cor:opposition-compatible}
Let $A \subseteq X$ be a cell of codimension at least $1$ and let $C = \pr_A(\sigma)$.
Then $\pr_{\rho(A)}(\sigma) = \rho(C)$ and for every subcomplex $Z \subseteq \SigmaStd$ we have
\[
\Set{D \subseteq \st_Z(A)}{D \op_{\st_X(A)} C} = \Set{D \subseteq \st_Z(A)}{\rho(D) \op_{\st_{\SigmaStd}(\rho(A))} \rho(C)}.
\]
\end{corollary}

The following definition extends Definition~\ref{def:upper-face} to the case of Euclidean buildings.

\begin{definition}\label{def:upper-face-building}
For each chamber $C$ in $\Sigma$ we define its \emph{upper face} $C^{\uparrow}$ as the intersection of all panels $P \subset C$ such that $\pr_P(\sigma) = C$ (the face $U$ in Lemma~\ref{lem:upper-face}).
The lower face $E^{\downarrow}$ of $E$ is defined to be the intersection of all panels $P \subset C$ with $\pr_P(\sigma) \neq C$.
\end{definition}

Let us note the following consequence of Corollary~\ref{cor:opposition-compatible}, which we formulate in terms of links rather than stars.

\begin{corollary}\label{cor:link-is-opp-complex}
Let $Z \subseteq \SigmaStd$ be a subcomplex.
Suppose that $Z$ contains a chamber $D$ such that $\st_Z(D^{\downarrow}) \subseteq \overline{D}$ and let $A$ be a cell in $\rho^{-1}(D^{\downarrow})$.
Let $a \subseteq \lk_X(A)$ be the chamber that corresponds to $\pr_{A}(\sigma)$.
Then $\lk_{\widehat{Z}}(A) = \Opp_{\lk_X(A)}(a)$.
\end{corollary}
\begin{proof}
In view of Corollary~\ref{cor:opposition-compatible}, it suffices to show that $\rho(D)$ is opposite to $\rho(\pr_{A}(\sigma))$ in $\st_{\SigmaStd}(\rho(A))$.
Let $\Sigma \in \mathcal{A}_{\sigma}$ be an apartment containing $D$ and let $\sigma_{\Sigma}^{\op}$ be the opposite chamber of $\sigma$ in $\pinf \Sigma$.
Then we have $D = \pr_{D^{\downarrow}}(\sigma_{\Sigma}^{\op})$, which is opposite to $\pr_{D^{\downarrow}}(\sigma)$ in $\op_{\st_{X}(D^{\downarrow})}$.
From Lemma~\ref{lem:rho-inv-sigma-min-gal} it therefore follows that $\rho(D)$ is opposite to
\[
\rho(\pr_{D^{\downarrow}}(\sigma)) = 
\pr_{\rho(D^{\downarrow})}(\sigma) = \pr_{\rho(A)}(\sigma)
\]
in $\rho(\st_{X}(D^{\downarrow})) = \st_{\SigmaStd}(\rho(A))$.
\end{proof}

\begin{definition}\label{def:link-property}
Let $\Delta$ be a spherical building.
We say that $\Delta$ has spherical opposition complexes if $\Opp_{\Delta}(C)$ is spherical for every chamber $C \subseteq \Delta$.
In this case we also say that $\Delta$ satisfies $(\SO)$.
If moreover the link of every simplex $A \subseteq \Delta$ satisfies $(\SO)$, then we say that $\Delta$ satisfies $(\SOL)$.
Similarly, we say that a Euclidean building satisfies $(\SOL)$ if all of its links satisfy $(\SO)$.
\end{definition}

Recall from Theorem~\ref{thm:connectivity-of-opp} that thick enough buildings $\Delta$ of type $X_n \in M \defeq \Set{X_i}{X \in \{A,C,D\},\ i \in \N}$ satisfy $(\SO)$, unless $\Delta$ is an exceptional $C_3$ building.
This implies that thick enough joins $\Delta = \ast_{i=1}^{n} \Delta_i$ of non-exceptional spherical buildings $\Delta_i$ of type $X_{n_i} \in M$ satisfy $(\SO)$.
Indeed, this directly follows from Lemma~\ref{lem:connected-joins} and the easy observation that the opposition complex of a chamber $c = \ast_{i=1}^{n} c_i \subseteq \Delta$ is given by $\Opp_{\Delta}(c) = \ast_{i=1}^{n} \Opp_{\Delta_i}(c_i)$.
Regarding the constants from Theorem~\ref{thm:connectivity-of-opp} and the types of links that appear in buildings of type $X_n \in M$, we deduce the following.

\begin{example}\label{ex:SOL}
Let $\Delta$ be an arbitrary building of type $A_{n+1}$, $C_{n+1}$ or $D_{n+1}$, but not an exceptional $C_3$ building.
Assume that $\thick(\Delta) \geq 2^{n}+1$ in the $A_{n+1}$ case, respectively $\thick(\Delta) \geq 2^{2n+1}+1$ in the other two cases.
Then $\Delta$ satisfies $(\SOL)$.
\end{example}

\begin{lemma}\label{lem:remove-one-chamber}
Suppose that $X$ satisfies $(\SOL)$.
Let $Z \subseteq \SigmaStd$ be a subcomplex that contains a chamber $D$ with $\st_{Z}(D^{\downarrow}) \subseteq \overline{D}$ and let $Y = Z \backslash \st_{\SigmaStd}(D^{\downarrow})$.
Then the inclusion $\iota \colon \widehat{Y} \rightarrow \widehat{Z}$ induces a monomorphism
\[
\pi_{k}(\iota) \colon \pi_{k}(\widehat{Y}) \rightarrow \pi_{k}(\widehat{Z})
\]
for every $0 \leq k \leq d-2$.
\end{lemma}
\begin{proof}
Let $I$ be the set of cells in $\rho^{-1}(D^{\downarrow})$.
Then $\widehat{Z}$ can be written as $\widehat{Z} = \widehat{Y} \cup \bigcup \limits_{A \in I} \overline{\st_{\widehat{Z}}(A)}$.
Note also that $\overline{\st_{\widehat{Z}}(A)} \cap \overline{\st_{\widehat{Z}}(B)} \subseteq \widehat{Y}$ for all $A,B \in I$ with $A \neq B$.
In order to apply Lemma~\ref{lem:gluing-for-inclusion}, it therefore remains to verify that $\overline{\st_{\widehat{Z}}(A)} \cap \widehat{Y}$ is $(d-2)$-connected for every $A \in I$.
To this end, we observe that
\[
\overline{\st_{\widehat{Z}}(A)} \cap \widehat{Y}
= \partial \st_{\widehat{Z}}(A)
\cong \partial A \ast \lk_{\widehat{Z}}(A).
\]
On the other hand, we have $\lk_{\widehat{Z}}(A) = \Opp_{\lk_X(A)}(a)$ by Corollary~\ref{cor:link-is-opp-complex}, where $a$ is the chamber induced by $\pr_{A}(\sigma)$.
The $(\SOL)$-property of $X$ tells us that $\Opp_{\lk_{X}(A)}(\pr_{A}(\sigma))$ is $(\dim(\lk_{X}(A))-1)$-connected.
Since moreover $\partial A$ is homeomorphic to a sphere of dimension $\dim(A)-1$, it follows from Lemma~\ref{lem:connected-joins} that
$\Opp_{\lk_{X}(A)}(\pr_{A}(\sigma)) \ast \partial A$ is $k$-connected for
\[
k = (\dim(\lk_{X}(A))-1)+(\dim(A)-2)+2.
\]
Now the claim follows since $\dim(A)+\dim(\lk_{X}(A)) = d-1$.
\end{proof}

Let us translate the notion of the upper complex from Definition~\ref{def:lower-complex} to our fixed apartment $\SigmaStd$.

\begin{definition}
For every $r \in \R$ let $M(r)$ denote the set of special vertices $w \in \SigmaStd$ of height $h(w) \geq r$.
We define the \emph{upper complex} in $\SigmaStd$ associated to $h$ and  $r \in \R$ by $U_h(r) = \bigcup \limits_{w \in M(r)} \overline{K_w(\sigma^{\op})}$, where $\sigma^{\op}$ denotes the opposite chamber of $\sigma$ in $\pinf \SigmaStd$.
\end{definition}

\begin{proposition}\label{prop:lower-complex-mono}
Suppose that $X$ satisfies $(\SOL)$.
Then for every $r \in \R$ and every special vertex $v \in \SigmaStd$, the canonical inclusion
\[
\iota \colon \widehat{U_h(r)} \rightarrow \widehat{U_h(r)} \cup \widehat{\overline{K_v(\sigma^{\op})}}
\]
induces monomorphisms
\[
\pi_k(\iota) \colon \pi_k(\widehat{U_h(r)}) \rightarrow \pi_k(\widehat{U_h(r)} \cup \widehat{\overline{K_v(\sigma^{\op})}})
\]
for every $0 \leq k \leq d-2$.
\end{proposition}
\begin{proof}
By Proposition~\ref{prop:removing-one-sector} there is a filtration
\[
U_h(r) = Z_0 \lneq Z_1 \lneq \ldots \lneq Z_n = U_h(r) \cup \overline{K_v(\sigma^{\op})}
\]
by subcomplexes $Z_i$ such that the following is satisfied for each $0 \leq m < n$.
\begin{enumerate}
\item $Z_{m+1} = Z_{m} \cup \overline{C_{m+1}}$ for some chamber
$C_{m+1} \subseteq Z_{m+1}$ with\\$\ell_{Z_{m+1}}(C_{m+1}) = 0$.
\item $Z_{m} \cap \overline{C_{m+1}} = Z_{m} \cap \partial(\st_{Z_{m+1}}(C_{m+1}^{\downarrow}))$.
\item $\st_{Z_{m+1}}(C_{m+1}^{\downarrow}) \subseteq \overline{C_{m+1}}$.
\end{enumerate}
In view of Lemma~\ref{lem:remove-one-chamber} this gives us a filtration
\[
\widehat{U_h(r)} = \widehat{Z_0} \lneq \widehat{Z_2} \lneq \ldots \lneq \widehat{Z_n} = \widehat{U_h(r)} \cup \widehat{\overline{K_v(\sigma^{\op})}}
\]
such that the each inclusion induces a monomorphism
\[
\pi_{k}(\iota) \colon \pi_{k}(\widehat{Z_m}) \rightarrow \pi_{k}(\widehat{Z_{m+1}})
\]
for every $0 \leq k \leq d-2$.
Now the claim follows be composing these monomorphisms.
\end{proof}

\begin{corollary}\label{cor:homotopy-lower-complex}
Suppose that $X$ satisfies $(\SOL)$.
Then $\widehat{U_h(r)}$ is $(d-2)$-connected for every $r \in \R$.
\end{corollary}
\begin{proof}
Let $0 \leq k \leq d-2$ be an integer and let $f \colon S^{k} \rightarrow \widehat{U_h(r)}$ be a continuous function.
Since $X$ is a $\CAT(0)$-space it is contractible.
Hence there is a compact subspace $Z \subseteq X$ such that $f$ can be contracted in $Z$.
Then $\rho(Z) \subseteq \SigmaStd$ is also compact and Lemma~\ref{lem:sector-covering-compact} implies that there is a special vertex $v \in \SigmaStd$ such that $\rho(Z)$ is contained in $\overline{K_v(\sigma^{\op})}$.
In particular we see that $f$ is contractible in $\widehat{U_h(r)} \cup \widehat{\overline{K_v(\sigma^{\op})}}$.
As a consequence, $f$ represents the trivial element in $\pi_k(\widehat{U_h(r)} \cup \widehat{\overline{K_v(\sigma^{\op})}})$.
On the other hand, Proposition~\ref{prop:lower-complex-mono} tells us that the inclusion
\[
\iota \colon \widehat{U_h(r)} \rightarrow \widehat{U_h(r)} \cup \widehat{\overline{K_v(\sigma^{\op})}}
\]
induces a monomorphism
\[
\pi_k(\iota) \colon \pi_k(\widehat{U_h(r)}) \rightarrow \pi_k(\widehat{U_h(r)} \cup \widehat{\overline{K_v(\sigma^{\op})}})
\]
for each $0 \leq k \leq d-2$.
Thus $f$ represents the trivial element in $\pi_k(\widehat{U_h(r)})$, which proves the claim.
\end{proof}

We are now ready to prove the main result of this section.
For easy reference, we recall some of the involved notation.

\begin{theorem}\label{thm:essentially-connected}
Let $X$ be a thick Euclidean building, let $\SigmaStd$ be an apartment in $X$, let $\sigma \subseteq \pinf \SigmaStd$ be a chamber, and let $\vstd \in \SigmaStd$ be a special vertex, which we think of as the origin of $\SigmaStd$.
Consider a height function $h \in X^{\ast}_{\sigma,\vstd}$ such that $h \circ [x,\xi) \colon [0,\infty) \rightarrow \R$ is strictly decreasing for every $\xi \in \overline{\sigma}$ and every $x \in \SigmaStd$.
Suppose that $X$ satisfies $(\SOL)$.
Then the system of superlevelsets $(X_{h \geq r})_{r \in \R}$ is essentially $(\dim(X)-2)$-connected.
\end{theorem}
\begin{proof}
Let $r \in \R$ be a real number and let $L_h(r) = \Sigma \backslash \bigcup \limits_{w \in M(r)} K_w(\sigma^{\op})$.
According to Proposition~\ref{prop:intersection-cones} there is some $\varepsilon > 0$ that gives us a chain of inclusions
\begin{equation}\label{eq:essentially-connected}
X_{h \leq r-\varepsilon} \rightarrow \widehat{L_h(r-\varepsilon)} \rightarrow X_{h \leq r}.
\end{equation}
By taking the closure of the complements of the sets in~\eqref{eq:essentially-connected}, we see that the inclusion $\iota \colon X_{h \geq r} \rightarrow X_{h \geq r-\varepsilon}$ factorizes as
\[
X_{h \geq r} \xrightarrow{\iota_1} \widehat{U_h(r-\varepsilon)} \xrightarrow{\iota_2} X_{h \geq r-\varepsilon}.
\]
Since $\widehat{U_h(r-\varepsilon)}$ is $(\dim(X)-2)$-connected by Corollary~\ref{cor:homotopy-lower-complex}, the claim follows from the functoriality of $\pi_k$.
\end{proof}

\begin{remark}\label{rem:is-essential-needed}
Note that Theorem~\ref{thm:essentially-connected} does not tell us that the superlevelsets $X_{h \geq r}$ themselves are $(\dim(X)-2)$-connected, which is typically the case when applying discrete Morse theory.
In fact, I am not aware of a classical Morse theoretic argument that proves Theorem~\ref{thm:essentially-connected}.
Instead, it was crucial to use the concept of \emph{essential} $n$-connectivity in order to apply the highly connected complexes $\widehat{U_h(r)}$, which sit between two $h$-levels.
As far as I know this is the first time that a part of the computation of the $\Sigma$-invariants of a group naturally benefits from the concept of \emph{essential} connectivity.
\end{remark}

\section{The negative direction in top dimension}\label{sec:neg-dir}

In the previous section we proved that certain systems of superlevelsets in an appropriate Euclidean building $X$ are essentially $(\dim(X)-2)$-connected.
In this section we will show that these systems are not essentially $(\dim(X)-1)$-connected.
In fact, we will prove this result for thick Euclidean buildings $X$ of arbitrary type under the mild assumption that some $\alpha \in \Aut(X)$ acts non-trivially on the superlevelsets of $X$.

\medskip

\noindent Let us fix a thick, $d$-dimensional Euclidean building $X$, an apartment $\SigmaStd \subset X$, a special vertex $\vstd \in \SigmaStd$, and a pair of opposite chambers $\sigma,\sigma^{\op} \subseteq \pinf \SigmaStd$.
The set of apartments in $X$ that contain a subsector of $K_{\vstd}(\sigma)$ will be denoted by $\mathcal{A}_{\sigma}$.

\subsection{The abstract cone}

We start by constructing some auxiliary cell complexes that will help us to ``pull compact subcomplexes of $\pinf X$ into $X$''.

\begin{lemma}\label{lem:ap-containing-opp-sectors}
For every chamber $\delta \subseteq \Opp_{\pinf X}(\sigma)$ there is a unique apartment $\Sigma \in \mathcal{A}_{\sigma}$ such that $\delta \subseteq \pinf \Sigma$.
\end{lemma}
\begin{proof}
The existence of $\Sigma$ is guaranteed by Theorem~\cite[11.63.(2)]{AbramenkoBrown08}.
The uniqueness statement follows from the easy observation that every apartment is the convex hull of a pair of its sectors that correspond to opposite chambers.
\end{proof}

In view of Lemma~\ref{lem:ap-containing-opp-sectors} the following definition makes sense.

\begin{definition}\label{def:aps-induced-by-s}
For every chamber $\delta \subseteq \Opp_{\pinf X}(\sigma)$ let $\Sigma_{\delta} \in \mathcal{A}_{\sigma}$ denote the unique apartment with $\delta \subseteq \pinf \Sigma_{\delta}$.
\end{definition}

For the rest of this section we fix a compact subcomplex $S$ of $\Opp_{\pinf X}(\sigma)$ in which all maximal simplices are chambers.
Recall that we write $\Ch(S)$ to denote the set of chambers in $S$.

\begin{lemma}\label{lem:special-vertex-in-intersection}
There is a special vertex $v \in \bigcap \limits_{\delta \in \Ch(S)} \Sigma_{\delta}$.
\end{lemma}
\begin{proof}
Since $S$ is compact, the lemma follows by inductive application of Proposition~\ref{prop:common-subsector}, which tells us that any two subsectors of $K_{\vstd}(\sigma)$ contain a common subsector.
\end{proof}

From now on we fix a special vertex $v$ as in Lemma~\ref{lem:special-vertex-in-intersection}.
Moreover we consider the subcomplex
\[
K_{S,v} = \bigcup \limits_{\delta \in \Ch(S)} \overline{K_v(\delta)}.
\]
of $X$.
As before, we will write $\widehat{Z} = \rho^{-1}(Z)$ to denote the preimage of a subset $Z \subseteq \SigmaStd$ under the retraction $\rho = \rho_{\sigma, \SigmaStd}$.

\begin{remark}\label{rem:where-k-lives}
Note that $K_{S,v}$ is a subcomplex of $\widehat{\overline{K_{v}(\sigma^{\op})}}$ and that $\rho$ restricts to an isomorphism $\rho_{\vert \overline{K_{v}(\tau)}} \colon \overline{K_{v}(\tau)} \rightarrow \overline{K_{v}(\sigma^{\op})}$ for every chamber $\tau \subseteq S$.
\end{remark}


\begin{definition}\label{def:auxiliary-cone}
Consider the disjoint union of closed sectors $\coprod \limits_{\delta \in \Ch(S)} \overline{K_v(\delta)}$.
We say that two points $(p,\delta),(p',\delta') \in \coprod \limits_{\delta \in \Ch(S)} \overline{K_v(\delta)}$ are \emph{equivalent}, denoted by $(p,\delta) \sim (p',\delta')$, if $p=p' \in \partial \overline{K_v(P)}$ for some panel $P \subseteq \Opp_{\pinf X}(\sigma)$.  
We define the \emph{abstract cone} as the quotient space $\widetilde{K}_{S,v} \defeq \coprod \limits_{\delta \in \Ch(S)} \overline{K_v(\delta)} / \hspace{-1.5mm} \sim$.
\end{definition}

Note we have a well-defined projection $\pi \colon \widetilde{K}_{S,v} \rightarrow K_{S,v}$, which is given by $\pi([(p,\delta)]) = p$.
To simplify the notation, we will often write $K \defeq K_{S,v}$ and $\widetilde{K} \defeq \widetilde{K}_{S,v}$.

\subsection{Homology of superlevelsets}

Recall from Definition~\ref{def:retr-height-functions} and Remark~\ref{rem:another-char-of-retr-pres} that $X^{\ast} \defeq X^{\ast}_{\sigma,\vstd} = \Set{\alpha \circ \rho}{\alpha \in \SigmaStd^{\ast}}$ denotes the space of $\rho$-invariant functions on $X$ whose restrictions to $\SigmaStd$ are linear.
Let us fix a function $h \in X^{\ast}$ for which $h \circ [x,\xi) \colon [0,\infty) \rightarrow \R$ is strictly decreasing for every $x \in \SigmaStd$ and every $\xi \in \overline{\sigma}$.

\begin{definition}\label{def:branching-number}
Let $A$ be a cell in $K$.
The \emph{branching number} of $A$, denoted by $b(A)$, is the number of chambers $\tau \subseteq S$ such that $A$ is contained in $K_v(\tau)$.
\end{definition}

Note for example that $b(v)=\vert \Ch(S) \vert$ and that $b(E)$ is the number of chambers in $\pi^{-1}(E)$ when $E$ is a chamber in $K$.

\medskip

\noindent In order to prove that $(X_{h \geq r})_{r \in \R}$ is not essentially $(d-1)$-connected, we will construct sequences of cycles in $C_{d-1}(X_{h \geq r};\F_2)$.
These cycles will appear as boundaries of $d$-chains in $C_{d}(X;\F_2)$, whose coefficients will depend on the branching numbers of certain cells.

\begin{notation}\label{def:support-chain}
For every $k \in \N_0$ and every $k$-chain
\[
c = \sum \limits_{A \subseteq X^{(k)}} \lambda_A \cdot A \in C_k(X;\F_2),
\]
let $\supp(c)$ denote the set of all $k$-cells $A$ with $\lambda_A = 1$.
\end{notation}

A nice feature of working with affine cell complexes is that the attaching map is a homeomorphism for each closed cell.
In this case, the cellular boundary formula (see~\cite[Section 2.2]{Hatcher02}) tells us that each boundary map $\partial_k$ is given by
\begin{equation}\label{eq:branching-numb-in-boundary}
\partial_k \colon C_{k}(X;\F_2) \rightarrow C_{k-1}(X;\F_2), c \mapsto \sum \limits_{A \subseteq X^{(k-1)}} \lambda_A \cdot A,
\end{equation}
where $\lambda_A$ denotes the number of $k$-dimensional cofaces of $A$ in $\supp(c)$.

\begin{lemma}\label{lem:no-branching}
For every $\delta \in \Ch(S)$ there is a special vertex $w \in K_{v}(\delta)$ such that $\overline{K_{w}(\delta)} \cap \Sigma_{\tau} = \emptyset$ for every $\tau \in \Ch(S) \backslash \{\delta\}$.
\end{lemma}
\begin{proof}
Let $\xi \in \delta$ be an arbitrary point.
Since $v$ lies in $\Sigma_{\delta}$, it follows that the ray $[v,\xi)$ is contained in $\Sigma_{\delta}$.
On the other hand, we have $\xi \notin \tau$ for every $\tau \in \Ch(S) \backslash \{\delta\}$.
Since $\tau$ is the unique chamber in $\partial \Sigma_{\tau}$ that is opposite to $\sigma$ we obtain $\xi \notin \partial \Sigma_{\tau}$.
It follows that for every $\tau \in \Ch(S) \backslash \{\delta\}$ there is a number $T_{\tau} > 0$ such that the point $[v,\xi)(T_{\tau})$ is not contained in $\Sigma_{\tau}$.
Since $S$ is finite we can choose $T$ such that $p \defeq [v,\xi)(T) \notin K_{v}(\tau)$ for every $\tau \in \Ch(S) \backslash \{\delta\}$.
Consider now an arbitrary point $x \in \overline{K_p(\delta)}$.
Suppose that $x \in \Sigma_{\tau}$ for some $\tau \in \Ch(S) \backslash \{\delta\}$.
Since $\sigma \subseteq \pinf \Sigma_{\tau}$, it then follows that $\overline{K_x(\sigma)}$ is contained in $\Sigma_{\tau}$.
On the other hand, the description of sectors given in Remark~\ref{rem:description-of-sectors} implies that $\overline{K_x(\sigma)}$ contains $p$, which contradicts our choice of $p$.
Now the claim follows by defining $w$ as a special vertex in $K_p(\delta)$.
\end{proof}

Let us observe the following direct consequence of Lemma~\ref{lem:no-branching}.

\begin{corollary}\label{cor:no-branching}
For every $\delta \in \Ch(S)$ there is a special vertex $w \in K_{v}(\delta)$ such that $b(A)=1$ for every cell $A \subseteq \overline{K_{w}(\delta)}$.
\end{corollary}

Note that our height function $h$ has a natural translation to $\widetilde{K}$ by defining
\[
\widetilde{h} \colon \widetilde{K} \rightarrow \R,\ [(p,\tau)] \mapsto h(p).
\]

Using $\widetilde{h}$ and $h$, we can define the complexes $\widetilde{K}_r = \Set{x \in \widetilde{K}}{\widetilde{h}(x) \leq r}$ and $K_r = K \cap X_{h \leq r}$ for every $r \in \R$, which we will study from now on.

\begin{lemma}\label{lem:on-finiteness-of-chain}
The subspace $K_r \subseteq X$ is compact.
In particular there are only finitely many chambers in $K_r$.
\end{lemma}
\begin{proof}
Recall that $\overline{K_{\rho(v)}(\sigma^{\op})} \cap X_{h \leq r} \subseteq \SigmaStd$ is compact by Lemma~\ref{lem:truncated-sectors}.
Since $S$ is finite, it follows from Remark~\ref{rem:where-k-lives} that $K_r$ is the union of finitely many subcomplexes that are homeomorphic to $\overline{K_{\rho(v)}(\sigma^{\op})} \cap X_{h \leq r}$.
Thus we see that $K_r$ is compact.
\end{proof}

In view of Lemma~\ref{lem:on-finiteness-of-chain}, the following notation makes sense.

\begin{notation}\label{not:chain}
For every $r \in \R$, we define the $d$-chain
\[
c_r \defeq \sum \limits_{E \in \Ch(K_r)} \overline{b(E)} \cdot E \in C_d(X;\F_2).
\]
\end{notation}

\begin{remark}\label{rem:induced-cycle-from-abstr-cone}
Note that $c_r$ can also be described as the image of the chain $\widetilde{c}_r \defeq \sum \limits_{E \in \Ch(\widetilde{K}_r)} \overline{1} \cdot E \in C_d(\widetilde{K};\F_2)$
under $C_{d}(\pi) \colon C_d(\widetilde{K};\F_2) \rightarrow C_d(K;\F_2)$.
\end{remark}

\begin{proposition}\label{prop:non-vanishing-part}
There is a constant $R \in \R$ such that the boundary $\partial_{d}(c_r) \in C_{d-1}(X;\F_2)$ is non-zero for every $r \geq R$.
\end{proposition}
\begin{proof}
Let $\delta \subseteq S$ be a chamber.
By Corollary~\ref{cor:no-branching} there is a special vertex $w \in K_{v}(\delta)$ such that $b(A)=1$ for every cell $A \subseteq K_{w}(\delta)$.
Let $R \in \R$ be large enough such that $K_R$ contains at least one chamber in $K_{w}(\delta)$ and let $r \geq R$.
Since $K_{w}(\delta)$ is not bounded above with respect to $h$, we can find a pair of adjacent chambers $E,F \subseteq K_{w}(\delta)$ such that $E \subseteq K_r$ and $F \nsubseteq K_r$.
Let $P$ be the common panel of $E$ and $F$.
Then $E$ is the unique chamber in $K_r$ that lies in the star of $P$.
In this case we see from~\eqref{eq:branching-numb-in-boundary} that the coefficient of $P$ in $\partial_{d}(c_r)$ is $\overline{1}$, which proves the claim.
\end{proof}

\subsection{Negative essential connectivity}\label{subsec:negative-essential}

Suppose from now on that $\Ch(S)$ consists of the support of a non-trivial cycle
\[
z \defeq \sum \limits_{\delta \in \Ch(S)} \delta \in Z_{d-1}(\Opp_{\pinf X}(\sigma);\F_2).
\]

\begin{lemma}\label{lem:even-in-abstr-cone}
Let $r \in \R$ and let $P \subseteq \widetilde{K}_r$ be a panel.
Suppose that $\Ch(\st_{\widetilde{K}_r}(P))$ consists of an odd number of chambers.
Then there is a chamber $E \subseteq \st_{\widetilde{K}}(P)$ that contains a point $p$ of height $\widetilde{h}(p) > r$.
In particular, the height of every point of $P$ is bounded below by $r-\varepsilon$, where $\varepsilon$ denotes the maximal diameter of a chamber with respect to $\widetilde{h}$.
\end{lemma}
\begin{proof}
It follows from~\eqref{eq:branching-numb-in-boundary} that every panel in $S$ is a face of an even number of chambers in $S$.
As a consequence, each panel $P \subseteq \widetilde{K}$ is a face of an even number of chambers in $\widetilde{K}$.
Since $\abs{\Ch(\st_{\widetilde{K}_r}(P))}$ is odd, it therefore follows that some chamber $E \subseteq \st_{\widetilde{K}}(P)$ is not contained in $\st_{\widetilde{K}_r}(P)$.
By definition, this means that there is a point $p \in E$ that satisfies $\widetilde{h}(p) > r$.
\end{proof}

Let us return to the chains $c_r$ from the previous subsection.
In Proposition~\ref{prop:non-vanishing-part} we showed that $\partial_d(c_r) \in B_{d-1}(K;\F_2)$ is non-trivial for certain $r \in \R$.
The next proposition gives us a lower bound for the height of the panels in $\supp(\partial_d(c_r))$.

\begin{proposition}\label{prop:vanishing-part}
There is a constant $\varepsilon > 0$ such that for every $r \in \R$ the panels $P \in \supp(\partial_d(c_r))$ are contained in $X_{r \geq h \geq r-\varepsilon}$.
\end{proposition}
\begin{proof}
Recall from Remark~\ref{rem:induced-cycle-from-abstr-cone} that $c_r$ is the image of
\[
\widetilde{c}_r = \sum \limits_{E \in \Ch(\widetilde{K}_r)} \overline{1} \cdot E \in C_d(\widetilde{K};\F_2)
\]
under $C_d(\pi)$.
Consider the following commutative diagram:
\begin{center}
\begin{tikzpicture}[scale=2]
\node (abscone1) at (0,0) {$C_d(\widetilde{K}_r;\F_2)$};
\node (cone1) at (2,0) {$C_d(K_r;\F_2)$};
\node (abscone2) at (0,-1) {$C_{d-1}(\widetilde{K}_r;\F_2)$};
\node (cone2) at (2,-1) {$C_{d-1}(K_r;\F_2)$};
\draw (abscone1) edge[->] node[above] {$C_{d}(\pi)$} (cone1);
\draw (abscone2) edge[->, left] node[below] {$C_{d-1}(\pi)$} (cone2);
\draw (abscone1) edge[->] node[left] {$\partial_d$} (abscone2);
\draw (cone1) edge[->,left] node[right] {$\partial_d$} (cone2);
\end{tikzpicture}
\end{center}
From Lemma~\ref{lem:even-in-abstr-cone} we know that the height of all panels in $\supp(\partial_d(\widetilde{c}_r))$ is bounded below by $r-\varepsilon$, where
\[
\varepsilon = \sup \Set{\abs{\widetilde{h}(x) - \widetilde{h}(y)}}{x,y \in C \text{ for some } C \in \Ch(\widetilde{K})}.
\]
In particular, we see that the height of all panels in $\supp(C_{d-1}(\pi) \circ \partial_d(\widetilde{c}_r))$ is bounded below by $r-\varepsilon$.
On the other hand, the above diagram tells us that
\[
\partial_d(c_r) = \partial_d \circ C_{d}(\pi)(\widetilde{c}_r) = C_{d-1}(\pi) \circ \partial_d(\widetilde{c}_r),
\]
which completes the proof.
\end{proof}

\begin{proposition}\label{prop:negative-global-geom-V1}
For every $t > 0$ there is a some $s \in \R$ such that the inclusion $\iota \colon X_{h \geq s+t} \rightarrow X_{h \geq s}$ induces a non-trivial morphism
\[
\widetilde{H}_{d-1}(\iota) \colon \widetilde{H}_{d-1}(X_{h \geq s+t};\F_2) \rightarrow \widetilde{H}_{d-1}(X_{h \geq s};\F_2).
\]
\end{proposition}
\begin{proof}
From Corollary~\ref{cor:no-branching} it follows that there is a number $s \in \R$ and a chamber $E \subseteq X_{h \leq s}$ such that $E \in \supp(c_r)$ for every $r \geq s$.
Moreover Proposition~\ref{prop:non-vanishing-part} tells us that $\partial_{d}(c_r)$ is non-trivial for every sufficiently large $r \in \R$ and Proposition~\ref{prop:vanishing-part} provides us with an $\varepsilon > 0$ such that every panel in $\supp(\partial_d(c_r))$ is contained in $X_{r \geq h \geq r-\varepsilon}$.
Thus we can deduce that $\partial_d(c_r) \in Z_{d-1}(X_{h \geq s+t};\F_2)$ is non-trivial for an appropriate $r > s+t+\varepsilon$.
Suppose that $\partial_d(c_r) \in B_{d-1}(X_{h \geq s};\F_2)$, i.e.\ that there is a chain $c \in C_d(X_{h \geq s};\F_2)$ with $\partial_d(c)=\partial_d(c_r)$.
Since $E \subseteq X_{h \leq s}$, it follows that $E \notin \supp(c)$, which implies $c \neq c_r$.
But this is a contradiction to the uniqueness statement in Lemma~\ref{lem:unique-bounding-disc}.
Together this shows that $\partial_d(c_r)$ represents a non-trivial element in $\widetilde{H}_{d-1}(X_{h \geq s+t};\F_2)$ and $\widetilde{H}_{d-1}(X_{h \geq s};\F_2)$, which completes the proof.
\end{proof}

Let us prove the main result of this section.
For easier reference we recall the assumptions on $X$ we made along the way.

\begin{theorem}\label{thm:negative-global-geom}
Let $X$ be a thick, $d$-dimensional Euclidean building, let $\SigmaStd \subset X$ be an apartment, let $\sigma \subseteq \pinf \SigmaStd$ be a chamber, let $\vstd \in \SigmaStd$ be a special vertex, and let $h \in X_{\sigma,v}^{\ast}$ be such that $h \circ [x,\xi)$ is strictly decreasing for every $x \in X$ and every $\xi \in \overline{\sigma}$.
Suppose there is an $\alpha \in \Aut(X)$ such that $h(\alpha(x)) = h(x) + a$ for some $a \in \R \setminus \{0\}$ and every $x \in X$.
Then $(X_{h \geq r})_{r \in \R}$ is not essentially $(d-1)$-acyclic.
\end{theorem}
\begin{proof}
Since $X$ is thick it follows that $\thick(\pinf X) = \infty$.
Moreover the strong transitivity of the action of $\Aut(X)$ on $X$ implies that $\Aut(\pinf X)$ acts strongly transitively on $\pinf X$.
In this case Proposition~\ref{prop:existence-of-opposite-app} provides us with an apartment $S$ in $\Opp_{\pinf X}(\sigma)$.
In particular, $\Ch(S)$ consists of the support of a non-trivial cycle in $Z_{d-1}(\Opp_{\pinf X}(\sigma);\F_2)$, which allows us to apply the results in this Subsection.
Suppose that $(X_{h \geq r})_{r \in \R}$ is essentially $(d-1)$-acyclic.
Then the inclusion induces a trivial morphism
\[
\widetilde{H}_{d-1}(X_{h \geq r+R};\F_2) \rightarrow \widetilde{H}_{d-1}(X_{h \geq r};\F_2)
\]
for appropriate $r \in \R$ and $R>0$.
From our assumption on $\alpha$ it follows that
\[
\widetilde{H}_{d-1}(X_{h \geq r+R+a \cdot k};\F_2) \rightarrow \widetilde{H}_{d-1}(X_{h \geq r + a \cdot k};\F_2)
\]
is trivial for every $k \in \Z$.
Moreover, Proposition~\ref{prop:negative-global-geom-V1} provides us with a constant $s \in \R$ such that
\[
\widetilde{H}_{d-1}(X_{h \geq s+t};\F_2) \rightarrow \widetilde{H}_{d-1}(X_{h \geq s};\F_2)
\]
is non-trivial for $t \defeq a+R$.
By choosing $k \in \Z$ to be the smallest integer with $r+a \cdot k \geq s$, we obtain the inequalities
\[
s \leq r+a \cdot k \leq r+R+a \cdot k \leq s+t.
\]
Together this gives us the following commutative diagram, where all maps are induced by inclusions:

\begin{center}
\begin{tikzpicture}[scale=2]
\node (st) at (0,0) {$\widetilde{H}_{d-1}(X_{h \geq s+t};\F_2)$};
\node (s) at (2.5,0) {$\widetilde{H}_{d-1}(X_{h \geq s};\F_2)$};
\node (rRak) at (0,-1) {$\widetilde{H}_{d-1}(X_{h \geq r+R+a \cdot k};\F_2)$};
\node (rak) at (2.5,-1) {$\widetilde{H}_{d-1}(X_{h \geq r+a \cdot k};\F_2)$};
\draw (st) edge[->] node[above] {} (s);
\draw (rRak) edge[->, left] node[below] {} (rak);
\draw (st) edge[->] node[above] {} (rRak);
\draw (s) edge[<-,left] node[below] {} (rak);
\end{tikzpicture}
\end{center}
But this is a contradiction since the morphism at the bottom is trivial, whereas the morphism at the top is not.
\end{proof}

\section{Convex functions on $\CAT(0)$-spaces}\label{sec:convex-on-cat0}

Our goal in this short section is to find mild conditions under which convex functions on $\CAT(0)$-spaces are continuous.

\begin{definition}\label{def:convex-function}
Let $(X,d)$ be a $\CAT(0)$-space.
A function $f \colon X \rightarrow \R$ is called \emph{convex} if for every two distinct points $a,b \in X$ the inequality
\[
f(x) \leq \frac{d(x,a)}{d(a,b)}f(b)+\frac{d(x,b)}{d(a,b)}f(a)
\]
holds for every $x \in [a,b]$.
\end{definition}

\begin{lemma}\label{lem:convexity-on-a-ray}
Let $(X,d)$ be a $\CAT(0)$-space, let $a,b \in X$ be two distinct points, and let $x,y \in [a,b]$.
Suppose that $x=[a,b](t)$ and $y=[a,b](t')$ for some real numbers $t,t'$ with $t < t'$.
Then every convex function $f \colon X \rightarrow \R$ satisfies
\[
\frac{f(x)-f(a)}{d(x,a)} \leq \frac{f(y)-f(x)}{d(y,x)} \leq \frac{f(b)-f(x)}{d(b,x)}.
\]
\end{lemma}
\begin{proof}
Since $y$ lies on the geodesic segment $[x,b]$, we may apply the convexity of $f$ to deduce that
\[
f(y) \leq \frac{d(y,x)}{d(x,b)}f(b)+\frac{d(y,b)}{d(x,b)}f(x).
\]
This gives us
{
\setlength{\jot}{5pt}
\begin{align*}
f(y)-f(x) & \leq \frac{d(y,x)}{d(x,b)}f(b)+\frac{d(y,b)}{d(x,b)}f(x)-f(x)\\ 
& = \frac{d(y,x)}{d(x,b)}f(b) + \frac{d(y,b)-d(x,b)}{d(x,b)} f(x)\\
& = \frac{d(y,x)}{d(x,b)}f(b) - \frac{d(y,x)}{d(x,b)} f(x)\\ 
& = \frac{d(y,x)}{d(x,b)}(f(b)-f(x)).
\end{align*}
}
We therefore obtain the second inequality
\[
\frac{f(y)-f(x)}{d(y,x)} \leq \frac{f(b)-f(x)}{d(b,x)}.
\]
For the first inequality, we note that $x$ lies on the geodesic segment $[a,y]$ and thus another application of the convexity of $f$ gives us
\[
f(x) \leq \frac{d(x,a)}{d(a,y)}f(y)+\frac{d(x,y)}{d(a,y)}f(a).
\]
By rearranging this inequality we see that
\[
f(y) \geq f(x) \cdot \frac{d(a,y)}{d(x,a)}-\frac{d(x,y)}{d(x,a)} f(a).
\]
Substracting $f(x)$ on both sides gives us
\begin{align*}
f(y)-f(x) & \geq f(x) \cdot \frac{d(a,y)}{d(x,a)} - \frac{d(x,y)}{d(x,a)} f(a)-f(x) \\
&= f(x) \cdot \frac{d(a,y)-d(x,a)}{d(x,a)} - \frac{d(x,y)}{d(x,a)} f(a) \\
&= f(x) \cdot \frac{d(x,y)}{d(x,a)} - f(a) \frac{d(x,y)}{d(x,a)}.
\end{align*}
This implies
\[
\frac{f(y)-f(x)}{d(x,y)} \geq \frac{f(x)-f(a)}{d(x,a)},
\]
which proves the claim.
\end{proof}

In general there is no need for a convex function on a $\CAT(0)$-space to be continuous.
For example, it is easy to define linear (and hence convex) functions on infinite-dimensional topological vector spaces that are not continuous.
The following definition aims to exclude such examples.

\begin{definition}\label{def:locally-bounded}
Let $X$ be a topological space.
A (not necessarily continuous) function $f \colon X \rightarrow \R$ is called \emph{locally bounded above}, if every point in $X$ admits a neighborhood $U$ such that $f(U) \subseteq \R$ is bounded above.
\end{definition}

Further types of discontinuous, convex functions on a $\CAT(0)$-space arise if the space has a ``boundary''.
Note for example that the function $f \colon [0,1] \rightarrow [0,1]$ that maps $1$ to $1$ and is constantly $0$ elsewhere is convex.
In order to exclude such behavior, we introduce the following property of geodesic metric spaces.

\begin{definition}\label{def:unif-local-extendable}
A geodesic metric space $(X,d)$ is \emph{locally uniformly extendible} if for every $x \in X$ there is some $\delta > 0$ such that the following holds.
For every $y \in X$ there is a geodesic segment $[a,b]$ that contains a segment $[x,y]$ such that $d(a,x)$, $d(b,y) \geq \delta$.
In this case $\delta$ will be called an \emph{extendibility constant} of $x$ in $X$.
\end{definition}

\begin{remark}\label{rem:smaller-extendibility}
Note that if $\delta$ is an extendibility constant of $x$ in $X$ then so is every number in the interval $(0,\delta]$.
\end{remark}

It turns out that being locally bounded above for a convex function on a locally uniformly extendible $\CAT(0)$-space is already enough to guarantee that the function is continuous.

\begin{proposition}\label{prop:equiv-cont-and-loc-bound}
Let $(X,d)$ be a locally uniformly extendible $\CAT(0)$-space.
A convex function $f \colon X \rightarrow \R$ is continuous if and only if it is locally bounded above.
\end{proposition}
\begin{proof}
It is clear that continuous functions are locally bounded.
Thus let us assume that $f$ is locally bounded above and let $x \in X$ be an arbitrary point.
Then there are constants $\varepsilon > 0$ and $c \in \R$ such that $f(y) \leq c$ for every $y \in B_{\varepsilon}(x)$.
By Remark~\ref{rem:smaller-extendibility} we can choose an extendibility constant $\delta \in (0,\frac{\varepsilon}{2})$ for $x \in X$.
Let $y \in B_{\frac{\varepsilon}{2}}(x)$ be a point with $y \neq x$.
Then we can choose two points $a,b \in B_{\varepsilon}(x)$ such that $[x,y] \subseteq [a,b]$ and $d(a,x) = d(b,y) = \delta$.
An application of Lemma~\ref{lem:convexity-on-a-ray} gives us
\[
\frac{f(x)-c}{\delta} \leq \frac{f(x)-f(a)}{d(x,a)} \leq \frac{f(y)-f(x)}{d(y,x)} \leq \frac{f(b)-f(x)}{d(b,x)} \leq \frac{c-f(x)}{\delta}.
\]
Note that $c_1 \defeq \frac{f(x)-c}{\delta}$ and $c_2 \defeq \frac{c-f(x)}{\delta}$ do not depend on $y$.
As a consequence, it follows from the above inequality that
\[
d(y,x) \cdot c_1 \leq f(y)-f(x) \leq d(x,y) \cdot c_2,
\]
for every $y \in B_{\frac{\varepsilon}{2}}(x)$, which shows that $f$ is continuous in $x$.
\end{proof}

The following application of Proposition~\ref{prop:equiv-cont-and-loc-bound} will be used in the next section.

\begin{corollary}\label{cor:convex-contin-on-Rn}
Every convex function $f$ on a locally compact Euclidean building $X$ is continuous.
In particular, convex functions on Euclidean vector spaces are continuous.
\end{corollary}
\begin{proof}
Since $(X,d)$ is a locally uniformly extendible $\CAT(0)$-space, Proposition~\ref{prop:equiv-cont-and-loc-bound} tells us that it is sufficient to show that $f$ is locally bounded above.
Let $x \in X$ be a point.
Since $X$ is locally compact there is a compact neighborhood $U$ of $x$ such that $U$ is covered by finitely many apartments $\{ \Sigma_1, \ldots, \Sigma_k \}$.
For every index $i \in \{1, \ldots, k\}$ let $U_i \defeq \Sigma_i \cap U$.
Since $U_i$ is a bounded subset of the Euclidean space $\Sigma_i$ we can find a finite set of points $\mathcal{V}_i \subseteq \Sigma_i$ such that $U_i$ lies in the convex hull of $\mathcal{V}_i$.
Thus every point in $U_i$ can be written as a convex combination of the points in $\mathcal{V}_i$.
From this it follows that the restriction of $f$ to $U_i$ is bounded by $c_i \defeq \max \limits_{v \in \mathcal{V}_i} f(v)$.
Now the claim follows since the restriction of $f$ to $U$ is bounded above by $\max \limits_{i \in \{1, \ldots, k\}} c_i$.
\end{proof}

\section{Parabolic buildings}\label{sec:parabolic}

Let $X$ be a Euclidean building of dimension $d+1$.
In Section~\ref{sec:pos-dir} and~\ref{sec:neg-dir} we considered height functions $h \in X^{\ast}_{\sigma,v}$ for some chamber $\sigma \subseteq \pinf X$ and some special vertex $v \in X$. We showed that, under certain conditions, the system of superleverlsets $(X_{h \geq r})_{r \in \R}$ is essentially $(d-1)$-connected but not essentially $d$-acyclic.
One of these conditions was about $h$.
We restricted ourselves to the case where $h \circ [x,\xi)$ is strictly decreasing for every $\xi \in \overline{\sigma}$ and every $x \in X$.
In this section we relax this condition by allowing $h \circ [x,\xi)$ to be constant for certain $\xi \in \overline{\sigma}$.
This will reduce the essential connectivity of $(X_{h \geq r})_{r \in \R}$ by the number of vertices $\xi \in \sigma^{(0)}$ for which $h \circ [x,\xi)$ is constant.
To do so, we introduce an auxiliary Euclidean building $X^{\tau}$, where $\tau$ a simplex in $\pinf X$.
To construct $X^{\tau}$, we identify points in $X$ that lie on a common geodesic ray $[x,\xi)$, where $\xi$ is a point in $\overline{\tau}$.
To simplify some arguments, we will introduce another Euclidean building $X_{\tau}$ that is isomorphic to $X^{\tau}$ and appears as a convex subspace of $X$.
Most of the time we restrict ourselves to the case where $\tau$ is a vertex and obtain the general case by induction.
In a more general context, this inductive procedure was introduced by Caprace~\cite[Section~4.1]{Caprace09} in order to study amenable groups that act on $\CAT(0)$-spaces.
From a more algebraic point of view, $X^{\tau}$ was constructed by Landvogt~\cite{Landvogt96}.
Still it seems that a geometric description of the building structure of $X^{\tau}$ that fits our needs did not appear in the literature so far.
We therefore take the opportunity to include it here.

\begin{definition}\label{def:parabolic-building}
Let $\xi \in \pinf X$ be a vertex.
Consider the set $\hat{X}^\xi$ of geodesic rays $[x,\xi) \subseteq X$ with the pseudo-distance
\[
d([x,\xi),[y,\xi)) = \inf \{d(x',y') \mid x' \in [x,\xi), y' \in [y,\xi)\}.
\]
The \emph{parabolic building $X^\xi$} associated to $\xi$ is the metric space obtained from $\hat{X}^\xi$ by identifying points of distance zero.
\end{definition}

\begin{definition}\label{def:levi-building}
Let $\xi,\xi' \in \pinf X$ be opposite vertices.
The \emph{Levi building} $X^{\xi,\xi'}$ associated to $\xi$ and $\xi'$ is the set of geodesic lines connecting $\xi$ and $\xi'$, equipped with the distance
\[
d(\ell,m) = \inf \{d(x',y') \mid x' \in \ell, y' \in m\}.
\]
The \emph{extended Levi building} $\overline{X}^{\xi,\xi'}$ is defined as the union of all geodesic lines in $X^{\xi,\xi'}$.
\end{definition}

Note that $X^{\xi,\xi'}$, being a subspace of $X$, has a natural distance given by the restriction of the distance in $X$.
The use of the words \emph{parabolic} and \emph{Levi} comes from the case where $X$ is the Bruhat--Tits building associated to a Chavalley group $\mathcal{G}$.
In this case $X^\xi$ is the Bruhat--Tits building of the parabolic subgroup of $\mathcal{G}$ corresponding to $\xi$ and $X^{\xi,\xi'}$ is the Bruhat--Tits building of the Levi subgroup of $\mathcal{G}$ corresponding to $\xi$ and $\xi'$.
Our first goal in this section is to show that $X^\xi$, $X^{\xi,\xi'}$, and $\overline{X}^{\xi,\xi'}$ are indeed Euclidean buildings.
Note that these spaces are related by the following commutative diagram:

\begin{center}
\begin{tikzpicture}[scale=2]
\node (extlevi) at (0,0) {$\overline{X}^{\xi,\xi'}$};
\node (levi) at (1,0) {$X^{\xi,\xi'}$};
\node (full) at (0,-1) {$X$};
\node (parabolic) at (1,-1) {$X^\xi$};
\draw (full) edge[->>] node[above] {$p$} (parabolic);
\draw (extlevi) edge[->>] node[above] {$q$} (levi);
\draw (extlevi) edge[left hook->] node[left] {$i$} (full);
\draw (levi) edge[->] node[left] {} node[right] {$j$} (parabolic);
\end{tikzpicture}
\end{center}

The maps $p$ and $q$ are the canonical quotient maps.
The map $i$ denotes the inclusion and $j$ is the map that takes a biinfinite line that is parametrized by a geodesic $c \colon \R \rightarrow X$ towards $\xi$ and maps it to the class in $X^{\xi}$ that is represented by the ray $c_{\vert [0,\xi)}$.
We will show that $p$ and $q$ admit continuous sections.
Moreover we will see that $j$ is an isomorphism of Euclidean buildings and that $\overline{X}^{\xi,\xi'}$ is a strong deformation retract of $X$.

\subsection{Apartments in the parabolic building}\label{subsec:ap-in-par}
Let us fix an apartment $\SigmaStd$ in $X$, a pair of opposite vertices $\xi,\xi' \in \pinf \SigmaStd$, and a chamber $\sigma \subseteq \pinf \SigmaStd$ that has $\xi$ as a vertex.
Further we fix a special vertex $\vstd \in \SigmaStd$, which will allow us to view $\SigmaStd$ as a vector space with origin $\vstd$.
The full apartment system of $X$ will be denoted by $\mathcal{A}$.
The following types of apartments will be important for us.

\begin{definition}\label{def:horizontal-ap}
An apartment $\Sigma \in \mathcal{A}$ is called \emph{horizontal} if it contains the opposite rays $[x,\xi)$ and $[x,\xi')$ for some (and hence every) point $x \in \Sigma$.
The set of horizontal apartments of $X$ will be denoted by $\mathcal{A}_{\hor}$.
Similarly, we say that a wall $H \subseteq X$ is \emph{horizontal} if $[x,\xi)$ and $[x,\xi')$ are contained in $H$ for some (and hence every) $x \in H$.
\end{definition}

Note that our fixed apartment $\SigmaStd$ is horizontal.

\begin{notation}\label{def:parab-walls}
Let $\mathcal{H}$ denote the set of walls in $\SigmaStd$ and let $\mathcal{H}_{\hor} \subseteq \mathcal{H}$ be the subset of horizontal walls.
Let further $\mathcal{H}(\vstd) \subseteq \mathcal{H}$ denote the set of walls that contain $\vstd$ and let $\mathcal{H}_{\hor}(\vstd) = \mathcal{H}_{\hor} \cap \mathcal{H}(\vstd)$.
\end{notation}

Recall that we have an isomorphism $\lk_{\SigmaStd}(\vstd) \rightarrow \pinf \SigmaStd$ given by extending geodesic germs to geodesic rays.
In particular, the ray $[\vstd,\xi)$ contains a vertex $w \in \SigmaStd$ that is incident to $\vstd$ via an edge $e$.
From the choice of $\sigma$ we see that $E \defeq \pr_{\vstd}(\sigma)$ is a chamber in $\st_{\SigmaStd}(e)$.
The Euclidean Coxeter system induced by $E$ will be denoted by $(W,S)$.
Let further $(W_{\vstd},S_{\vstd})$ be the corresponding spherical Coxeter system induced by $E$.
The theory of Euclidean Coxeter groups provides us with a decomposition $W = W_{\vstd} \ltimes L$, where $L$ denotes the group of translations in $W$ (see for example~\cite[Proposition 10.17]{AbramenkoBrown08}).

\begin{notation}\label{not:cox-grp-for-lvl-0}
The subgroup of $W$ that is generated by the reflections $s_{H}$ with $H \in \mathcal{H}_{\hor}$ will be denoted by $W^{\xi}$.
\end{notation}

Note that the action of $W^{\xi}$ on $\SigmaStd$ is not cocompact.
To make up for that, we will consider a natural Euclidean Coxeter complex $\SigmaStd_{\xi}$ for which $W^{\xi}$ is the associated Coxeter group.
Let $\beta \colon X \rightarrow \R$ denote the Busemann function corresponding to $\xi$ that is centered at $\vstd$.
We will frequently use the fact that $\SigmaStd \cap \beta^{-1}(0)$ is a hyperplane in $\SigmaStd$ that is orthogonal to $[\vstd,\xi)$ (see~\cite[II.8.24.(1)]{BridsonHaefliger99}).

\begin{lemma}\label{lem:restriction-is-well-def}
The group $W^{\xi}$ fixes $\xi$ and $\xi'$.
Moreover $W^{\xi}$ stabilizes the sets $\mathcal{H}_{\hor}$ and $\SigmaStd \cap \beta^{-1}(0)$.
\end{lemma}
\begin{proof}
Let $H \in \mathcal{H}_{\hor}$ and let $s_{H}$ be the reflection at $H$.
Since $H$ is horizontal, it follows that $H$ contains $[x,\xi)$ and $[x,\xi')$ for every $x \in H$.
By definition, $s_{H}$ fixes $H$.
In particular, $s_H$ fixes $[x,\xi)$ and $[x,\xi')$ for every $x \in H$, which proves the first claim.
Since $s_H$ is an isometry that fixes $[x,\xi)$ pointwise, it has to stabilize the orthogonal complement $\SigmaStd \cap \beta^{-1}(0)$ of $[x,\xi)$ in $\SigmaStd$.
It remains to verify that $s_H$ stabilizes $\mathcal{H}_{\hor}$.
But this directly follows from the fact that $s_H$ fixes $\xi$ and $\xi'$ and, being an isometry, maps hyperplanes to hyperplanes.
\end{proof}

\begin{notation}\label{not:alphas}
Let $\mathcal{B} = \{\alpha_0,\ldots,\alpha_d\}$ be a set of linear forms in $\SigmaStd^{\ast}$ such that
\[
\overline{K_{\vstd}(\sigma)} = \Set{x \in \SigmaStd}{\alpha_k(x) \geq 0 \text{ for every } 0 \leq k \leq d}.
\]
Suppose that $\alpha_{0}^{-1}(0)$ is the unique non-horizontal wall of $\overline{K_{\vstd}(\sigma)}$.
\end{notation}

Clearly $\mathcal{B}$ is a basis of $\SigmaStd^{\ast}$ whose elements vanish on the panels of $E$ that contain $\vstd$.
It will be important for us to note that $\mathcal{B}$ remains a basis if we replace $\alpha_0$ by the restriction of $\beta$ to $\SigmaStd$.
For convenience, we will denote this restriction by $\beta$ as well.

\begin{lemma}\label{lem:basis-of-dual-sigma}
The set $\{\beta,\alpha_1,\ldots,\alpha_n\}$ forms a basis of $\SigmaStd^{\ast}$.
\end{lemma}
\begin{proof}
Note that the walls that bound $\overline{K_{\vstd}(\sigma)}$ are the kernels of the maps $\alpha_k \in \mathcal{B}$.
Since $\alpha_0$ is the unique linear form in $\mathcal{B}$ whose kernel is non-horizontal it follows that $\xi \in \pinf (\bigcap \limits_{k=1}^{d} \ker(\alpha_k))$.
On the other hand, basic linear algebra tells us that $L \defeq \bigcap \limits_{k=1}^{d} \ker(\alpha_k)$ is a one dimensional linear subspace of $\SigmaStd$ and therefore consists of the linear span of $[\vstd,\xi)$.
Since $\SigmaStd \cap \beta^{-1}(0)$ is a hyperplane in $\SigmaStd$ that is orthogonal to $[\vstd,\xi)$, we obtain
\[
\ker(\beta) \cap \bigcap \limits_{i=1}^{d} \ker(\alpha_i)  = \{\vstd\}.
\]
From this is follows that $\{\beta,\alpha_1,\ldots,\alpha_d\}$ forms a basis of $\SigmaStd^{\ast}$.
\end{proof}

\begin{notation}\label{not:zero-levels}
Let $\SigmaStd_{\xi} = \SigmaStd \cap \beta^{-1}(0)$ and let $\mathcal{H}_{\xi} = \Set{H \cap \SigmaStd_{\xi}}{H \in \mathcal{H}_{\hor}}$.
\end{notation}

Lemma~\ref{lem:basis-of-dual-sigma} particularly implies that $\mathcal{H}_{\xi}$ is a set of hyperplanes in $\SigmaStd$.
Note that by intersecting horizontal apartments with $\SigmaStd_{\xi}$, we get a bijection
$\mathcal{H}_{\hor} \rightarrow \mathcal{H}_{\xi}                                                                                                                                                                     $.

\begin{proposition}\label{prop:parab-app-is-an-app}
The hyperplane arrangement $\mathcal{H}_{\xi}$ turns $\SigmaStd_{\xi}$ into a Euclidean Coxeter complex.
Its Coxeter group, denoted by $W_{\xi}$, is canonically isomorphic to $W^{\xi}$ via the restriction map $\phi \colon W^{\xi} \rightarrow W_{\xi},\ f \mapsto f_{\vert \SigmaStd_{\xi}}$.
\end{proposition}
\begin{proof}
The local finiteness of $\mathcal{H}_{\xi}$ follows directly from the local finiteness of $\mathcal{H}$.
Let $C$ be a connected component in $\SigmaStd_{\xi} \setminus \bigcup_{H \in \mathcal{H}_{\xi}} H$.
By definition we have $\ker(\alpha_k) \cap \SigmaStd_{\xi} \in \mathcal{H}_{\xi}$ for every $1 \leq k \leq d$.
It therefore follows that $C$ lies between two parallel walls of each $\ker(\alpha_k)$.
Thus $\alpha_k(C) \subseteq \R$ is bounded for every $1 \leq k \leq d$.
By construction we also have $\beta(C) = 0$.
Since $\{\beta,\alpha_1,\ldots,\alpha_n\}$ is a basis of $\SigmaStd^{\ast}$ by Lemma~\ref{lem:basis-of-dual-sigma}, it follows that $C$ is a bounded subset of $\SigmaStd_{\xi}$.
It remains to show that $\mathcal{H}_{\xi}$ is stable under the action of $W_{\xi}$.
From Lemma~\ref{lem:restriction-is-well-def} we know that $W^{\xi}$ stabilizes $\SigmaStd_{\xi}$ and $\mathcal{H}_{\xi}$.
Let $s_H \in W^{\xi}$ be the reflection at a horizontal hyperplanes $H$.
Then $s_H$ restricts to a reflection at $H \cap \SigmaStd_{\xi} \in \mathcal{H}_{\xi}$ in $\SigmaStd_{\xi}$.
Since $W^{\xi}$ is generated by the reflections at horizontal hyperplanes, it follows that $\phi$ is well-defined and surjective.
Now let $f$ be an isometry of $\SigmaStd$ that fixes $\SigmaStd_{\xi} = \SigmaStd \cap \beta^{-1}(0)$.
Then $f$ is either trivial or the reflection $s$ at $\SigmaStd_{\xi}$.
Since $s$ maps $\xi$ to $\xi'$ and $W^{\xi}$ fixes $\xi$ by Lemma~\ref{lem:restriction-is-well-def}, it follows that $\phi$ is injective.
\end{proof}

\subsection{A subbuilding of X}
Our next goal is to show that the extended Levi building $\overline{X}^{\xi,\xi'}$ is a subbuilding of $X$.

\begin{lemma}\label{lem:existence-of-merging}
For every $x \in X$ there is a number $t \in [0,\infty)$ such that $[x,\xi)(t)$ lies in $\overline{X}^{\xi,\xi'}$.
\end{lemma}
\begin{proof}
From~\cite[Theorem 11.63.(1)]{AbramenkoBrown08} it follows that $x$ is contained in an apartment $\Sigma$ that contains $\sigma$ in its boundary $S \defeq \pinf \Sigma$.
Since $\pinf X$ is a spherical building, Lemma~\ref{lem:existence-aps-sph-build} tells us that there is an apartment $S' \subseteq \pinf X$ that contains $\st_{S}(\xi)$ and $\xi'$.
By~\cite[Theorem 11.79]{AbramenkoBrown08} there is an apartment $\Sigma'$ of $X$ with $\pinf \Sigma' = S'$.
In particular $\Sigma'$ is a horizontal apartment.
Thus it suffices to show that $[x,\xi)(t) \in \Sigma'$ for some $t \geq 0$.
For every chamber $\delta \subseteq \st_{S}(\xi)$ let $x_{\delta} \in \Sigma'$ be such that the sector $K_{x_{\delta}}(\delta)$ is contained in $\Sigma' \cap \Sigma$.
Since $\Sigma'$ is a convex subcomplex of $X$, it follows that the convex hull of the sectors $K_{x_{\delta}}(\delta)$ is contained in $\Sigma'$.
Now the lemma follows from the simple observation that $[x,\xi)$ runs into this convex hull.
\end{proof}

\begin{notation}\label{not:geod-lines-in-X}
For every point $x \in \overline{X}^{\xi,\xi'}$ let $c_x \colon \R \rightarrow \overline{X}^{\xi,\xi'}$ be the unique geodesic that is determined by $c_x(0)=x$, $c_x(-\infty)=\xi'$, and $c_x(\infty)=\xi$.
\end{notation}

The following well-known fact about Euclidean buildings (see~\cite[Theorem 11.53]{AbramenkoBrown08}) will help us to verify the building axiom (B1) for $\overline{X}^{\xi,\xi'}$.

\begin{theorem}\label{thm:subsets-eucl-build}
Let $Y$ be a subset of an $n$-dimensional Euclidean building $X$.
Assume either that $Y$ is convex or that $Y$ has non-empty interior.
If $Y$ is isometric to a subset of $\R^n$, then $Y$ is contained in an apartment of $X$.
\end{theorem}

\begin{proposition}\label{prop:convexity-of-XXi-mu}
The space $\overline{X}^{\xi,\xi'}$ is a subbuilding of $X$.
An apartment system for $\overline{X}^{\xi,\xi'}$ is given by $\mathcal{A}_{\hor}$.
\end{proposition}
\begin{proof}
By definition, every $\Sigma \in \mathcal{A}_{\hor}$ is contained in $\overline{X}^{\xi,\xi'}$.
Since each $\Sigma \in \mathcal{A}_{\hor}$ is an apartment of $X$, we see that the buildings axioms (B0) and (B2) automatically satisfied for $\overline{X}^{\xi,\xi'}$.
To prove (B1), it suffices to show that every two points in $\overline{X}^{\xi,\xi'}$ lie in some horizontal apartment.
Let $x,y \in \overline{X}^{\xi,\xi'}$ be arbitrary points and let $c_x,c_y \colon \R \rightarrow X$ be the corresponding isometric lines.
Since $c_x$ and $c_y$ converge to the same ends at infinity it follows that the function
\[
\R \rightarrow \R,\ t \mapsto d(c_x(t),c_y(t))
\]
is bounded.
In this case, the flat strip theorem (see~\cite[Theorem II.2.13]{BridsonHaefliger99}) tells us that $\conv(c_x(\R),c_y(\R))$ is isometric to the strip $\R \times [0,D]$ for some $D \geq 0$.
In particular, $\conv(c_x(\R),c_y(\R)) \subseteq X$ is a convex subspace that is isometric to a subset of the Euclidean space $\R^{d+1}$.
We may therefore apply Theorem~\ref{thm:subsets-eucl-build} to deduce that there is an apartment $\Sigma$ of $X$ containing $\conv(c_x(\R),c_y(\R))$.
Thus $\Sigma$ is a horizontal apartment that contains $x$ and $y$, which proves the first claim.
\end{proof}

\begin{lemma}\label{lem:thickness}
Let $H \subseteq \overline{X}^{\xi,\xi'}$ be a wall and let $P \subseteq H$ be a panel.
\begin{enumerate}
\item If $H$ is horizontal, then $\st_{\overline{X}^{\xi,\xi'}}(P) = \st_{X}(P)$.
\item If $H$ is non-horizontal, then there are exactly two chambers in $\overline{X}^{\xi,\xi'}$ that are incident to $P$.
\end{enumerate}
\end{lemma}
\begin{proof}
Suppose that $H$ is horizontal.
From Proposition~\ref{prop:convexity-of-XXi-mu} we know that $\overline{X}^{\xi,\xi'}$ is a building.
By assumption we have $\sigma \subseteq \pinf \SigmaStd \subseteq \pinf \overline{X}^{\xi,\xi'}$.
Thus it follows from~\cite[Theorem 11.63.(1)]{AbramenkoBrown08} that there is an apartment $\Sigma$ in $\overline{X}^{\xi,\xi'}$ with $P \subseteq \Sigma$ and $\sigma \subseteq \pinf \Sigma$.
Let $R \subseteq \Sigma$ be the closed half space that is detemined by $P \subseteq \partial R$ and $\sigma \subseteq \pinf R$.
Since $P$ was assumed to be a panel in the horizontal wall $H$, it follows that $\partial R$ is a horizontal wall as well.
Let us now consider a chamber $D$ in $\st_X(P)$ and let $Z = R \cup D$.
Since $\sigma \subseteq \pinf \Sigma$, the retraction $\rho_{\sigma, \SigmaStd}$ restricts to an isometric isomorphism $Z \rightarrow \rho_{\sigma, \SigmaStd}(Z)$.
Thus Theorem~\ref{thm:subsets-eucl-build} provides us with an apartment $\Sigma' \subseteq X$ that contains $Z$.
In particular, $\Sigma'$ contains $\pinf R$ and is therefore horizontal and lies in $\overline{X}^{\xi,\xi'}$.

Suppose now that $H$ is non-horizontal.
Let $p \in P$ and let $C$, respectively $D$, be the chamber in $\overline{X}^{\xi,\xi'}$ that contain an initial segments of $(p,\xi)$, respectively $(p,\xi')$.
From the definition a horizontal apartment it directly follows that $C$ and $D$ are contained in every horizontal apartment that contains $p$.
On the other hand, Proposition~\ref{prop:convexity-of-XXi-mu} implies that every chamber in $\st_{\overline{X}^{\xi,\xi'}}(P)$ is contained in a horizontal apartment, which necessarily contains $p$.
Thus we see that $C$ and $D$ are the only chambers in $\st_{\overline{X}^{\xi,\xi'}}(P)$.
\end{proof}

\begin{corollary}\label{cor:all-apps-are-levi}
The full apartment system of $\overline{X}^{\xi,\xi'}$ is given by $\mathcal{A}_{\hor}$.
\end{corollary}
\begin{proof}
In view of Proposition~\ref{prop:convexity-of-XXi-mu}, it remains to show that every apartment in $\overline{X}^{\xi,\xi'}$ is horizontal.
To see this, let us consider an arbitrary apartment $\Sigma \subseteq \overline{X}^{\xi,\xi'}$ and a point $x \in \Sigma$ that is contained in some (open) chamber $C \subseteq \Sigma$.
We have to show that the rays $[x,\xi)$ and $[x,\xi')$ are contained in $\Sigma$.
Let $\Sigma'$ be a horizontal apartment that contains $C$.
Then $\Sigma \cap \Sigma'$ is a non-empty convex subcomplex of $\Sigma$.
In this case, it follows from~\cite[Proposition 3.94]{AbramenkoBrown08} that $\Sigma \cap \Sigma' = \bigcap \limits_{R \in \mathcal{R}} \overline{R}$, where $\mathcal{R}$ is the set of all closed half spaces in $\Sigma$ that contain $\Sigma \cap \Sigma'$.
Since $\Sigma$ contains only finitely many parallel classes of walls, we can assume
that $\mathcal{R}$ is finite and minimal.
Let $H = \partial R$ for some $R \in \mathcal{R}$.
Since $C \subseteq \Sigma \cap \Sigma'$, it follows that $\Sigma \cap \Sigma'$ is a union of chambers.
From the minimality of $\mathcal{R}$
it therefore follows that some panel $P \subseteq H$ is contained in $\Sigma \cap \Sigma'$.
Let us fix a point $p \in P$.
If $H$ is non-horizontal, then exactly one of the rays $(p,\xi)$ and $(p,\xi')$ has empty intersection with $\Sigma \cap \Sigma'$.
In this case, there are at least three chambers incident to $P$, which contradicts Lemma~\ref{lem:thickness}.
Hence $\Sigma \cap \Sigma'$ is the intersection of horizontal half spaces, i.e.\ half spaces that are bounded by horizontal walls.
Since neither $[x,\xi)$ nor $[x,\xi')$ can leave a horizontal half space that contains $x$, it follows that $\Sigma$ is horizontal, which proves the claim.
\end{proof}

\subsection{The building $X_{\xi}$}

In Subsection~\ref{subsec:ap-in-par} we saw that $\SigmaStd_{\xi} = \SigmaStd \cap \beta^{-1}(0)$ can be naturally endowed with the structure of a Euclidean Coxeter complex.
Our next goal is to endow the space $X_{\xi} \defeq \overline{X}^{\xi,\xi'} \cap \beta^{-1}(0)$ with the structure of a Euclidean building in which $\SigmaStd_{\xi}$ is an apartment.
To this end, we consider the set $\mathcal{H}_{\hor}(X)$ of horizontal walls in $X$ or, which is equivalent, the set of horizontal walls in $\overline{X}^{\xi,\xi'}$.
Let further $\Lambda$ denote the union of all horizontal walls.
In order to define a building structure on $X_{\xi}$ we first have to introduce a cell structure on $X_{\xi}$.
This cell structure will be induced by the connected components in $\overline{X}^{\xi,\xi'} \backslash \Lambda$.

\begin{lemma}\label{lem:the-complement-of-hor-walls}
For every connected component $C$ in $\overline{X}^{\xi,\xi'} \backslash \Lambda$, there is some apartment $\Sigma \in \mathcal{A}_{\hor}$ such that $C$ is a connected component in~$\Sigma \backslash \Lambda$.
\end{lemma}
\begin{proof}
Let $D \subseteq \overline{X}^{\xi,\xi'}$ be a chamber that is contained in $C$.
In view of Proposition~\ref{prop:convexity-of-XXi-mu}, there is some $\Sigma \in \mathcal{A}_{\hor}$ that contains $D$.
We have to show that $C$ is entirely contained in $\Sigma$.
To see this, let $F \subseteq C$ be a further chamber of $\overline{X}^{\xi,\xi'}$.
Since $C$ is connected, we can choose a continuous path $\gamma$ that starts in $D$ and ends in $F$.
Now $C$ is open by construction.
Thus for every cell $A$ in $C$ we have $\st_{\overline{X}^{\xi,\xi'}}(A) \subseteq C$.
In particular, the star of every cell that is crossed by $\gamma$ is contained in $C$.
We can therefore choose a gallery $\Gamma$ from $D$ to $F$ that stays in $C$.
Let $P$ be a panel that separates two adjacent chambers in $\Gamma$ and let $H$ be a wall in $\overline{X}^{\xi,\xi'}$ that contains $P$.
Then $H$ is non-horizontal since otherwise one of these two chambers would lie outside of $C$.
An application of Lemma~\ref{lem:thickness} therefore implies that $\st_{\overline{X}^{\xi,\xi'}}(P) \subseteq \Sigma$ whenever $P \subseteq \Sigma$.
This shows that $\Gamma$, and hence the arbitrarily chosen chamber $D \subseteq C$, lies in $\Sigma$, which proves the claim.
\end{proof}

\begin{lemma}\label{lem:eucl-str-on-app-of-beta0}
Given $\widetilde{\Sigma} \in \mathcal{A}_{\hor}$, the intersection $\Sigma = \widetilde{\Sigma} \cap X_{\xi}$ is a hyperplane in $\widetilde{\Sigma}$.
Moreover $\Sigma$ can be endowed with the structure of a Euclidean Coxeter complex whose chambers are given by the connected components in $\Sigma \backslash \Lambda$.
\end{lemma}
\begin{proof}
By Proposition~\ref{prop:convexity-of-XXi-mu} we know that $\overline{X}^{\xi,\xi'}$ is a Euclidean building.
In this case, it follows from~\cite[Theorem 11.63.(1)]{AbramenkoBrown08} that we can choose a horizontal apartment $\Sigma'$ that satisfies $\sigma \subseteq \partial \Sigma'$ and $\Sigma' \cap \widetilde{\Sigma} \neq \emptyset$.
Since $\sigma$ was chosen to be a chamber in $\partial \SigmaStd$, we also have $\Sigma' \cap \SigmaStd \neq \emptyset$.
Let $x \in \Sigma' \cap \widetilde{\Sigma}$ and $y \in \Sigma' \cap \SigmaStd$ be arbitrary points.
By definition, $\widetilde{\Sigma}$, $\Sigma'$ and $\SigmaStd$ are all horizontal.
Thus we see that the geodesic lines $\ell_x,\ell_y$ from $\xi'$ to $\xi$ that contain $x$, respectively $y$, are contained in $\Sigma' \cap \widetilde{\Sigma}$, respectively $\Sigma' \cap \SigmaStd$.
From the building axiom (B2) we see that there are isomorphisms $\varphi_1 \colon \Sigma' \rightarrow \widetilde{\Sigma}$ and $\varphi_2 \colon \SigmaStd \rightarrow \Sigma'$ that fix $\ell_x$, respectively $\ell_y$.
This implies that $\varphi_1$ and $\varphi_2$ map horizontal hyperplanes to horizontal hyperplanes.
Moreover, it follows that the composition $\varphi_1 \circ \varphi_2$ restricts to an isometry
\[
\varphi \colon
\SigmaStd_{\xi}
= \SigmaStd \cap \beta^{-1}(0)
\rightarrow \widetilde{\Sigma} \cap \beta^{-1}(0)
= \Sigma
\]
that maps each connected component in $\SigmaStd_{\xi} \setminus \Lambda$ to a connected component in $\Sigma \setminus \Lambda$.
It therefore remains to prove the lemma for $\SigmaStd_{\xi}$.
To this end, we note that for every $H \in \mathcal{H}_{\hor}(X)$, there is a horizontal wall $H' \subseteq \SigmaStd$ with $H \cap \SigmaStd \subseteq H' \cap \SigmaStd$.
Thus we see that $\SigmaStd_{\xi} \setminus \Lambda$ coincides with $\SigmaStd_{\xi} \setminus \Lambda'$, where $\Lambda'$ denotes the union of horizontal hyperplanes in $\SigmaStd$.
Now it remains to recall from Proposition~\ref{prop:parab-app-is-an-app} that $\SigmaStd_{\xi}$ is a Euclidean Coxeter complex whose chambers are given by the connected components in $\SigmaStd_{\xi} \setminus \Lambda'$.
\end{proof}

In view of Lemma~\ref{lem:eucl-str-on-app-of-beta0}, we see that $X_{\xi}$ is a cell complex that is covered by Euclidean Coxeter complexes in $\mathcal{A}_{\xi} \defeq \Set{\Sigma \cap X_{\xi}}{\Sigma \in \mathcal{A}_{\hor}}$.
In fact it will turn out that $\mathcal{A}_{\xi}$ is the complete apartment system for $X_{\xi}$.

\begin{proposition}\label{prop:beta-0-is-a-building}
The space $X_{\xi}$ is a Euclidean building.
\end{proposition}
\begin{proof}
From Proposition~\ref{prop:convexity-of-XXi-mu} it follows that $\overline{X}^{\xi,\xi'}$ is covered by horizontal apartments.
As a consequence, $X_{\xi}$ is covered by the complexes in $\mathcal{A}_{\xi}$.
In view of Lemma~\ref{lem:eucl-str-on-app-of-beta0}, each complex in $\mathcal{A}_{\xi}$ is a Euclidean Coxeter complex.
Hence (B0) is satisfied.
In order to check (B1), let $A,B \subseteq X_{\xi}$ be cells and let $a \in A$, $b \in B$.
Using the corresponding axiom for $\overline{X}^{\xi,\xi'}$, we obtain a horizontal apartment $\widetilde{\Sigma}$ that contains $a$ and $b$.
Thus we get $a,b \in \Sigma \defeq \widetilde{\Sigma} \cap \beta^{-1}(0) \in \mathcal{A}_{\xi}$.
Since $\Sigma$ is a subcomplex of $X_{\xi}$ it follows that $A$ and $B$ are contained in $\Sigma$.
To check (B2), let $\Sigma_1,\Sigma_2$ be two apartments in $\mathcal{A}_{\xi}$ and let $\widetilde{\Sigma}_1,\widetilde{\Sigma}_2$ be the corresponding horizontal apartments.
Suppose that $\Sigma_1 \cap \Sigma_2$ contains a chamber $c$ of $\Sigma_1$ and let $C \subseteq \widetilde{\Sigma}_1 \backslash \Lambda$ be the component containing $c$.
From Lemma~\ref{lem:the-complement-of-hor-walls} it follows that $\widetilde{\Sigma}_1$ and $\widetilde{\Sigma}_2$ contain $C$.
By~\cite[Remark 4.5]{AbramenkoBrown08} there is an isomorphism $f \colon \widetilde{\Sigma}_1 \rightarrow \widetilde{\Sigma}_2$ that fixes $\widetilde{\Sigma}_1 \cap \widetilde{\Sigma}_2$ pointwise.
As a consequence, $f$ restricts to an isomorphism that fixes $\Sigma_1 \cap \Sigma_2$ pointwise.
In view of~\cite[Remark 4.4]{AbramenkoBrown08}, it follows that $X_{\xi}$ is a Euclidean building.
\end{proof}

\begin{lemma}\label{lem:tranlating-auto}
Suppose that there is some $\alpha \in \Aut(X)$ such that $\alpha(\sigma) = \sigma$, $\alpha(\SigmaStd) = \SigmaStd$, and $h(\alpha(x)) = h(x) + a$ for some $a \neq 0$ and every $x \in X$.
Then there is some $\alpha_{\xi} \in \Aut(X_{\xi})$ such that $\alpha_{\xi}(\sigma_{\xi}) = \sigma_{\xi}$, $\alpha_{\xi}(\SigmaStd_{\xi}) = \SigmaStd_{\xi}$, and $h_{\xi}(\alpha_{\xi}(x)) = h_{\xi}(x) + a$ for every $x \in X_{\xi}$.
\end{lemma}
\begin{proof}
Since $\alpha(\sigma) = \sigma$ and $\alpha(\SigmaStd) = \SigmaStd$, it follows that $\alpha$ restricts to a translation on $\SigmaStd$.
It therefore follows that $\alpha(\xi) = \xi$ and $\alpha(\xi') = \xi'$.
Since $\alpha$ is an isometry, we can deduce that $\alpha(\overline{X}^{\xi,\xi'}) = \overline{X}^{\xi,\xi'}$ and
\[
\alpha(X_{\xi})
= \alpha(\overline{X}^{\xi,\xi'} \cap \beta^{-1}(0))
= \overline{X}^{\xi,\xi'} \cap \beta^{-1}(t)
\]
for certain $t \in \R$.
Recall that for each $x \in X_{\xi}$ we write $c_x$ to denote the biinfinite geodesic from $\xi'$ to $\xi$ that is determined by $c_x(0) = x$.
By definition, $\overline{X}^{\xi,\xi'}$ is covered by the images of the $c_x$ so that we can define the isometry
\[
f \colon \overline{X}^{\xi,\xi'} \cap \beta^{-1}(t) \rightarrow X_{\xi},\ c_x(t) = x.
\]
Then $\alpha_{\xi} \defeq (f \circ \alpha)_{\vert X_{\xi}} \in \Aut(X_{\xi})$ restricts to a translation on $\SigmaStd_{\xi}$, which implies $\alpha_{\xi}(\sigma_{\xi}) = \sigma_{\xi}$ and $\alpha_{\xi}(\SigmaStd_{\xi}) = \SigmaStd_{\xi}$.
Since $\xi$ was chosen such that $h \circ [x,\xi)$ is constant, is moreover follows that
\[
h_{\xi}(\alpha_{\xi}(x)) = h(f(\alpha(x))) = h(\alpha(x)) = h(x) + a = h_{\xi}(x)+a
\]
for every $x \in X_{\xi}$, which completes the proof.
\end{proof}

Let us collect some further facts about $X_{\xi}$.

\begin{lemma}\label{lem:prop-of-XXi}
The building $X_{\xi}$ satisfies the following properties.
\begin{enumerate}
\item The complete apartment system of $X_{\xi}$ is given by $\mathcal{A}_{\xi}$.
\item Given a cell $A \subseteq X_{\xi}$, there is a cell $A' \subseteq X$ with $\lk_{X_{\xi}}(A) \cong \lk_{X}(A')$.
\end{enumerate}
\end{lemma}
\begin{proof}
Let $\Sigma \subseteq X_{\xi}$ be an apartment.
Then $\Sigma$ is convex in $\overline{X}^{\xi,\xi'}$ and we can apply Theorem~\ref{thm:subsets-eucl-build} to deduce that there is an apartment $\widetilde{\Sigma} \subseteq \overline{X}^{\xi,\xi'}$ that contains $\Sigma$.
On the other hand, we know from Corollary~\ref{cor:all-apps-are-levi} that $\widetilde{\Sigma}$ is horizontal.
Thus $\Sigma = \widetilde{\Sigma} \cap X_{\xi} \in \mathcal{A}_{\xi}$, which proves the first claim.
To prove the second claim we observe that $\lk_{X_{\xi}}(A) \cong \lk_{\lk_{X_{\xi}}(v)}(B)$, where $v$ is a vertex of $A$ and $ B$ is an appropriate cell in $\lk_{X_{\xi}}(v)$.
It therefore suffices to consider the case where $A$ is a vertex, which we call $v$.
In this case there is an edge $f \subseteq \overline{X}^{\xi,\xi'}$ with $v \in \overline{f}$ and $[v,\xi)(t) \in f$ for some $t > 0$.
Thus every direction in $\lk_{X_{\xi}}(v)$ is orthogonal to $f$, which gives us $\lk_{X_{\xi}}(v) \cong \lk_{\overline{X}^{\xi,\xi'}}(f)$.
Let $C$ be a chamber in $\overline{X}^{\xi,\xi'}$ that has $f$ as a face.
Then $f$ lies in the intersection of all closed panels of $C$ that lie in horizontal walls.
As a consequence, we deduce from Lemma~\ref{lem:thickness} that $\lk_{\overline{X}^{\xi,\xi'}}(f) \cong \lk_{X}(f)$, which gives us $\lk_{X_{\xi}}(v) \cong \lk_{X}(f)$.
\end{proof}

\begin{remark}\label{rem:X_xi-isom-X^xi}
Recall that the parabolic building $X^{\xi}$ was defined as the space of equivalence classes of rays $[x,\xi)$ in $X$, where two rays $[x,\xi),[y,\xi)$ are identified if their pseudo-distance
\[
d([x,\xi),[y,\xi)) = \inf \{d(x',y') \mid x' \in [x,\xi), y' \in [y,\xi)\}
\]
is zero.
From Lemma~\ref{lem:existence-of-merging} we know that for every $x \in X$ there is some $t \in [0,\infty)$ with $[x,\xi)(t) \in \overline{X}^{\xi,\xi'}$.
Since $X_{\xi} = \overline{X}^{\xi,\xi'} \cap \beta^{-1}(0)$, it therefore follows that $p \colon X \rightarrow X^{\xi},\ x \mapsto [[x,\xi)]$ restricts to an isometry $X_{\xi} \cong X^{\xi}$.
In particular, we see that $X^{\xi}$ is a Euclidean building.
Despite the fact that $X^{\xi}$ has a simpler description than $X_{\xi}$, it will typically be more convenient to work with $X_{\xi}$.
\end{remark}

\section{The geometric main result}\label{sec:the-geom-main}

In this section we prove Theorem~\ref{introtheorem:C} from the introduction in a more general context.
To this end, we fix a Euclidean building $X$, an apartment $\SigmaStd \subseteq X$, a pair of opposite vertices $\xi,\xi' \in \pinf \SigmaStd$, a chamber $\sigma \subseteq \pinf \SigmaStd$ that has $\xi$ as a vertex, and a special vertex $\vstd \in \SigmaStd$.
Recall from Definition~\ref{def:retr-height-functions} that this allows us to view $\SigmaStd$ as a vector space with origin $\vstd$ and to define the space of height functions $X^{\ast}_{\sigma,\vstd} = \Set{\alpha \circ \rho_{\sigma, \SigmaStd}}{\alpha \in \SigmaStd^{\ast}}$.
In Section~\ref{sec:pos-dir} and~\ref{sec:neg-dir} we studied the essential connectivity properties of the system of superlevelsets $(X_{h \geq r})_{r \in \R}$, where $h \in X_{\sigma,\vstd}^{\ast}$ was chosen with $\overline{\sigma} \subseteq \pinf \SigmaStd_{h < 0}$.
In this section we relax this condition by allowing height functions $h \in X_{\sigma,\vstd}^{\ast}$ that satisfies $\overline{\sigma} \subseteq \pinf \SigmaStd_{h \leq 0}$.
Moreover we assume that $\xi \in \pinf \SigmaStd_{h = 0}$, or equivalently, that $h \circ [x,\xi)$ is constant for every $x \in X$.
Our goal is to relate the essential connectivity properties of $(X_{h \geq r})_{r \in \R}$ to the essential connectivity properties of $((X_{\xi})_{h_{\xi} \geq r})_{r \in \R}$, where $h_{\xi}$ is the restriction of $h$ to $X_{\xi}$.

\subsection{A height function for $X_{\xi}$}

Recall from Lemma~\ref{lem:existence-of-merging} that for every $x \in X$ there is some $t \in [0,\infty)$ such that $[x,\xi)(t)$ lies in $\overline{X}^{\xi,\xi'}$.
We may therefore consider the function $T \colon X \rightarrow \R$ that is given by
\begin{equation}\label{eq:def-Tx}
T(x) \defeq \inf \Set{t \in [0,\infty)}{[x,\xi)(T(x)) \in \overline{X}^{\xi,\xi'}}.
\end{equation}
Since $[x,\xi)$ is continuous and $\overline{X}^{\xi,\xi'}$ is closed in $X$, it follows that the infimum in~\eqref{eq:def-Tx} is in fact a minimum.

\begin{lemma}\label{lem:continuity-of-T-and-M}
The function $T \colon X \rightarrow \R$ is continuous.
\end{lemma}
\begin{proof}
Let us first prove that the restriction of $T$ to an apartment $\Sigma \subseteq X$ with $\sigma \subseteq \pinf \Sigma$ is continuous.
Since $\xi$ is a vertex of $\sigma$, we have $[x,\xi)(t) \in \Sigma$ for every $x \in \Sigma$ and every $t \geq 0$.
From Proposition~\ref{prop:convexity-of-XXi-mu} we know that $\overline{X}_{\xi,\xi'}$ is a subbuilding of $X$.
In particular, $\overline{X}_{\xi,\xi'}$ is convex in $X$, which implies that $\Sigma \cap \overline{X}_{\xi,\xi'}$ is convex as well.
Note that this tells us that the restriction of $T$ to $\Sigma$ is a convex function, which is continuous by Corollary~\ref{cor:convex-contin-on-Rn}.
To deduce the continuity of $T$ on $X$, it now remains to apply~\cite[Theorem 11.63.(1)]{AbramenkoBrown08}, which tells us that $X$ is covered by apartments $\Sigma$ with $\sigma \subseteq \pinf \Sigma$.
Indeed, since $X$ is a cell complex, the continuity of a function on $X$ is equivalent to the continuity of its restriction to closed cells.
\end{proof}

\begin{proposition}\label{prop:cont-of-deform-to-XXi}
The building $\overline{X}^{\xi,\xi'}$ is a strong deformation retract of $X$.
Moreover, the pair $(\overline{X}^{\xi,\xi'}_{h \geq s},\overline{X}^{\xi,\xi'}_{h \geq t})$ is a strong deformation retract of the pair $(X_{h \geq s},X_{h \geq t})$ for all $s,t \in \R$ with $s \leq t$.
\end{proposition}
\begin{proof}
As a consequence of Lemma~\ref{lem:continuity-of-T-and-M}, we see that
\[
H \colon X \times [0,1] \rightarrow X,\ (x,t) \mapsto [x,\xi)(t \cdot T(x)),
\]
being a composition of continuous maps, is homotopy.
Since $H(x,0) = [x,\xi)(0) = x$ for every $x \in X$, we see that $H(-,0) = \id_X$.
On the other hand, it follows from the definition of $T$ that $H(x,1) = [x,\xi)(T(x)) \in \overline{X}^{\xi,\xi'}$ for every $x \in X$.
To prove the first claim, it now remains to observe that
\[
H(x,t) = [x,\xi)(t \cdot T(x)) = x
\]
for every $x \in \overline{X}^{\xi,\xi'}$ and every $t \in [0,1]$.
But this follows from $T(x) = 0$ for $x \in \overline{X}^{\xi,\xi'}$, which is a direct consequence of the definition of $T$.
From our assumption that $\xi \in \pinf \SigmaStd_{h = 0}$, we deduce that $h \circ [x,\xi)$ is constant for each $x \in X$.
Thus we have
\[
h(H(x,t)) = h([x,\xi)(t \cdot T(x)) = h(x)
\]
for all $(x,t) \in X \times [0,1]$ and the second claim follows by restricting $H$ to the $h$-superlevel sets of $X$.
\end{proof}

Let us translate Proposition~\ref{prop:cont-of-deform-to-XXi} to an analogous statement for $X_{\xi}$.
To do so, we write $h_{\xi}$ to denote the restriction of $h$ to $X_{\xi}$.

\begin{lemma}\label{lem:def-to-zero-level}
The building $X_{\xi}$ is a strong deformation retract of $X$.
Moreover $((X_{\xi})_{h_{\xi} \geq s}, (X_{\xi})_{h_{\xi} \geq t})$ is a strong deformation retract of $(X_{h \geq s},X_{h \geq t})$ for all $s,t \in \R$ with $s \leq t$.
\end{lemma}
\begin{proof}
For each $x \in X_{\xi}$ let $c_x$ denote the biinfinite geodesic from $\xi'$ to $\xi$ determined by $c_x(0) = x$ and let $\ell_x \defeq c_x(\R)$.
Then
\[
H_x \colon \ell_x \times [0,1] \rightarrow \ell_x,\ (c_x(s),t) \mapsto c_x((1-t)s)
\]
defines a strong deformation retraction with $H_x(x,t) = x$ for all $t \in [0,1]$.
Since $\overline{X}^{\xi,\xi'}$ is covered by spaces of the form $\ell_x$, we see that $X_{\xi}$ is a strong deformation retract of $\overline{X}^{\xi,\xi'}$.
From the choice of $\xi$ it moreover follows that $h \circ c_x$ is constant and hence that $((X_{\xi})_{h_{\xi} \geq s}, (X_{\xi})_{h_{\xi} \geq t})$ is a strong deformation retract of $(\overline{X}^{\xi,\xi'}_{h \geq s},\overline{X}^{\xi,\xi'}_{h \geq t})$ for all $s,t \in \R$ with $s \leq t$.
Now the lemma follows from Proposition~\ref{prop:cont-of-deform-to-XXi}.
\end{proof}

Let us write down the following direct consequence of Lemma~\ref{lem:def-to-zero-level}.

\begin{corollary}\label{cor:equivalence-of-essential-connectivity}
For every $k \in \N_0$, the following statements are equivalent:
\begin{enumerate}
\item $(X_{h \geq r})_{r \in \R}$ is essentially $k$-connected.
\item $((X_{\xi})_{h_{\xi} \geq r})_{r \in \R}$ is essentially $k$-connected.
\end{enumerate}
\end{corollary}

\subsection{Retractions in parabolic buildings}
Our next goal is to study the space of functions $(X_{\xi})_{\sigma_{\xi},\vstd}^{\ast}$ for an appropriate chamber $\sigma_{\xi} \subseteq \pinf X_{\xi}$, and to see how it is related to the space $X^{\ast}_{\sigma,\vstd}$.
To this end, we fix a set of linear forms $\mathcal{B} = \{\alpha_0,\ldots,\alpha_d\} \subseteq \SigmaStd^{\ast}$ as in Notation~\ref{not:alphas}.
More precisely, we assume that
\[
\overline{K_{\vstd}(\sigma)} = \Set{x \in \SigmaStd}{\alpha_k(x) \geq 0, \text{ for every } 0 \leq k \leq d}.
\]
From our choice of $\xi$, it follows that $\xi$ is the unique vertex of $\sigma$ that is not contained in $\pinf(\ker(\alpha_0))$.
In order to define $\sigma_{\xi} \subseteq \pinf X_{\xi}$, we consider the space
\[
K_{p}^{\xi,\xi'}(\sigma) \defeq \Set{x \in \SigmaStd}{\alpha_k(x) > \alpha_k(p) \text{ for every } 1 \leq k \leq d}
\]
for each $p \in \SigmaStd$.
Moreover, we define
\[
K_{p,\xi}(\sigma) \defeq K_{p}^{\xi,\xi'}(\sigma) \cap \SigmaStd_{\xi}
\]
if $p$ lies in $\SigmaStd_{\xi}$.

\begin{lemma}\label{lem:a-sector-in-Sigma-Xi}
The space $K_{\vstd,\xi}(\sigma)$ is a sector in $\SigmaStd_{\xi}$.
\end{lemma}
\begin{proof}
Let us first consider the case $p = \vstd$.
Consider the projection chamber $E \defeq \pr_{\vstd}(\sigma) \subseteq \st_{\SigmaStd}(\vstd)$.
As in the construction of $X_{\xi}$, let $\Lambda$ denote the union of horizontal walls in $X$.
Let $C$ denote the component in $\SigmaStd \backslash \Lambda$ that contains $E$.
By definition, the set $E_{\xi} \defeq C \cap \SigmaStd_{\xi}$ is a chamber in $\st_{\SigmaStd_{\xi}}(\vstd)$.
Note that $\ker(\alpha_i)$ is the affine span of a panel of $E$ for each $1 \leq i \leq d$.
Thus the wall $\SigmaStd_{\xi} \cap \ker(\alpha_i)$ is induced by a panel of $E_{\xi}$ that is incident to $\vstd$.
Since $X_{\xi}$ is $d$-dimensional, this tells us that every panel of $E_{\xi}$ that is incident to $\vstd$, spans a wall of the form $\SigmaStd_{\xi} \cap \ker(\alpha_i)$ with $1 \leq i \leq d$.
Now we see that $K_{\vstd,\xi}(\sigma)$ coincides with the sector induced by $\vstd$ and $E_{\xi}$, which proves the claim.
\end{proof}

In view of Lemma~\ref{lem:a-sector-in-Sigma-Xi} we can define the chamber $\sigma_{\xi} \defeq \pinf K_{\vstd,\xi}(\sigma)$ in $\pinf \SigmaStd_{\xi}$.
This allows us to consider the set $\mathcal{A}_{\xi}^{\sigma}$ consisting of apartments $\Sigma \in \mathcal{A}_{\xi}$ with $\sigma \subseteq \pinf \Sigma$.
Moreover, we can define the retraction from infinity  corresponding to $\SigmaStd_{\xi}$ and $\sigma_{\xi}$, which will be denoted by $\rho_{\xi} \defeq \rho_{\SigmaStd_{\xi},\sigma_{\xi}}$.
The following lemma tells us how $\rho_{\xi}$ is related to $\rho \defeq \rho_{\SigmaStd,\sigma}$.

\begin{lemma}\label{lem:restr-of-retr-doesnt-change}
The restriction of $\rho$ to $X_{\xi}$ coincides with $\rho_{\xi}$.
\end{lemma}
\begin{proof}
Let $p \in X_{\xi}$.
Since $\mathcal{A}_{\xi}$ is the complete apartment system of $X_{\xi}$ by Lemma~\ref{lem:prop-of-XXi}, we can apply~\cite[Theorem 11.63]{AbramenkoBrown08} to deduce that there is some apartment $\Sigma \in \mathcal{A}_{\xi}^{\sigma}$ containing $p$.
From $\sigma_{\xi} \subseteq \pinf \SigmaStd_{\xi}$ we see that $\Sigma$ contains $K_{q}(\sigma_{\xi})$ for some point $q \in \SigmaStd_{\xi}$.
By definition, there is an apartment $\Sigma' \in \mathcal{A}_{\hor}$ with $\Sigma = \Sigma' \cap X_{\xi}$.
In particular, $\Sigma'$ contains $K_{q}(\sigma_{\xi})$.
Since $[x,\xi)(t) \in \Sigma'$ for every $x \in K_{q}(\sigma_{\xi})$ and every $t \geq 0$, it follows that $\Sigma'$ contains an open subset of $\SigmaStd$.
This implies that there is a unique isomorphism $f \colon \Sigma' \rightarrow \SigmaStd$ that fixes $\Sigma' \cap \SigmaStd$.
Thus the restriction of $\rho$ to $\Sigma'$ concides with $f$.
Since $\beta$ is a Busemann function corresponding to the vertex $\xi$ of $\overline{\sigma}$ we see that $\beta \circ \rho = \beta$.
Hence $\rho$ maps $\Sigma = \Sigma' \cap X_{\xi}$ isomorphically to $\SigmaStd \cap X_{\xi} = \SigmaStd_{\xi}$, while fixing $K_{q}(\sigma_{\xi})$.
Since $p \in \Sigma$, it now remains to note that $K_{q}(\sigma_{\xi})$ is open in $\SigmaStd_{\xi} \cap \Sigma$, which implies that the restrictions of $\rho$ and $\rho_{\xi}$ coincide on $\Sigma$.
\end{proof}

Recall that we write $h_{\xi}$ to denote the restriction of $h$ to $X_{\xi}$.

\begin{corollary}\label{cor:new-height-pres-retr}
The function $h_{\xi}$ lies in $(X_{\xi})_{\sigma_{\xi},\vstd}^{\ast}$.
\end{corollary}
\begin{proof}
From Lemma~\ref{lem:restr-of-retr-doesnt-change} we know that $\rho_{\xi}$ is the restriction of $\rho$ to $X_{\xi}$.
Thus our assumption that $h \in X_{\vstd,\sigma}^{\ast}$ gives us
\[
h_{\xi}(\rho_{\xi}(x)) = h(\rho(x)) = h(x) = h_{\xi}(x)
\]
for every $x \in X_{\xi}$.
On the other hand, $h_{\xi} \colon \SigmaStd_{\xi} \rightarrow \R$ is the restriction of the linear function $h \colon \SigmaStd \rightarrow \R$ and hence is linear itself.
Now the claim follows from Remark~\ref{rem:another-char-of-retr-pres}.
\end{proof}

\begin{notation}
To simplify notation, we will often just write $X_{\xi}^{\ast}$ instead of $(X_{\xi})_{\sigma_{\xi},\vstd}^{\ast}$.
\end{notation}

\begin{remark}
We can also define a height function on $X^{\xi}$ by setting
\[
h^{\xi} \colon X^{\xi} \rightarrow \R,\ [[x,\xi)] \mapsto h(x).
\]
This is well-defined since $\xi$ was chosen such that $h \circ [x,\xi)$ is constant for every $x \in X$.
Using the isomorphism $X_{\xi} \rightarrow X^{\xi},\ x \mapsto p(x) = [[x,\xi)]$ from Remark~\ref{rem:X_xi-isom-X^xi}, we obtain an apartment $\SigmaStd^{\xi} \defeq p(\SigmaStd) = p(\SigmaStd_{\xi})$ in $X^{\xi}$, a special vertex $\vstd^{\xi} \defeq p(\vstd)$, a sector $p(K_{\vstd}(\sigma)) = p(K_{\vstd}(\sigma_{\xi}))$ in $\SigmaStd^{\xi}$, and the corresponding chamber $\sigma^{\xi} \subset \pinf \SigmaStd^{\xi}$.
Since $h^{\xi}(p(x)) = h_{\xi}(x)$ for every $x \in X_{\xi}$, it follows that $h^{\xi}$ lies in $(X^{\xi})_{\sigma^{\xi},\vstd^{\xi}}^{\ast}$.
\end{remark}

\subsection{Reduction of the horizontal dimension}

Recall that we have chosen a height function $h \in X_{\sigma,\vstd}^{\ast}$ with $\sigma \subseteq \pinf \SigmaStd_{h \leq 0}$.
Since moreover $\xi$ is a vertex of $\sigma$ that lies in $\pinf \SigmaStd_{h = 0}$, it follows that $\overline{\sigma} \cap \pinf \SigmaStd_{h = 0}$ is the closure of a non-empty face $\sigma_{\hor}$ of $\sigma$, which we will refer to as the \emph{horizontal face} of $\sigma$ with respect to $h$.
Similarly, we define $(\sigma_{\xi})_{\hor}$ as the unique (possibly empty) face of $\sigma_{\xi}$ whose closure coincides with $\overline{\sigma_{\xi}} \cap \pinf (\SigmaStd_{\xi})_{h_{\xi} = 0}$.
Note that this is possible since $\sigma_{\xi}$ was constructed as a subset of $\pinf ((\SigmaStd_{\xi})_{h_{\xi} \leq 0})$.
In the following, we use the convention that the dimension of the empty set is $-1$.

\begin{lemma}\label{lem:new-chamber-is-a-chamber}
With the notation above we have
\[
\dim((\sigma_{\xi})_{\hor}) = \dim(\sigma_{\hor}) - 1.
\]
\end{lemma}
\begin{proof}
Let $\mathcal{B} = \Set{\alpha_i}{0 \leq i \leq d}$ be as in Definition~\ref{not:alphas} and let $0 \leq j \leq d$ be such that $\sigma_{\hor} \subseteq \pinf (\ker(\alpha_{j}))$.
Since $\ker(\alpha_0)$ was defined as the unique non-horizontal wall among the kernels $\ker(\alpha_i)$, it follows that $\xi \notin \pinf \ker(\alpha_0)$.
Thus $j \geq 1$ and we see that the dimension of $\sigma_{\hor}$ is given by
\begin{align*}
\dim(\sigma_{\hor})
&= \max \limits_{1 \leq i \leq d} (\dim(\ker(\alpha_{i}) \cap \SigmaStd_{h=0}))-1\\
&= \dim(\ker(\alpha_j) \cap \SigmaStd_{h=0})-1.
\end{align*}
Moreover, each $1 \leq i \leq d$ gives rise to the chain of inclusions
\begin{equation}\label{eqn:equation62}
\ker(\alpha_i) \cap \SigmaStd_{h=0} \cap \SigmaStd_{\xi}
\subseteq \ker(\alpha_j) \cap \SigmaStd_{h=0} \cap \SigmaStd_{\xi}
\subsetneq \ker(\alpha_j) \cap \SigmaStd_{h=0},
\end{equation}
where the properness of the second inclusion follows from the observation that the ray $[\vstd,\xi)$ is contained in $\ker(\alpha_j) \cap \SigmaStd_{h=0}$ but not in $\SigmaStd_{\xi}$.
From the fact that $\SigmaStd_{\xi}$ is a hyperplane in $\SigmaStd$, we can now deduce that
\begin{equation}\label{eqn:equation63}
\dim(\ker(\alpha_j) \cap \SigmaStd_{h=0} \cap \SigmaStd_{\xi})
= \dim(\ker(\alpha_j) \cap \SigmaStd_{h=0})-1.
\end{equation}
For each $1 \leq i \leq d$, let $\alpha_i^{\xi}$ denote the restriction of $\alpha_i$ to $\SigmaStd_{\xi}$.
By Lemma~\ref{lem:a-sector-in-Sigma-Xi}, it follows that the walls corresponding to the sector $K_{\vstd}(\sigma_{\xi}) \subseteq \SigmaStd_{\xi}$ can be written as $\ker(\alpha_i^{\xi}) = \ker(\alpha_i) \cap \SigmaStd_{\xi}$ for $1 \leq i \leq d$.
Since the zero level of $h_{\xi}$ in $\SigmaStd_{\xi}$ is given by $\SigmaStd_{\xi} \cap \SigmaStd_{h=0}$, we obtain
\begin{equation}\label{eq:equation61}
\dim(\sigma_{\xi}^{\hor})
=\max \limits_{1 \leq i \leq d} \dim(\ker(\alpha_{i}) \cap \SigmaStd_{\xi} \cap \SigmaStd_{h=0})-1.
\end{equation}
By combining~\eqref{eq:equation61} with~\eqref{eqn:equation62} and~\eqref{eqn:equation63}, we obtain
\begin{align*}
\dim(\sigma_{\xi}^{\hor})
&= \dim(\ker(\alpha_j) \cap \SigmaStd_{h=0} \cap \SigmaStd_{\xi})-1\\
&= \dim(\ker(\alpha_j) \cap \SigmaStd_{h=0})-2\\
&=\dim(\sigma_{\hor})-1,
\end{align*}
which proves the claim.
\end{proof}

Let us summarize what we have done so far.

\begin{lemma}\label{lem:reduction-one-step}
With the notation above, we have:
\begin{enumerate}
\item The dimension of $X_{\xi}$ is given by $\dim(X_{\xi}) = \dim(X) - 1$.
\item Given a cell $A \subseteq X_{\xi}$, there is a cell $A' \subseteq X$ with $\lk_{X_{\xi}}(A) \cong \lk_{X}(A')$.
\item Given any $k \in \N_0$, the system $(X_{h \geq r})_{r \in \R}$ is essentially $k$-connected if and only if the system $((X_{\xi})_{h_{\xi} \geq r})_{r \in \R}$ is essentially $k$-connected.
\item The dimension of $(\sigma_{\xi})_{\hor}$ is given by $\dim((\sigma_{\xi})_{\hor}) = \dim(\sigma_{\hor}) - 1$.
\item If there is an $\alpha \in \Aut(X)$ and some $a \in \R \setminus \{0\}$
such that $\alpha(\sigma) = \sigma$, $\alpha(\SigmaStd) = \SigmaStd$, and $h(\alpha(x)) = h(x) + a$ for every $x \in X$, then there is an $\alpha_{\xi} \in \Aut(X_{\xi})$ with $\alpha_{\xi}(\sigma_{\xi}) = \sigma_{\xi}$, $\alpha_{\xi}(\SigmaStd_{\xi}) = \SigmaStd_{\xi}$, and $h_{\xi}(\alpha_{\xi}(x)) = h_{\xi}(x) + a$ for every $x \in X_{\xi}$.
\end{enumerate}
\end{lemma}
\begin{proof}
The first assertion follows from the fact that the apartment $\SigmaStd_{\xi}$ in $X_{\xi}$ was defined as the hyperplane $\SigmaStd \cap \beta^{-1}(0)$ in $\SigmaStd$.
The claims (2)-(5) can be found in Lemma~\ref{lem:prop-of-XXi}, Corollary~\ref{cor:equivalence-of-essential-connectivity}, Lemma~\ref{lem:new-chamber-is-a-chamber}, and Lemma~\ref{lem:tranlating-auto}.
\end{proof}

Using Lemma~\ref{lem:reduction-one-step}, we can replace $(X_{h \geq r})_{r \in \R}$ by $((X_{\xi})_{h_{\xi} \geq r})_{r \in \R}$ in order to reduce the horizontal dimension of $\sigma$, without changing the essential connectivity properties.
By iterating this process, we are able to reduce our assumption that $\overline{\sigma} \subseteq \pinf \Sigma_{h \leq 0}$ to the case where $\overline{\sigma} \subseteq \pinf \Sigma_{h < 0}$.
This will be done in the following Corollary.
For the sake of easy reference, we recall and adapt some of the involved notations.

\begin{corollary}\label{cor:reduction}
Let $X$ be a Euclidean building, let $\Sigma \subseteq X$ be an apartment, let $\sigma \subseteq \pinf \Sigma$ be a chamber, let $v \in \Sigma$ be a special vertex, and let $h \in X_{\sigma,v}^{\ast}$ be such that $\overline{\sigma} \subseteq \pinf \Sigma_{h \leq 0}$.
Suppose that $\overline{\sigma} \cap \pinf \Sigma_{h = 0}$ is non-empty and let $\tau$ be the face of $\sigma$ that is determined by $\overline{\tau} = \overline{\sigma} \cap \pinf \Sigma_{h = 0}$.
Then there is a Euclidean building $X_{\tau}$, an apartment $\Sigma_{\tau} \subseteq X_{\tau}$, a chamber $\sigma_{\tau} \subseteq \pinf \Sigma_{\tau}$, a special vertex $v_{\tau} \in \Sigma_{\tau}$, and a height function $h_{\tau} \in (X_{\tau})_{\sigma_{\tau},v_{\tau}}^{\ast}$ with $\overline{\sigma_{\tau}} \subseteq \pinf ((\Sigma_{\tau})_{h_{\tau} < 0})$ such that the following properties are satisfied:
\begin{enumerate}
\item The dimension of $X_{\tau}$ is given by $\dim(X_{\tau}) = \dim(X)-\dim(\tau)-1$.
\item Given a cell $A \subseteq X_{\tau}$, there is a cell $A' \subseteq X$ with $\lk_{X_{\tau}}(A) \cong \lk_{X}(A')$.
In particular, $\thick(X_{\tau}) \geq \thick(X)$ and $X_{\tau}$ is locally finite if $X$ is locally finite.
\item Given any $k \in \N_0$, the system $(X_{h \geq r})_{r \in \R}$ is essentially $k$-connected if and only if the system $((X_{\tau})_{h_{\tau} \geq r})_{r \in \R}$ is essentially $k$-connected.
\item If there is an $\alpha \in \Aut(X)$ and some $a \in \R \setminus \{0\}$
such that $\alpha(\sigma) = \sigma$, $\alpha(\SigmaStd) = \SigmaStd$, and $h(\alpha(x)) = h(x) + a$ for every $x \in X$, then there is an $\alpha_{\tau} \in \Aut(X_{\tau})$ with $\alpha_{\tau}(\sigma_{\tau}) = \sigma_{\tau}$, $\alpha_{\tau}(\SigmaStd_{\tau}) = \SigmaStd_{\tau}$, and $h_{\tau}(\alpha_{\tau}(x)) = h_{\tau}(x) + a$ for every $x \in X_{\tau}$.
\end{enumerate}
\end{corollary}
\begin{proof}
By assumption, $\tau$ is non-empty.
We can therefore choose a vertex $\xi$ of $\tau$ and its opposite vertex $\xi' \in \pinf \Sigma$.
In this case we can define $X_{\xi}$, $\sigma_{\xi}$, $(\sigma_{\xi})_{\hor}$, and $h_{\xi} \in (X_{\xi})_{\sigma_{\xi},v}^{\ast}$ as above and apply Lemma~\ref{lem:reduction-one-step} to deduce that the assertions (1)-(4) are true if $\tau = \xi$.
Moreover, Lemma~\ref{lem:reduction-one-step} tells us that $\dim((\sigma_{\xi})_{\hor}) = \dim(\sigma_{\hor}) - 1$.
Thus if $\dim(\sigma_{\hor}) = 0$, it follows that $(\sigma_{\xi})_{\hor} = \emptyset$.
In this case we have $\sigma_{\xi} \subseteq \pinf ((\Sigma_{\xi})_{h_{\xi} < 0})$ and the claim follows.
Now the general case follows by induction on the dimension of $\sigma_{\hor}$.
\end{proof}

Let us recall the following typical property of groups acting on buildings.

\begin{definition}\label{def:strongly-transitive}
Let $\Delta$ be a building and let $\Aut(\Delta)$ denote the group of type preserving, cellular automorphisms of $\Delta$.
We say that a subgroup $G \subseteq \Aut(\Delta)$ acts \emph{strongly transitively} on $\Delta$ if $G$ acts transitively on the set of pairs $(\Sigma,E)$, where $\Sigma$ is an apartment of $\Delta$ and $E$ is a chamber of $\Sigma$.
\end{definition}


\begin{lemma}\label{lem:lifting-shifts}
Suppose that $\Aut(X)$ acts strongly transitively on $X$.
Then there is an automorphism $\alpha \in \Aut(X)$ and a constant $a \in \R \setminus \{0\}$ such that
\begin{enumerate}
\item $\alpha(\sigma) = \sigma$,
\item $\alpha(\SigmaStd) = \SigmaStd$, and
\item $h(\alpha(x)) = h(x) + a$ for every $x \in X$.
\end{enumerate}
\end{lemma}
\begin{proof}
Let $T \in \Aut(\SigmaStd)$ be type-preserving translation with $a \defeq h(T(\vstd))-h(\vstd) \neq 0$.
Let $E \subseteq \SigmaStd$ be a chamber.
Since $\Aut(X)$ acts strongly transitively on $X$, there is some $\alpha \in \Aut(X)$ such that $\alpha(\SigmaStd) = \SigmaStd$ and $\alpha(E) = T(E)$.
Note that $T$, being a translation, is completely determined by its action on $E$.
Thus the restriction of $\alpha$ to $\SigmaStd$ is given by $\alpha_{\vert \SigmaStd} = T$.
In particular this implies $\alpha(K_{\vstd}(\sigma)) = K_{\alpha(\vstd)}(\sigma)$ and hence $\alpha(\sigma) = \sigma$.
We claim that $\rho \circ \alpha = \alpha \circ \rho.$
Since $X$ is covered by apartments in $\mathcal{A}_{\sigma}$, it suffices to verify this claim for an every apartment in $\mathcal{A}_{\sigma}$.
Let us therefore fix an apartment $\Sigma \in \mathcal{A}_{\sigma}$.
Note that the maps $\rho \circ \alpha$ and $\alpha \circ \rho$ restrict to isomorphisms $\Sigma \rightarrow \SigmaStd$.
Since isomorphisms between Coxeter complexes are determined by the image of any chamber, it is sufficient to find a chamber $C \subseteq \Sigma$ with $\rho \circ \alpha(C) = \alpha \circ \rho(C)$.
By assumption on $\Sigma \in \mathcal{A}_{\sigma}$, there is a sector $K_{w}(\sigma)$ lying in $\Sigma \cap \SigmaStd$.
In particular, we see that there is a chamber $C \subseteq \Sigma \cap \SigmaStd$ and that
\[
\rho \circ \alpha(C)
= \rho \circ T(C)
= T(C)
= \alpha(C)
= \alpha \circ \rho(C).
\]
It remains to show that $h(\alpha(x)) = h(x) + a$ for every $x \in X$.
Note that this is clear for $x \in \SigmaStd$.
As a function in $X^{\ast}$, our height function $h$ satisfies $h(\rho(x)) = h(x)$ for every $x \in X$.
We can therefore deduce that
\begin{center}
\begin{tabular}{lll}
$h(\alpha(x)) - h(x)$ & $= h(\rho \circ \alpha(x)) - h(\rho(x))$ & $= h(\alpha \circ \rho(x)) - h(\rho(x))$\\[2ex]
       & $= h(T(\rho(x))) - h(\rho(x))$   & $= a$
       \end{tabular}
\end{center}
for every $x \in X$, which proves the claim.
\end{proof}

\subsection{The geometric main result}\label{subsec:geo-main}

We are now ready to prove the main theorem of this section.
To formulate it, we fix a $d$-dimensional, thick Euclidean building $X$, an apartment $\Sigma$, a special vertex $v \in \Sigma$, and a chamber $\sigma \subseteq \pinf \Sigma$.
Let $\mathcal{B} = \{\alpha_1,\ldots,\alpha_d\}$ be a set of linear forms in $\Sigma^{\ast}$ such that
\[
\overline{K_{v}(\sigma)} = \Set{x \in \Sigma}{\alpha_i(x) \geq 0 \text{ for every } 1 \leq i \leq d}.
\]
For each $1 \leq i \leq d$, we consider the height function $h_i \defeq - \alpha_i \circ \rho_{\sigma, \Sigma}$ on $X$.
Since $\sigma$ is a chamber, it follows that $\mathcal{B}$ is a basis of $\Sigma^{\ast}$ and hence that $\mathcal{B}_{\sigma,v} \defeq \{h_1,\ldots,h_d\}$ is a basis of $X^{\ast}_{\sigma,v} = \Set{\alpha \circ \rho_{\sigma, \Sigma}}{\alpha \in \Sigma^{\ast}}$.
Recall from Notation~\ref{not:homothety-classes-general} that for every finite dimensional real vector space $V$, we write $S(V) = (V \setminus \{0\}) / \hspace{-1.5mm} \sim$ to denote the space of positive homothety classes of non-trivial elements of $V$.
When it comes to describe the $\Sigma$-invariants of a group, the following definition is often useful.

\begin{definition}\label{def:conv-hull}
Let $V$ be a finite dimensional real vector space and let $M \subseteq S(V)$ be an arbitrary subset.
For each $k \in \N$, we define the \emph{$k$-convex hull} of $M$ as the subset $\conv_k(M) \subseteq S(V)$ whose elements are represented by non-zero vectors of the form $\sum_{i=1}^k \lambda_i v_i$ with non-negative coefficients $\lambda_i$ such that the vectors $v_i$ represent pairwise different elements in $M$.
The \emph{convex hull} of $M$ is given by $\conv(M) \defeq \bigcup \limits_{k \in \N} \conv_k(M)$.
\end{definition}

Let us apply Definition~\ref{def:conv-hull} on the subset $\Delta_{\sigma,v}^{(0)} \subseteq S(X_{\sigma,v}^{\ast})$ consisting of classes that are represented by the elements in $\mathcal{B}_{\sigma,v}$.
As the notation suggests, we think of $\Delta_{\sigma,v}^{(0)}$ as the vertex set of a simplex in $S(X_{\sigma,v}^{\ast})$.
Accordingly, we write $\Delta_{\sigma,v}^{(k)}$ to denote the $(k+1)$-convex hull of $\Delta_{\sigma,v}^{(0)}$.
Note that the notation makes sense because $\mathcal{B}_{\sigma,v}$ is a basis of $X_{\sigma,v}^{\ast}$.
Since $\dim(X) = d$, the convex hull of $\Delta_{\sigma,v}^{(0)}$ is given by $\Delta_{\sigma,v} = \Delta_{\sigma,v}^{(d-1)}$.

\begin{theorem}\label{thm:A-new}
With the notation above, let $h$ be a height function in $X_{\sigma,v}^{\ast}$ and let $k \geq 0$ be an integer.
Then the following hold.
\begin{enumerate}
\item If $h$ does not represent an element of $\Delta_{\sigma,v}$, then $X_{h \geq r}$ is contractible for every $r \in \R$.
\item Suppose that $X$ satisfies $(\SOL)$.
If $h$ does not represent a class in $\Delta_{\sigma,v}^{(k)}$, then $(X_{h \geq r})_{r \in \R}$ is essentially $k$-connected.
\item Suppose there is an $\alpha \in \Aut(X)$ and an $a \in \R \setminus \{0\}$
such that $\alpha(\sigma) = \sigma$, $\alpha(\Sigma) = \Sigma$, and $h(\alpha(x)) = h(x) + a$ for every $x \in X$.
If $h$ represents a class in $\Delta_{\sigma,v}^{(k)}$, then $(X_{h \geq r})_{r \in \R}$ is not essentially $k$-acyclic.
\end{enumerate}
\end{theorem}
\begin{proof}
Let $s \in \R$ be arbitrary.
Suppose first that $h$ does not represent an element of $\Delta_{\sigma,v}$.
Then $h = \sum \limits_{i=1}^d \lambda_{i} h_i$, where at least one coefficient, say $\lambda_1$, is negative.
Consider the $1$-dimensional subspace $U \defeq \bigcap \limits_{i=2}^d \ker(\alpha_i) \leq \Sigma$ and let $\xi$ denote the vertex of $\sigma$ that lies in $\pinf U$.
Then $\alpha_i \circ [v,\xi)$ is constantly $0$ for each $1 < i \leq d$, while $\alpha_1 \circ [v,\xi)$ is strictly increasing.
It follows that
\[
h \circ [v,\xi)
= \lambda_{1} h_1 \circ [v,\xi)
= \lambda_{1} (- \alpha_1 \circ \rho_{\sigma, \Sigma}) \circ [v,\xi)
= - \lambda_{1} \alpha_1 \circ [v,\xi)
\]
is strictly increasing.
We can therefore choose an (interior) point $\eta \in \sigma$ such that $h \circ [x,\eta)$ is strictly increasing for every $x \in X$.
It follows that for all $x,y \in X_{h \geq s}$, there is some $t \geq 0$ such that $[x,\eta)(t)$ is contained in $K_y(\sigma)$.
Let $n \in \N_0$ and let $f \colon S^n \rightarrow X_{h \geq s}$ be a continuous map.
Since $S^n$ is compact, it follows that there is a sector $K_w(\sigma) \subseteq \Sigma$ and a number $t \geq 0$ such that $[f(z),\eta)(t) \in K_w(\sigma) \cap X_{h \geq s}$ for every $z \in S^n$.
From the convexity of $K_w(\sigma) \cap  X_{h \geq s} = K_w(\sigma) \cap \Sigma_{h \geq s}$ we deduce that $f$ is null-homotopic, which proves the first claim.

Let us now consider the case where $[h]$ is contained in $\Delta_v(\sigma)$ and let $k$ be minimal with $[h] \in \Delta_v(\sigma)^{(k)}$.
In this case we have $\overline{\sigma} \subseteq \pinf \Sigma_{h \leq r}$ so that we can define the horizontal face $\sigma_{\hor}$ of $\sigma$ with respect to $h$.
Note that the dimension of $\sigma_{\hor}$ is given by
\[
\dim(\sigma_{\hor}) = \dim(\sigma)-k-1 = \dim(X)-k-2.
\]

If $\sigma_{\hor} = \emptyset$, then the second and third claim is covered by Theorem~\ref{thm:essentially-connected}, respectively by Theorem~\ref{thm:negative-global-geom}.
We can therefore assume that $\sigma_{\hor} \neq \emptyset$.
Then Corollary~\ref{cor:reduction} equips us with a $(k+1)$-dimensional Euclidean building $X_{\tau}$, an apartment $\Sigma_{\tau}$, a chamber $\sigma_{\tau} \subseteq \pinf \Sigma_{\tau}$, a special vertex $v_{\tau}$, and a height function $h_{\tau} \in (X_{\tau})_{\sigma_{\tau},v_{\tau}}^{\ast}$ with $\overline{\sigma_{\tau}} \subseteq \pinf ((\Sigma_{\tau})_{h_{\tau} < 0})$ such that
\begin{enumerate}
\item for every $m \in \N_0$ the system $(X_{h \geq r})_{r \in \R}$ is essentially $m$-connected (respectively $m$-acyclic) if and only if $((X_{\tau})_{h_{\tau} \geq r})_{r \in \R}$ is essentially $m$-connected  (respectively $m$-acyclic),
\item the horizontal part of $\sigma_{\tau}$ with respect to $h_{\tau}$ is empty.
\end{enumerate}
Suppose that $X$ satisfies $(\SOL)$.
Then it follows from Corollary~\ref{cor:reduction}(2) that $X_{\tau}$ satisfies the $(\SOL)$ so that we may apply Theorem~\ref{thm:essentially-connected} to deduce that $((X_{\tau})_{h_{\tau} \geq r})_{r \in \R}$ and $(X_{h \geq r})_{r \in \R}$ are essentially $(k-1)$-connected, which proves the second claim.
Suppose next that $\Aut(X)$ there is an $\alpha \in \Aut(X)$ and an $a \in \R \setminus \{0\}$ such that $\alpha(\sigma) = \sigma$, $\alpha(\Sigma) = \Sigma$, and $h(\alpha(x)) = h(x) + a$ for every $x \in X$.
In this case Corollary~\ref{cor:reduction}(4) tells us that there is an $\alpha_{\tau} \in \Aut(X_{\tau})$ with $h_{\tau}(\alpha_{\tau}(x)) = h_{\tau}(x) + a$ for every $x \in X_{\tau}$.
Now we can apply Theorem~\ref{thm:negative-global-geom} to deduce that $((X_{\tau})_{h_{\tau} \geq r})_{r \in \R}$ and $(X_{h \geq r})_{r \in \R}$ are not essentially $k$-acyclic, which completes the proof.
\end{proof}

\begin{corollary}\label{cor:A}
With the notation above, let $h$ be a height function in $X_{\sigma,v}^{\ast}$.
Suppose that $\Aut(X)$ acts strongly transitively on $X$ and that $X$ satisfies $(\SOL)$.
Then for every integer $k \geq -1$, the following are equivalent:
\begin{enumerate}
\item $(X_{h \geq r})_{r \in \R}$ is essentially $k$-connected.
\item $(X_{h \geq r})_{r \in \R}$ is essentially $k$-acyclic.
\item $h$ does not represent an element of $\Delta_{\sigma,v}^{(k)}$.
\end{enumerate}
\end{corollary}
\begin{proof}
In view of Theorem~\ref{thm:A-new}, it only remains to recall from Lemma~\ref{lem:lifting-shifts} that the strong transitivity provides us with an $\alpha \in \Aut(X)$ and some $a \in \R \setminus \{0\}$ such that, $\alpha(\sigma) = \sigma$,
$\alpha(\Sigma) = \Sigma$, and $h(\alpha(x)) = h(x) + a$ for every $x \in X$.
\end{proof}

\section{Chevalley groups, Borel groups, and their $S$-arithmetic subgroups}\label{sec:chevalley-groups}

In order to apply our topological results on buildings to the computation of the $\Sigma$-invariants of a group, we need the group to act on a building.
One source of groups acting on Euclidean buildings is the class of Chevalley groups over valued fields.
To define these groups, we have to recall some classical facts from the theory of complex Lie algebras.
All the details can be found in~\cite{Humphreys78}.

\subsection{Background on Lie algebras}

For the rest of this section we fix a complex semisimple Lie algebra $\mathcal{L}$.
Let $\kappa$ be its Killing form and let $\mathcal{H} \subseteq \mathcal{L}$ be a fixed a Cartan subalgebra.
Since the restriction of $\kappa$ on $\mathcal{H}$ is non-degenerate by~\cite[Corollary 8.2]{Humphreys78}, we can identify $\mathcal{H}$ with its dual $\mathcal{H}^{\ast}$ via $H \mapsto \kappa(H,-)$.
For each $\gamma \in \mathcal{H}^{\ast}$ let $H_{\gamma}' \in \mathcal{H}$ be the corresponding element with respect to this identification.
The set of roots associated to $\mathcal{H}$ will be denoted by $\Phi \subseteq \mathcal{H}^{\ast}$.
Let $\mathcal{V} = \langle \Phi \rangle_{\R}$ be its real span.
We endow $\mathcal{V}$ with an inner product $\kappa^{\ast}$ given by $\kappa^{\ast}(\alpha,\beta) \defeq \kappa(H_{\alpha}', H_{\beta}')$.
The Euclidean space $(\mathcal{V},\kappa^{\ast})$ is going to be the standard apartment of our Euclidean building.
Let $\Delta \defeq \{ \alpha_1, \ldots, \alpha_{\ell} \} \subseteq \Phi$ be a system of simple roots, let $\Phi^{+} \subseteq \Phi$ be the corresponding set of positive roots, and let $\widetilde{\alpha} \in \Phi^{+}$ be the highest root.
The Weyl group associated to $\Phi$ will be denoted by $W_{\Phi}$, i.e.\ $W_{\Phi} \subseteq \Isom(\mathcal{V})$ is the group generated by the reflections $s_{\alpha}$ through the hyperplanes
\[
W_{\alpha} \defeq \Set{v \in \mathcal{V}}{\kappa^{\ast}(\alpha,v) = 0}.
\]
The set of reflections corresponding to the simple roots will be denoted by $S_{\Delta} = \Set{s_{\alpha}}{\alpha \in \Delta}$.
For every two elements $v,w \in \mathcal{V}$ we write $\langle v,w \rangle = 2 \frac{\kappa^{\ast}(v,w)}{\kappa^{\ast}(v,v)}$.
With this notation we can express reflections by $s_{\alpha}(v) = v - \langle v,\alpha \rangle \alpha$.
Recall that $\mathcal{L}$ can be decomposed as a direct sum $\mathcal{L} = \mathcal{H} \oplus \bigoplus \limits_{\alpha \in \Phi} \mathcal{L}_{\alpha}$, where
\[
\mathcal{L}_{\alpha} = \Set{X \in \mathcal{L}}{[H,X] = \alpha(H) X,\ \forall H \in \mathcal{H}}.
\]
It can be shown that each $\mathcal{L}_{\alpha}$ is one-dimensional and that $[\mathcal{L}_{\alpha},\mathcal{L}_{-\alpha}] \subseteq \mathcal{H}$.
The following result can be found in~\cite[Theorem 1]{Steinberg68}.
In order to state it, it is convenient to rescale the elements $H_{\alpha}'$ by defining $H_{\alpha} = \frac{2H_{\alpha}'}{\kappa(H_{\alpha}',H_{\alpha}')}$ for each root $\alpha$.
In the case of simple roots we will write $H_i = H_{\alpha_i}$.

\begin{Citation}\label{cit:exist-chevalley-basis}
For $(H_i)_{i=1}^{\ell}$ as above there are elements $X_{\alpha} \in \mathcal{L}_{\alpha}$ for each $\alpha \in \Phi$ such that the set
\begin{equation}\label{eq:cheballey-basis}
\Set{H_i}{i = 1, \ldots,\ell} \cup \Set{X_{\alpha}}{\alpha \in \Phi}
\end{equation}
is a linear basis of $\mathcal{L}$ and the following relations are satisfied.
\begin{enumerate}
\item[$(a)$] $[H_i, H_j] = 0$.
\item[$(b)$] $[H_i,X_{\alpha}] = \alpha(H_i) X_{\alpha}$.
\item[$(c)$] $[X_{\alpha},X_{-\alpha}] = H_{\alpha}$, where $H_{\alpha}$ is an integral linear combination of the $H_i$.
\item[$(d)$] $[X_{\alpha},X_{\beta}] = \pm(r + 1)X_{\alpha+\beta}$ if $\alpha + \beta \in \Phi$, where
\[
r \defeq \max \Set{k \in \N_0}{\beta - k \alpha \in \Phi}.
\]
\item[$(e)$] $[X_{\alpha},X_{\beta}] = 0$ if $\alpha + \beta \neq 0$ and $\alpha + \beta \notin \Phi$.
\end{enumerate}
\end{Citation}

The set in~\eqref{eq:cheballey-basis} will be referred to as a \emph{Chevalley basis} of $\mathcal{L}$.
To define Chevalley groups, we have to choose an irreducible, faithful, finite dimensional representation $\rho \colon \mathcal{L} \rightarrow \End(V)$.
By~\cite[Theorem 3]{Steinberg68} there is a finite set $\Psi \subseteq \mathcal{H}^{\ast}$ of so-called \emph{weights} such that $V$ splits into a direct sum $V = \bigoplus \limits_{\mu \in \Psi} V_{\mu}$ of non-trivial weight spaces
\[
V_{\mu} = \Set{v \in V}{\rho(H)(v) = \mu(H) v,\ \forall H \in \mathcal{H}}.
\]
The following result, which is summarizes~\cite[Corollaries 1 and 2]{Steinberg68}, finally leads us to the definition of Chevalley groups.

\begin{Citation}\label{cit:invariant-lattice}
The $\mathcal{L}$-module $V$ contains a lattice $M$, i.e.\ the $\Z$-span of a basis of $V$, such that the following hold.
\begin{enumerate}
\item $M$ is invariant under the action of $\frac{\rho(X_{\alpha})^{n}}{n!}$.
\item The lattice $M$ splits as a direct sum $M = \bigoplus \limits_{\mu \in \Psi} M_{\mu}$
with ${M_{\mu} \defeq M \cap V_{\mu}}$.
\item The part of $\mathcal{L}$ that leaves $M$ invariant is given by
\[
\mathcal{L}_{\Z} = \mathcal{H}_{\rho} \oplus \bigoplus \limits_{\alpha \in \Phi} \langle X_{\alpha} \rangle_{\Z}, \text{ where } \mathcal{H}_{\rho} = \Set{H \in \mathcal{H}}{\mu(H) \in \Z,\ \forall \mu \in \Psi}.
\]
\end{enumerate}
\end{Citation}

\subsection{Chevalley groups and their associated subgroups}

Let $\mathcal{L}$, $\mathcal{H}$, $\Phi$, $\mathcal{V}$, $V$, $M$, and $\Psi$ be as above.
In order to construct Chevalley groups there are three choices to make, two of which we already made.
The first choice was the semisimple Lie algebra $\mathcal{L}$, which by~\cite[Theorem 14.2]{Humphreys78} is the same as choosing the root system $\Phi$.
The second choice was the representation $\rho$ of $\mathcal{L}$.
The last choice we have to make is the choice of a field $K$, which we will fix from now on.
All the other choices we made, i.e.\ the choice of the invariant lattice $M$, the choice of the Cartan subalgebra $\mathcal{H}$, as well as the choice of the system of simple roots $\Delta$ do not change the structure of the resulting Chevalley group up to isomorphism.
The Chevalley group assiciated to $\Phi$, $\rho$, and $K$ is defined via its action on $V \defeq M \otimes_{\Z} K$.
From the theory of real Lie groups it is well known that the exponential map induces a function from the Lie algebra to the Lie group.
A similar approach can be taken in the case of arbitrary fields.
Here we have to restrict the exponentiation to nilpotent elements.
From~\cite[Lemma 11]{Steinberg68} it follows that $\rho(X_{\alpha}) \in \End(V)$ is nilpotent for every $\alpha \in \Phi$.
In view of this, the following Notation makes sense.

\begin{notation}\label{not:elementary-generators}
For every $t \in K$ we write $x_{\alpha}(t) = \sum \limits_{n \geq 0} \frac{t^n \cdot X_{\alpha}^{n}}{n!}$, where $X_{\alpha}$ is identified with its image $\rho(X_{\alpha}) \in \End(V)$.
\end{notation}

A quick computation shows that $x_{\alpha}(s)x_{\alpha}(t) = x_{\alpha}(s+t)$ for every two elements $s,t \in K$.
Since $x_{\alpha}(0) = \id_V$, this implies $x_{\alpha}(t) \in \GL(V)$.

\begin{definition}\label{def:chavalleygroup}
The \emph{Chevalley group} associated to $\Phi$, $\rho$, and $K$, denoted by $\mathcal{G} = \mathcal{G}(\Phi,\rho,K)$, is the subgroup of $\GL(V)$ that is generated by the elements $x_{\alpha}(t)$ for every $\alpha \in \Phi$ and every $t \in K$.
\end{definition}

The following result, which can be found in~\cite[Page 16]{Steinberg68}, contains a list of useful relations that hold in $\mathcal{G}$.
To formulate it conviently, let us introduce some notation.

\begin{notation}\label{not:special-elements}
For every $\alpha \in \Phi$ and every $t \in K^{\times}$ we consider the elements $w_{\alpha}(t) = x_{\alpha}(t)x_{-\alpha}(-t^{-1})x_{\alpha}(t)$, $h_{\alpha}(t) = w_{\alpha}(t)w_{\alpha}(1)^{-1}$, and $\omega_{\alpha} = w_{\alpha}(1)$.
\end{notation}

\begin{Citation}\label{cit:relation-in-G}
Given $\alpha, \beta \in \Phi$ and $s,t \in K$, we have
\begin{enumerate}
\item[$(a)$] $x_{\alpha}(s)x_{\alpha}(t) = x_{\alpha}(s+t)$,
\item[$(b)$] $[x_{\alpha}(s),x_{\beta}(t)] = \prod \limits_{i,j > 0} x_{i \alpha + j \beta}(c_{i,j,\alpha,\beta}t^i s^j)$ if $\alpha + \beta \neq 0$ and where the constants $c_{i,j,\alpha,\beta} \in K$ are suitably defined,
\item[$(c)$] $\omega_{\alpha} h_{\beta}(t) \omega_{\alpha}^{-1} = \prod \limits_{\gamma \in \Phi} h_{\gamma}(t_{\gamma})$ for appropriate $t_{\gamma} \in K^{\times}$,
\item[$(d)$] $\omega_{\alpha} x_{\beta}(t) \omega_{\alpha}^{-1} = x_{\omega_{\alpha}(\beta)}(ct)$, where $c \in \{-1,1\}$,
\item[$(e)$] $h_{\alpha}(t)x_{\beta}(s)h_{\alpha}(t)^{-1} = x_{\beta}(t^{\langle \beta,\alpha \rangle}s)$, where $t \in K^{\times}$.
\end{enumerate}
\end{Citation}

Equipped with the elements defined in Notation~\ref{not:special-elements}, we can define some subgroups of $\mathcal{G}$, which will be of particular importance for us.

\begin{definition}\label{def:subgroups-of-G}
With the notation above let
\begin{enumerate}
\item $\mathcal{U}_{\alpha}$ be the group $\Set{x_{\alpha}(t)}{t \in K}$ for some $\alpha \in \Phi$,
\item $\mathcal{U}$ be the group generated by all subgroups $\mathcal{U}_{\alpha}$ with $\alpha \in \Phi^{+}$,
\item $\mathcal{T}$ be the group generated by all elements of the form $h_{\alpha}(t)$,
\item $\mathcal{B}$ be the group generated by $\mathcal{U}$ and $\mathcal{T}$, and
\item $\mathcal{N}$ be the group generated by all elements of the form $w_{\alpha}(t)$.
\end{enumerate}
The groups $\mathcal{U}$,$\mathcal{T}$, and $\mathcal{B}$ are called the \emph{unipotent}, \emph{torus}, and \emph{Borel subgroup} of $\mathcal{G}$.
\end{definition}

The following result describes the structure of $\mathcal{T}$, see~\cite[Lemma 28]{Steinberg68}.

\begin{Citation}\label{cit:structure-of-T}
The function
\[
h \colon (K^{\times})^{\ell} \rightarrow \mathcal{T},\ (t_i)_{i=1}^{\ell} \mapsto \prod \limits_{i=1}^{\ell} h_{\alpha_i}(t_i)
\]
is an epimorphism.
\end{Citation}

From~\cite[Lemma 22]{Steinberg68} we know that the map $s_{\alpha} \mapsto w_{\alpha}(1)\mathcal{T}$, where $\alpha$ runs over $\Phi$, extends to an isomorphism $W_{\Phi} \rightarrow \mathcal{N} / \mathcal{T}$.
Moreover we have $\mathcal{T} = \mathcal{B} \cap \mathcal{N}$ by~\cite[Lemma 7.93]{AbramenkoBrown08}.
As a consequence, we can identify $S_{\Delta}$ with its image in $\mathcal{N} / (\mathcal{B} \cap \mathcal{N})$, which allows us to formulate the following well-known result.

\begin{theorem}\label{thm:sph-BN-pair-for-G}
The tuple $(\mathcal{G}, \mathcal{B}, \mathcal{N}, S_{\Delta})$ is a Tits system.
\end{theorem}
\begin{proof}
From~\cite[Section 7.9.2]{AbramenkoBrown08} we know that the triple $(\mathcal{G},(\mathcal{U}_{\alpha})_{\alpha \in \Phi},\mathcal{T})$ is a so-called $\RGD$-system of type $(W_{\Phi},S_{\Delta})$.
In this case,~\cite[Theorem 7.115]{AbramenkoBrown08}) tells us that $(\mathcal{G}, \mathcal{B}, \mathcal{N}, S_{\Delta})$ is a Tits system.
\end{proof}

\subsection{A Euclidean BN-pair}\label{subsection:from-ValRtSys-To-BN}

Suppose now that $K$ possesses a discrete valuation $v \colon K \rightarrow \Z \cup \{ \infty \}$.
Then there is a Euclidean $\BN$-pair associated to $\mathcal{G}$.
In order to construct it, we have to associate the $(\mathcal{V},\kappa^{\ast})$ with the structure of a Euclidean Coxeter complex.

\begin{definition}\label{def:V-as-eucl-complex}
Let $\widetilde{\Phi} \defeq \Phi \times \Z$.
For each $(\alpha,k) \in \Phi \times \Z$, we define $s_{\alpha,k} \in \Isom(\mathcal{V})$ as the reflection given by $s_{\alpha,k}(v) = s_{\alpha}(v) + k \alpha^{V}$, where $\alpha^{V} \defeq \frac{2 \alpha}{\kappa^{\ast}(\alpha,\alpha)}$ denotes the \emph{coroot} of $\alpha$, and let
\[
H_{\alpha,k} = \Set{v \in \mathcal{V}}{\kappa^{\ast}(v,\alpha) = k}
\]
denote its invariant hyperplane.
\end{definition}

Let $W_{\widetilde{\Phi}}$ be the group generated by all reflections $s_{\alpha,k}$.
From~\cite[Section 10.1.3]{AbramenkoBrown08} we know that $W_{\widetilde{\Phi}}$ is generated by
\[
S_{\widetilde{\Delta}} \defeq \Set{s_{\alpha,0}}{\alpha \in \Delta} \cup \{s_{\widetilde{\alpha},1} \}.
\]
Note that $W_{\widetilde{\Phi}}$ contains $W_{\Phi} = \langle \Set{s_{\alpha,0}}{\alpha \in \Delta} \rangle$ as a subgroup.
A quick calculation shows that the set of all hyperplanes $H_{\alpha,k}$ is invariant under the action of $W_{\widetilde{\Phi}}$.
As a consequence, it follows that $(W_{\widetilde{\Phi}},S_{\widetilde{\Delta}})$ is a Euclidean Coxeter system and that the hyperplanes $H_{\alpha,k}$ induces a structure of a Euclidean Coxeter complex on $\mathcal{V}$.

\medskip

Let us introduce some notation that will help us to conveniently define the groups that are involved in our Euclidean $\BN$-pair.

\begin{notation}
Given $\alpha \in \Phi$, we define $\phi_{\alpha} \colon \mathcal{U}_{\alpha} \setminus \{1\} \rightarrow \Z,\ x_{\alpha}(t) \mapsto v(t)$.
For each $k \in \Z$ we set $\mathcal{U}_{\alpha,k} \defeq \Set{u \in \mathcal{U}_{\alpha}}{\phi_{\alpha}(u) \geq k}$.
Moreover we write $m(x_{\alpha}(t)) \defeq w_{-\alpha}(-t^{-1})$ for every $t \in K^{\times}$.
\end{notation}

Note that each $\mathcal{U}_{\alpha,k}$ is a group.
By~\cite[Proposition 14.4]{Weiss09} there is an epimorphism $\pi \colon \mathcal{N} \rightarrow W_{\widetilde{\Phi}}$ that satisfies $\pi(m(u)) = s_{\alpha,-\phi_{\alpha}(u)}$ for every $\alpha \in \Phi$ and every $u \in \mathcal{U}_{\alpha} \setminus \{1\}$.

\begin{definition}\label{def:the-tilde-groups}
With the notation above let
\begin{enumerate}
\item $\widetilde{\mathcal{T}}$ denote the kernel of $\pi$,
\item $\widetilde{\mathcal{U}}$ denote the group generated by $\bigcup \limits_{\alpha \in \Phi^{+}} \mathcal{U}_{\alpha,0}$ and $\bigcup \limits_{\alpha \in \Phi^{-}} \mathcal{U}_{\alpha,1}$,
\item $\widetilde{\mathcal{B}}$ denote the group generated by $\widetilde{\mathcal{T}}$ and $\widetilde{\mathcal{U}}$,
\item $m_i \in \mathcal{N} / \widetilde{\mathcal{T}}$ denote the class represented by $m(x_{\alpha_i}(1))$ for $1 \leq i \leq \ell$ and $m_0 \in \mathcal{N} / \widetilde{\mathcal{T}}$ denote the class represented by $m(u)$, where $u \in \mathcal{U}_{-\widetilde{\alpha}}$ satisfies $\phi_{-\widetilde{\alpha}}(u)=1$,
\item $\widetilde{S} = \Set{m_i}{0 \leq i \leq \ell}$.
\end{enumerate}
\end{definition}

The Euclidean $\BN$-pair we are looking for is given by $(\widetilde{\mathcal{B}},\mathcal{N})$.
In fact the following result, which is a special case of~\cite[Theorem 14.38]{Weiss09}, provides us with a Tits system corresponding to $(\widetilde{\mathcal{B}},\mathcal{N})$.

\begin{Citation}\label{cit:BN-pair-from-VRGD}
The tuple $(\mathcal{G}, \widetilde{\mathcal{B}}, \mathcal{N}, \widetilde{S})$ is a Tits system.
Moreover we have $\widetilde{\mathcal{T}} = \widetilde{\mathcal{B}} \cap \mathcal{N}$ and the map $\widetilde{S} \rightarrow S_{\widetilde{\Delta}}$, defined by $m_i \mapsto s_{\alpha_i,0}$ for $1 \leq i \leq \ell$ and $m_0 \mapsto s_{\widetilde{\alpha},1}$, extends to an isomorphism $\mathcal{N} / \widetilde{\mathcal{T}} \cong W_{\widetilde{\Phi}}$.
\end{Citation}

\subsection{A metric for $\Delta(\widetilde{\mathcal{B}},\mathcal{N})$}\label{subsec:a-metric-on-X}

Let us next describe the canonical $\CAT(0)$-metric on the building $\Delta(\widetilde{\mathcal{B}},\mathcal{N})$ associated to $(\widetilde{\mathcal{B}},\mathcal{N})$, which was introduced in Definition~\ref{def:from-BN-to-building}.

\begin{notation}\label{not:chamber}
Let $E$ be the chamber in $\mathcal{V}$ defined by
\[
\Set{v \in \mathcal{V}}{\kappa^{\ast}(v, \alpha_i) > 0 \text{ for every } 1 \leq i \leq \ell \text{ and } \kappa^{\ast}(v, \widetilde{\alpha}) < 1}.
\]
For every $1 \leq i \leq \ell$, let $v_i$ be the vertex of $E$ not lying in $H_{\alpha_i,0}$ and let $v_0 = 0$.
We write $\mathcal{K} \subseteq \mathcal{V}$ to denote the sector corresponding to $E$ and $v_0$.
Let $\sigma$ denote the chamber in $\pinf \mathcal{V}$ induced by $\mathcal{K}$.
\end{notation}

Recall that the vertices in $\Sigma(\widetilde{\mathcal{B}},\mathcal{N})$ are given by the set of cosets of the form $nP_i$, where $n \in \mathcal{N}$ and $P_i \defeq P_{\widehat{m_i}} = P_{\widetilde{S} \backslash \{m_i\}}$ is the maximal parabolic subgroup associated to $m_i \in \widetilde{S}$.

\begin{notation}\label{not:the-complex-associated-to-V}
Let $\Sigma(\mathcal{V},W_{\widetilde{\Phi}})$ be the abstract simplicial complex induced by triangulating $\mathcal{V}$ with the hyperplanes $H_{\alpha,k}$ and let $\vert \Sigma(V,W_{\widetilde{\Phi}}) \vert$ denote its geometric realization.
We write $\iota_{\mathcal{V}} \colon \vert \Sigma(\mathcal{V},W_{\widetilde{\Phi}}) \vert \rightarrow \mathcal{V}$ for the canonical homeomorphism.
\end{notation}

According to~\cite[Proposition 10.13]{AbramenkoBrown08}, the map $nP_i \mapsto \pi(n)(v_i)$ extends to an $\mathcal{N}$-equivariant isomorphism of abstract simplicial complexes $f \colon \Sigma(\widetilde{\mathcal{B}},\mathcal{N}) \rightarrow \Sigma(V,W_{\widetilde{\Phi}})$.
By composing its geometric realizations $\vert f \vert$ with $\iota_{\mathcal{V}}$, we obtain an $\mathcal{N}$-equivariant homeomorphism
$\iota \colon \vert \Sigma(\widetilde{\mathcal{B}},\mathcal{N}) \vert \rightarrow \mathcal{V}$.
Thus there is a canonical metric $d$ on $\vert \Sigma(\widetilde{\mathcal{B}},\mathcal{N}) \vert$ given pulling back the metric from $\mathcal{V}$.
By~\cite[Theorem 11.16]{AbramenkoBrown08} we can endow $\vert \Delta(\widetilde{\mathcal{B}},\mathcal{N}) \vert$ with a metric $d_{\Delta(\widetilde{\mathcal{B}},\mathcal{N})}$ by setting $d_{\Delta(\widetilde{\mathcal{B}},\mathcal{N})}(x,y) = d(\vert f_{\Sigma} \vert(x),\vert f_{\Sigma} \vert(y))$ for $x,y \in \vert \Delta(\widetilde{\mathcal{B}},\mathcal{N}) \vert$, where $\Sigma$ is an apartment containing $x,y$ and $f_{\Sigma} \colon \Sigma \rightarrow \Sigma(\widetilde{\mathcal{B}},\mathcal{N})$ is an isomorphism.
Using these identifications, we can view $\sigma$ as a chamber in $\pinf(\Delta(\widetilde{\mathcal{B}},\mathcal{N}))$.
In particular we can consider the stabilizer of $\sigma$ in $\mathcal{G}$, which is given by the following result, see~\cite[Theorem 14.46]{Weiss09}.

\begin{Citation}\label{cit:B-is-the-stab-of-ch-at-infty}
The group $\mathcal{B}$ is the stabilizer of $\sigma$ in $\mathcal{G}$.
\end{Citation}

\subsection{The case of $p$-adic valuations}\label{subsec:p-adic-value}
Suppose from now on that $K = \Q$.
For each prime number $p \in \N$, let $v_p$ be the $p$-adic valuation on $\Q$.
As discussed in the previous subsections, this leads to  the construction of a spherical $\BN$-pair $(\mathcal{B},\mathcal{N})$ and a Euclidean $\BN$-pair $(\widetilde{\mathcal{B}}_p,\mathcal{N})$.
Let $X_p = \Delta(\widetilde{\mathcal{B}}_p,\mathcal{N})$ denote the Euclidean building associated to $(\widetilde{\mathcal{B}}_p,\mathcal{N})$ and let $\Sigma_p = \Sigma(\widetilde{\mathcal{B}}_p,\mathcal{N})$.
It will be important for us to view $\mathcal{G}$ and its subgroups as topological groups.
To this end we fix a basis $\{ b_1, \ldots, b_d \}$ of our lattice $M$ and identify $\mathcal{G}$ with its matrix representation in $\GL_d(\Q)$ that corresponds to this basis.
Recall that $v_p$ induces a $p$-adic absolute value on $\Q$ given by $\vert x \vert_p = p^{-v_p(x)}$.
Let us endow $\Q$ with the $p$-adic topology coming from $\abs{\cdot}_p$.
Moreover we endow the space $M_d(\Q)$ of $d \times d$-matrices with rational coefficients with the product topology.
It can be easily seen that the subspace topology turns $\mathcal{G} \subset M_d(\Q)$ and its subgroups into topological groups, which will be denoted by $\mathcal{G}_p$, $\mathcal{N}_p$, $\mathcal{U}_p$, etc.

\begin{notation}\label{not:R-points}
Let $R \subseteq \Q$ be a unital subring and let $M_d(R)$ denote the set of $d \times d$-matrices with coefficients in $R$.
For every subgroup $H \leq \mathcal{G}$ we write $H(R) = H \cap M_d(R)$.
\end{notation}

It follows from~\cite[Lemma 11]{Steinberg68} that $\mathcal{G}$ lies in $\SL_d(\Q)$.
Thus the inverse of a matrix $A \in \mathcal{G}$ coincides with its adjoint matrix $A^{adj}$.
As a consequence, we see that $H(R)$ is a group for every unital subring $R \subseteq \Q$ and every subgroup $H \leq \mathcal{G}$.

\begin{definition}\label{def:the-subrings}
Let $S \subseteq \N$ be a finite set of prime numbers.
We write $A_S$ to denote the subring of $\Q$ consisting of elements $x$ with $v_p(x) \geq 0$ for every $p \in S$.
The subring of $\Q$ consisting of elements $x$ with $v_p(x) \geq 0$ for every prime $p$ that is \emph{not} contained in $S$ will be denoted by $\mathcal{O}_S$.
If $S = \{p\}$ is a singleton, then we will just write $A_p$ and $\mathcal{O}_p$.
\end{definition}

Note that $\mathcal{O}_S$ coincides with $\Z[1/N]$, where $N = \prod \limits_{p \in S} p$.

\begin{lemma}\label{lem:structure-of-K-points}
Let $R \subseteq \Q$ be a subring of the form $A_S$ or $\mathcal{O}_S$.
Then
\begin{enumerate}
\item[$(a)$] $\mathcal{G}(R)$ is generated by $\Set{x_{\alpha}(t)}{\alpha \in \Phi,\ t \in R}$,
\item[$(b)$] $\mathcal{B}(R) = \mathcal{U}(R)\mathcal{T}(R)$,
\item[$(c)$] $\mathcal{U}(R) = \Set{\prod \limits_{\alpha \in \Phi^{+}} x_{\alpha}(t_{\alpha})}{t_{\alpha} \in R}$, and
\item[$(d)$] $\mathcal{T}(R) = \Set{\prod \limits_{i = 1}^{\ell} h_{\alpha_i}(t_{\alpha_i})}{t_{\alpha_i} \in R^{\times}}$.
\end{enumerate}
\end{lemma}
\begin{proof}
The statements $(b),(c)$ and $(d)$ are covered by~\cite[Lemma 49]{Steinberg68}.
Moreover~\cite[Corollary 3]{Steinberg68} tells us that $(a)$ holds if $R$ is a Euclidean domain.
Thus the claim follows from~\cite[Theorem 3.33]{Clark15}), which says that localizations of Euclidean domains are Euclidean.
\end{proof}

\begin{lemma}\label{lem:translation}
For every root $\alpha \in \Phi$ the element $t_{\alpha, k} \defeq s_{-\alpha,k} s_{\alpha} \in W_{\widetilde{\Phi}}$
acts on $\mathcal{V}$ by $t_{\alpha, k}(v)= v-k \alpha^{V}$.
\end{lemma}
\begin{proof}
Recall that $s_{\alpha}(v) = v - \kappa^{\ast}(v,\alpha) \alpha^{V}$ and that $s_{\alpha,k}(v) = s_{\alpha}(v) + k \alpha^{V}$ for every $v \in \mathcal{V}$.
Thus
\[
s_{-\alpha,k}(s_{\alpha}(v)) =
s_{-\alpha}(s_{\alpha}(v)) - k \alpha^{V} =
v - k \alpha^{V},
\]
which proves the lemma.
\end{proof}

Recall that $\mathcal{\widetilde{T}}_p$ denotes the kernel of $\pi \colon \mathcal{N}_p \rightarrow W_{\widetilde{\Phi}}$.
The following lemma provides us with a more explicit description of $\mathcal{\widetilde{T}}_p$.

\begin{lemma}\label{lem:the-structure-of-H}
Let $p \in \N$ be a prime.
Then $\mathcal{\widetilde{T}}_p = \Set{\prod \limits_{i=1}^{\ell} h_{\alpha_i}(t_i)}{t_i \in A_p^{\times}}$.
Moreover, we have $\pi(h_{\alpha_i}(t_i)) = t_{\alpha_i,k_i}$ with $k_i = v_p(t_i)$ for each $1 \leq i \leq \ell$.
\end{lemma}
\begin{proof}
By definition, the group $\mathcal{\widetilde{T}}_p$ fixes $\mathcal{V}$ pointwise.
In particular $\mathcal{\widetilde{T}}_p$ fixes $\pinf \mathcal{V}$ pointwise.
By Citation~\ref{cit:B-is-the-stab-of-ch-at-infty}, we therefore obtain $\mathcal{\widetilde{T}}_p \subseteq \mathcal{B} \cap \mathcal{N} = \mathcal{T}$.
In this case Citation~\ref{cit:structure-of-T} tells us that
every element in $\mathcal{\widetilde{T}}_p$ can be written as $\prod \limits_{i=1}^{\ell} h_{\alpha_i}(t_i)$ for appropriate $t_i \in \Q^{\times}$.
The image of such a generator $h_{\alpha_i}(t_i)$ under $\pi$ is given by
\begin{align*}
\pi(h_{\alpha_i}(t_i))
& = \pi(w_{\alpha_i}(t_i) w_{\alpha_i}(1)^{-1}) \\
& = \pi(w_{\alpha_i}(t_i))\pi(w_{\alpha_i}(1)^{-1}) \\
& = \pi(m_{-\alpha_i}(-t_{i}^{-1})) \pi(m_{-\alpha_i}(-1))\\
& = s_{-\alpha_i,-v_p(-t_{i}^{-1})} s_{-\alpha_i,-v_p(-1)}\\
& = s_{-\alpha_i,v_p(t_{i})} s_{-\alpha_i,0} \\
& = s_{-\alpha_i,v_p(t_{i})} s_{\alpha_i}.
\end{align*}
By setting $k_i \defeq v_p(t_i)$, Lemma~\ref{lem:translation} implies $\pi(h_{\alpha_i}(t_i)) = t_{\alpha_i,k_i}$.
Suppose that $w = \prod \limits_{i=1}^{\ell} h_{\alpha_i}(t_i)$ lies in $\mathcal{\widetilde{T}}_p$, i.e.\ $\pi(w) = \id \in W_{\widetilde{\Phi}}$.
Then we deduce from the above chain of equalities that
\[
v
= w(v)
= (\prod \limits_{i=1}^{\ell} \pi(h_{\alpha_i}(t_i)))(v)
= (\prod \limits_{i=1}^{\ell} t_{\alpha_i,k_i})(v)
= v - \sum \limits_{i=1}^{\ell} v_p(t_i) \alpha_i^{V}.
\]
Since $\Delta$ is a basis of $\mathcal{V}$, this implies $v_p(t_i) = 0$ for every $1 \leq i \leq \ell$.
Now the lemma follows from the fact that an element $t \in \Q$ is a unit in $A_p$ if and only if $v_p(t) = 0$.
\end{proof}

In view of the next result, we see that $A_p$ naturally appears in the study of $\mathcal{G}$ and its action on $X_p$.

\begin{lemma}\label{lem:open-stabilizer-in-U}
The intersection $\widetilde{\mathcal{B}}_p \cap \mathcal{U}$ coincides with $\mathcal{U}(A_p)$.
\end{lemma}
\begin{proof}
Recall that $\widetilde{\mathcal{B}}_p = \widetilde{\mathcal{T}}_p\widetilde{\mathcal{U}}_p$, where $\widetilde{\mathcal{U}}_p$ is generated by the sets
\[
\Set{x_{\alpha}(t)}{\alpha \in \Phi^{+},\ t \in A_p} \text{ and } \Set{x_{\alpha}(t)}{\alpha \in \Phi^{-},\ t \in p A_p},
\]
and where $\widetilde{\mathcal{T}}_p = \Set{\prod \limits_{i=1}^{\ell} h_{\alpha_i}(t_i)}{t_i \in A_p^{\times}}$ by Lemma~\ref{lem:the-structure-of-H}.
In view of Lemma~\ref{lem:structure-of-K-points}$(c)$, this shows that $\mathcal{U}(A_p)$ is contained in $\widetilde{\mathcal{B}}_p \cap \mathcal{U}_p$.
On the other hand, we have $\widetilde{\mathcal{T}}_p,\ \widetilde{\mathcal{U}}_p \subseteq \mathcal{G}(A_p)$, which implies $\widetilde{\mathcal{B}}_p \cap \mathcal{U} \subseteq \mathcal{U}(A_p)$.
\end{proof}

\section{$\Sigma$-invariants of S-arithmetic Borel groups}\label{sec:final}

We keep the notation from the previous section.
That is, we consider a Chevalley group $\mathcal{G} = \mathcal{G}(\Phi,\rho,\Q)$, a system of simple roots $\Delta = \{ \alpha_1, \ldots, \alpha_{\ell}\}$ and the corresponding Borel subgroup $\mathcal{B} \subseteq \mathcal{G}$.
Moreover we fix a finite set of prime numbers $S \subseteq \N$.

\subsection{Cocompactness and finiteness properties of stabilizers}

Our goal is to determine the $\Sigma$-invariants of $\Gamma \defeq \mathcal{B}(\mathcal{O}_S)$.
To this end, we will prove that the diagonal action of $\Gamma$ on $X_S \defeq \prod \limits_{p \in S} X_p$, where $X_p = \Delta(\widetilde{\mathcal{B}}_p,\mathcal{N})$, satisfies the conditions of Theorem~\ref{thm:sigmas-and-stabilizers}.
That is, we have to show that $\Gamma$ acts cocompactly on $X_S$ and that the cell stabilizers with respect to this action are of type $F_{\infty}$.

\begin{lemma}\label{lem:specific-stab-id-F-infty}
The group $\bigcap \limits_{p \in S} \widetilde{\mathcal{B}}_p \cap \Gamma$ is of type $F_{\infty}$.
\end{lemma}
\begin{proof}
Recall from the proof of Lemma~\ref{lem:open-stabilizer-in-U} that $\widetilde{\mathcal{B}}_p \subseteq \SL_d(A_p)$.
Since a number $x \in \Q$ lies in $\Z$ if and only if $v_p(x) \geq 0$ for all primes $p$, we see that
\begin{equation}\label{eq:specific-stab-id-F-infty}
\bigcap \limits_{p \in S} \widetilde{\mathcal{B}}_p \cap \Gamma \subseteq \mathcal{B}(\bigcap \limits_{p \in S} A_p \cap \mathcal{O}_S) = \mathcal{B}(\Z) = \mathcal{T}(\Z)\mathcal{U}(\Z),
\end{equation}
where the last equality follows from Lemma~\ref{lem:structure-of-K-points}.
Moreover, we know from Lemma~\ref{lem:structure-of-K-points} that $\mathcal{T}(\Z)$ consists of elements of the form $\prod \limits_{i = 1}^{\ell} h_{\alpha_i}(t_{\alpha_i})$ with $t_{\alpha_i} \in \Z^{\times}$, which shows that $\mathcal{T}(\Z)$ is a finite group.
On the other hand, it follows from Citation~\ref{cit:relation-in-G} $(a)$ and $(b)$ that $\mathcal{U}(\Z)$ is finitely generated and nilpotent.
Together with~\eqref{eq:specific-stab-id-F-infty} this implies that $\bigcap \limits_{p \in S} \widetilde{\mathcal{B}}_p \cap \Gamma$
is a finitely generated virtually nilpotent group.
As such it is well-known to be of type $F_{\infty}$.
Indeed, this follows from~\cite[Exercise 7.2.1]{Geoghegan08} and the obvious fact that finitely generated abelian groups are of type $F_{\infty}$.
\end{proof}

We are now ready to determine the finiteness properties of the cell stabilizers of the diagonal action of $\Gamma$ on $X_S$.

\begin{proposition}\label{prop:stabilizer-of-type-F-infty}
Given a non-empty cell $A$ in $X_S$, its stabilizer $\St_{\Gamma}(A)$ is of type $F_{\infty}$.
\end{proposition}
\begin{proof}
From the $\BN$-characterization of $X_p$ we see that $\widetilde{\mathcal{B}}_p$ is the stabilizer of a chamber in $X_p$ under the action on $\mathcal{G}$.
Thus the corresponding stabilizer of the diagonal action of $\Gamma$ on $X_S$ is given by $\bigcap \limits_{p \in S} \widetilde{\mathcal{B}}_p \cap \Gamma$, which is of type $F_{\infty}$ by Lemma~\ref{lem:specific-stab-id-F-infty}.
Since $X_S$ is locally finite, it follows that all proper cell stabilizers are commensurable.
Now the claim follows from~\cite[Corollary 9]{Alonso95}, which tells us that being of type $F_n$ is an invariant under commensurability.
\end{proof}

Let us next show that the diagonal action of $\Gamma$ on $X_S$ is cocompact.
From the $\BN$-characterization we already know that $\mathcal{G}$ acts cocompactly on $X_p$ for every $p \in S$.
We want to show that this is still the case if we restrict the action to $\mathcal{B}$.
Recall from Subsection~\ref{subsec:a-metric-on-X} that we have identified the standard apartment $\Sigma_p = \Sigma(\widetilde{\mathcal{B}}_p,\mathcal{N})$ in $X_p$ with the Euclidean vector space $(\mathcal{V},\kappa^{\ast})$.
Let $o_p \in \Sigma_p$ denote the point that corresponds to the origin in $\mathcal{V}$ under this identification and let $E_p$ denote the chamber in $\Sigma_p$ that consists of the points $v \in \Sigma_p$ with $\kappa^{\ast}(v, \widetilde{\alpha}) < 1$ and $\kappa^{\ast}(v, \alpha_i) > 0$ for every $1 \leq i \leq \ell$.

\begin{lemma}\label{lem:stab-of-center}
The group $\mathcal{G}(A_p)$ coincides with the stabilizer of $o_p$ in $\mathcal{G}$.
\end{lemma}
\begin{proof}
We start by proving the inclusion $\mathcal{G}(A_p) \subseteq \St_{\mathcal{G}}(o_p)$.
By Lemma~\ref{lem:structure-of-K-points}, $\mathcal{G}(A_p)$ is generated by all elements of the form $x_{\alpha}(t)$ with $\alpha \in \Phi$ and $t \in A_p$.
From the identification $\Sigma_p \cong \mathcal{V}$ in Subsection~\ref{subsec:a-metric-on-X} and the $\BN$-characterization of $\Sigma_p$ we see that $\widetilde{\mathcal{B}}_p$ is the pointwise stabilizer of $E_p$.
In particular, it follows that each generator $x_{\alpha}(t) \in \widetilde{\mathcal{B}}_p$ with $\alpha \in \Phi^{+}$ and $t \in A_p$ fixes $o_p$.
On the other hand, we know from Citation~\ref{cit:BN-pair-from-VRGD} that $o_p$ is fixed by $\omega_{\alpha}$ for every $\alpha \in \Phi$.
Using Citation~\ref{cit:relation-in-G}$(d)$, it follows that $\omega_{\alpha} x_{\beta}(t) \omega_{\alpha}^{-1} = x_{s_{\alpha}(\beta)}(\pm t)$ fixes $o_p$ for every $\alpha \in \Phi^{+}$ and every $t \in A_p$.
Thus the inclusion $\mathcal{G}(A_p) \subseteq \St_{\mathcal{G}}(o_p)$ follows from the basic equality $s_{\widetilde{\alpha}}(\Phi^{-}) = \Phi^{+}$.
To see the reverse inclusion, recall that $\mathcal{G}(A_p)$ is a standard parabolic subgroup by Remark~\ref{rem:independence-of-S}.
Since we have just shown that $\mathcal{G}(A_p)$ contains the maximal parabolic subgroup $P_0$, which is the stabilizer of $o_p$ in $\mathcal{G}$, the lemma follows from the fact that $\mathcal{G}(A_p)$ is a proper subgroup of $\mathcal{G}$.
\end{proof}

Recall from Subsection~\ref{subsec:p-adic-value} that $\mathcal{G}_p$, $\mathcal{B}_p$, $\mathcal{U}_p$, etc.\ denote the corresponding groups $\mathcal{G}$, $\mathcal{B}$, $\mathcal{U}$ etc.\ associated with the $p$-adic topology.

\begin{corollary}\label{cor:open-stabilizers}
Given a cell $C \subset X_S$, its stabilizer under the action of $\prod \limits_{p \in S} \mathcal{U}_p$ is open.
\end{corollary}
\begin{proof}
Let $p \in S$.
Since $A_p \subset \Q$ is open with respect to the $p$-adic topology, it follows from Lemma~\ref{lem:stab-of-center} that the stabilizer of $o_p$ is open in $\mathcal{G}_p$.
Using the cocompactness of the action of $\mathcal{G}_p$ on $X_p$, we can therefore deduce that the stabilizer of each cell $A \subset X_p$ is open in $\mathcal{G}_p$.
Thus the cell stabilizer $\St_{\mathcal{U}_p}(A) = \mathcal{U}_p \cap \St_{\mathcal{G}_p}(A)$ is open in $\mathcal{U}_p$.
Since every cell $C \subset X_S$ is a product $\prod \limits_{p \in S} C_p$ of cells $C_p \subset X_p$, we see that the stabilizer of $C$ in $\prod \limits_{p \in S} \mathcal{U}_p$, being a product of open subgroups, is open.
\end{proof}

\begin{lemma}\label{lem:fund-dom-for-U}
The action of $\mathcal{U}$ on $X_p$ satisfies $X_p = \bigcup \limits_{g \in \mathcal{U}} g \cdot \Sigma_p$.
\end{lemma}
\begin{proof}
From~\cite[Theorem 4]{Steinberg68} we know that $\mathcal{G} = \bigcup \limits_{w \in W_{\Phi}} \mathcal{B} \widetilde{w} \mathcal{B}$, where $\widetilde{w} \in \mathcal{N}$ is a representative of $w \in W_{\Phi} \cong \mathcal{N} / \mathcal{T}$.
In this case~\cite[Proposition 11.100]{AbramenkoBrown08} tells us that $X_p = \bigcup \limits_{g \in \mathcal{B}} g \cdot \Sigma_p$.
Recall that $\mathcal{B}$ is the semidirect product of $\mathcal{U}$ and $\mathcal{T}$.
Since $\mathcal{T}$ stabilizes $\Sigma_p$, we therefore obtain $X_p = \bigcup \limits_{g \in \mathcal{U}} g \cdot \Sigma_p$.
\end{proof}

We want to extend Lemma~\ref{lem:fund-dom-for-U} to the diagonal action of $\mathcal{U}(\mathcal{O}_S)$ on $X_S$.
To this end, we consider the orbit of the apartment $\Sigma_S \defeq \prod \limits_{p \in S} \Sigma_p$ under this action.

\begin{lemma}\label{lem:R-pts-in-U-have-dom-in-V}
We have $X_S = \bigcup \limits_{g \in \mathcal{U}(\mathcal{O}_S)} g \cdot \Sigma_S$, where $\mathcal{U}(\mathcal{O}_S)$ acts diagonally on $X_S$.
\end{lemma}
\begin{proof}
From Lemma~\ref{lem:fund-dom-for-U} it follows that $X_S$ is the union of apartments $g \cdot \Sigma_S$ with $g \in \prod \limits_{p \in S} \mathcal{U}_p$.
Moreover, we know from Corollary~\ref{cor:open-stabilizers} that the cell stabilizers of the action of $\prod \limits_{p \in S} \mathcal{U}_p$ on $X_S$ are open.
On the other hand, it is a direct consequence of the proof of~\cite[Theorem 20]{Steinberg68} that the diagonal embedding $\mathcal{U}(\mathcal{O}_S) \rightarrow \prod \limits_{p \in S} \mathcal{U}_p$ has dense image.
Together this shows that $\mathcal{U}(\mathcal{O}_S)$ and $\prod \limits_{p \in S} \mathcal{U}_p$ have the same orbits of cells, which proves the lemma.
\end{proof}

\begin{lemma}\label{lem:cocompactness-torus}
The diagonal action of $\mathcal{T}(\mathcal{O}_S)$ on $\Sigma_S$ is cocompact.
\end{lemma}
\begin{proof}
From Lemma~\ref{lem:the-structure-of-H} we know that  $\pi(h_{\alpha_i}(t_i)) = t_{\alpha_i,k_i}$, where $k_i = v_p(t_i)$ and $t_{\alpha_i,k_i}(v) = v-k_i \alpha_i^{V}$.
Since $\Delta$ is a basis of $\mathcal{V} \cong \Sigma_p$, it follows that $\mathcal{T}(\mathcal{O}_p)$ acts cocompactly on $\Sigma_p$, while $\mathcal{T}(A_p)$ fixes $\Sigma_p$ pointwise.
Now the claim follows from the multiplicativity of each $h_{\alpha_i}$ as stated in Citation~\ref{cit:structure-of-T}.
\end{proof}

\begin{proposition}\label{prop:cocomp-of-borel-over-R}
The diagonal action of $\Gamma$ on $X_S$ is cocompact.
\end{proposition}
\begin{proof}
From Lemma~\ref{lem:structure-of-K-points} we know that $\Gamma = \mathcal{T}(\mathcal{O}_S) \mathcal{U}(\mathcal{O}_S)$.
Hence the claim is a direct consequence of Lemma~\ref{lem:R-pts-in-U-have-dom-in-V} and Lemma~\ref{lem:cocompactness-torus}.
\end{proof}

\subsection{The structure of the character sphere}

In order to compute the $\Sigma$-invariants of $\Gamma = \mathcal{B}(\mathcal{O}_S)$, we first have to describe its character sphere.
To this end, it will be useful to understand the relationship between $\mathcal{U}(\mathcal{O}_S)$ and the commutator subgroup of $\Gamma$.

\begin{lemma}\label{lem:the-character-sphere}
The commutator subgroup $[\Gamma,\Gamma]$ lies in $\mathcal{U}(\mathcal{O}_S)$.
The quotient $Q \defeq \mathcal{U}(\mathcal{O}_S) / [\Gamma,\Gamma]$ is a torsion group.
\end{lemma}
\begin{proof}
Recall from Lemma~\ref{lem:structure-of-K-points} that $\Gamma = \mathcal{T}(\mathcal{O}_S)\mathcal{U}(\mathcal{O}_S)$.
Thus $[\Gamma,\Gamma] \subseteq \mathcal{U}(\mathcal{O}_S)$ since $\mathcal{T}(\mathcal{O}_S)$ is an abelian group by Citation~\ref{cit:structure-of-T}.
For the second claim we use Citation~\ref{cit:relation-in-G}$(e)$, which tells us that
\[
h_{\alpha}(t)x_{\beta}(s)h_{\alpha}(t)^{-1} = x_{\beta}(t^{\langle \beta,\alpha \rangle}s)
\]
for all $s \in \Q$, $t \in \Q^{\ast}$ and $\alpha,\beta \in \Phi$.
For each $p \in S$ this gives us
\begin{center}
\begin{tabular}{lll}
$[h_{\alpha}(p),x_{\alpha}(s)]$ &
$=h_{\alpha}(p)x_{\alpha}(s)h_{\alpha}(p)^{-1}x_{\alpha}(s)^{-1}$ &
$=x_{\alpha}(p^{2}s)x_{\alpha}(s)^{-1}$ \\[2ex]&
$=x_{\alpha}(p^{2}s-s)$ &
$=x_{\alpha}(s)^{(p^{2}-1)}$.
\end{tabular}
\end{center}
Thus $x_{\alpha}(s)^{(p^{2}-1)} \in [\Gamma,\Gamma]$ for every $\alpha \in \Phi^{+}$ and every $s \in \mathcal{O}_S$.
Let $q \defeq p^{2}-1$.
Consider an element $u \in \mathcal{U}(\mathcal{O}_S)$.
By Lemma~\ref{lem:structure-of-K-points} we have $u = \prod \limits_{\alpha \in \Phi^{+}} x_{\alpha}(t_{\alpha})$ for appropriate $\alpha \in \Phi^{+}$ and $t_{\alpha} \in \mathcal{O}_S$.
Let $\overline{u}$ be the image of $u$ in $Q$.
Since $Q$ is abelian, we deduce from Citation~\ref{cit:relation-in-G}$(a)$ that $\overline{u^{q}} = \prod \limits_{\alpha \in \Phi^{+}} \overline{x_{\alpha}(t_{\alpha}^{q})}$.
Now the claim follows as we have just seen that $x_{\alpha}(t_{\alpha}^{q}) \in [\Gamma,\Gamma]$.
\end{proof}

The following observation reduces the study of the characters of $\Gamma$ to those of $\mathcal{T}(\mathcal{O}_S)$.
To formulate it, let $\delta \colon \Gamma \rightarrow \mathcal{T}(\mathcal{O}_S)$ denote the canonical projection that arises from the splitting $\Gamma = \mathcal{T}(\mathcal{O}_S) \mathcal{U}(\mathcal{O}_S)$.

\begin{proposition}\label{prop:the-character-sphere}
The inclusion $i \colon \mathcal{T}(\mathcal{O}_S) \rightarrow \Gamma$ induces an isomorphism
\[
i^{\ast} \colon \Hom(\Gamma,\R) \rightarrow \Hom(\mathcal{T}(\mathcal{O}_S)),\ \chi \mapsto \chi \circ i.
\]
The inverse isomorphism is given by
\[
\delta^{\ast} \colon \Hom(\mathcal{T}(\mathcal{O}_S),\R) \rightarrow \Hom(\Gamma,\R),\ \chi \mapsto \chi \circ \delta.
\]
\end{proposition}
\begin{proof}
Note that $\delta \circ i$ is the identity on $\mathcal{T}(\mathcal{O}_S)$.
It therefore remains to verify that $\chi = \delta^{\ast}(i^{\ast}(\chi)) = \chi \circ i \circ \delta$ for every $\chi \in \Hom(\Gamma,\R)$.
To see this, let $\gamma \in \Gamma$.
By Lemma~\ref{lem:structure-of-K-points} we have $\gamma = tu$ for appropriate $t \in \mathcal{T}(\mathcal{O}_S)$ and $u \in \mathcal{U}(\mathcal{O}_S)$.
From Lemma~\ref{lem:the-character-sphere} we know that a power of $u$ lies in $[\Gamma,\Gamma]$.
Thus $\chi(u)=0$, from which we deduce that $\chi(g) = \chi(t)$.
On the other hand, we have $\delta(tu) = t$ by definition of $\delta$.
Together this shows that $\chi(g) = \chi(t) = \chi \circ i \circ \delta(g)$, which proves the claim.
\end{proof}

\begin{notation}\label{not:coordinates-of-torus}
Given $\alpha \in \Delta$ and $p \in S$, let $\mathcal{T}_{\alpha,p}$ denote the subgroup of $\mathcal{T}(\mathcal{O}_S)$ consisting of elements of the form $h_{\alpha}(p^n)$ with $n \in \Z$.
The group generated by all groups $\mathcal{T}_{\alpha,p}$ with $\alpha \in \Delta$ and $p \in S$ will be denoted by $\mathcal{T}_S$.
\end{notation}

The following lemma describes the structure of $\mathcal{T}_S$.

\begin{lemma}\label{lem:decomp-torus-into-components}
$\mathcal{T}_S$ is a free abelian group.
A basis for $\mathcal{T}_S$ is given by $\Set{h_{\alpha}(p)}{\alpha \in \Delta,\ p \in S}$.
In particular, we have $\mathcal{T}_S \cong \bigoplus \limits_{p \in S} \bigoplus \limits_{\alpha \in \Delta} \mathcal{T}_{\alpha,p}$.
\end{lemma}
\begin{proof}
Suppose that $\gamma \defeq \prod \limits_{(p,\alpha) \in S \times \Delta} h_{\alpha}(p^{n_{\alpha,p}})$ is trivial in $\mathcal{T}_S$.
Then $\gamma$ acts as the identity on $\Sigma_p$ for each $p \in S$.
Recall from Lemma~\ref{lem:translation} and Lemma~\ref{lem:the-structure-of-H} that $\gamma$ acts on $\Sigma_p \cong \mathcal{V}$ by
\[
\gamma(x)
= \prod \limits_{\alpha \in \Delta} \pi(h_{\alpha}(p^{n_{\alpha,p}}))(x)
= \prod \limits_{\alpha \in \Delta} t_{\alpha,n_{\alpha,p}}(x)
= x - \sum \limits_{\alpha \in \Delta} n_{\alpha,p} \alpha^{V}.
\]
Since $\Delta$ is a basis of $\mathcal{V}$, it follows that $n_{\alpha,p} = 0$ for every $\alpha \in \Delta$ and every $p \in S$.
\end{proof}

\begin{lemma}\label{lem:splitting-the-torus}
Every element in $\mathcal{T}(\mathcal{O}_S)$ can be uniquely written as $t_{\varepsilon} t_S$, where $t_{\varepsilon}$ has finite order and $t_S \in \mathcal{T}_S$.
\end{lemma}
\begin{proof}
From Lemma~\ref{lem:structure-of-K-points} we know that every $\gamma \in \mathcal{T}(\mathcal{O}_S)$ can be written as $\gamma = \prod \limits_{\alpha \in \Delta} h_{\alpha}(t_{\alpha})$ with $t_{\alpha} \in \mathcal{O}_S^{\times}$.
Thus $t_{\alpha} = \varepsilon_{\alpha} \prod \limits_{p \in S} p^{k_{\alpha}}$ for appropriate $k_{\alpha} \in \Z$ and $\varepsilon_{\alpha} \in \{\pm 1\}$.
This gives us
\[
\gamma = \prod \limits_{\alpha \in \Delta} h_{\alpha}(t_{\alpha})
= \prod \limits_{\alpha \in \Delta} h_{\alpha}(\varepsilon_{\alpha}) \prod \limits_{p \in S} h_{\alpha}(p^{k_{p,\alpha}}).
\]
Now we obtain the predicted decomposition of $\gamma$ by setting
\[
t_{\varepsilon} \defeq \prod \limits_{\alpha \in \Delta} h_{\alpha}(\varepsilon_{\alpha}) \text{ and
} t_S \defeq \prod \limits_{\alpha \in \Delta} \prod \limits_{p \in S} h_{\alpha}(p^{k_{p,\alpha}}).
\]
For the uniqueness we observe that if $\gamma = t_{\varepsilon}t_S = t_{\varepsilon}'t_S'$ are two such decompositions, then $t_S'^{-1}t_S = t_{\varepsilon}'t_{\varepsilon}^{-1}$ has finite order.
Since $\mathcal{T}_S$ is torsion-free by Lemma~\ref{lem:decomp-torus-into-components}, this implies $t_S = t_S'$ and $t_{\varepsilon} = t_{\varepsilon}'$.
\end{proof}

\begin{corollary}\label{cor:res-to-tors-free-part-of-torus}
The inclusion $\iota \colon \mathcal{T}_S \rightarrow \Gamma$ induces an isomorphism
\[
\iota^{\ast} \colon \Hom(\Gamma,\R) \rightarrow \Hom(\mathcal{T}_S,\R),\ \chi \mapsto \chi \circ \iota.
\]
\end{corollary}
\begin{proof}
From Lemma~\ref{lem:splitting-the-torus} it follows that the inclusion $\mathcal{T}_S \rightarrow \mathcal{T}(\mathcal{O_S})$ induces an isomorphism
\[
\Hom(\mathcal{T}(\mathcal{O_S}),\R) \rightarrow \Hom(\mathcal{T}_S,\R).
\]
Now the claim follows from Proposition~\ref{prop:the-character-sphere}.
\end{proof}

Using Corollary~\ref{cor:res-to-tors-free-part-of-torus}, we are now able to define a basis of $\Hom(\Gamma,\R)$.

\begin{definition}\label{def:valuation-at-one-coordinate}
For every $\alpha \in \Delta$ and every $p \in S$ let $\chi_{\alpha,p}' \colon \mathcal{T}_S \rightarrow \R$ be the map induced by $h_{\beta}(t) \mapsto \langle \beta, \alpha \rangle v_p(t)$ for every $\beta \in \Delta$ and $t \in \Q^{\times}$.
The unique extension of $\chi_{\alpha,p}'$ to $\Gamma$, provided by Corollary~\ref{cor:res-to-tors-free-part-of-torus}, will be denoted by $\chi_{\alpha,p}$.
Let $B_{\mathcal{G},\mathcal{B}}(S) = \Set{\chi_{\alpha,p}}{\alpha \in \Delta,\ p \in S}$ denote the union of these characters.
\end{definition}


\begin{proposition}\label{prop:basis-of-charactersphere}
The set $B_{\mathcal{G},\mathcal{B}}(S)$ is a basis of $\Hom(\Gamma,\R)$.
\end{proposition}
\begin{proof}
This follows from the fact that $\kappa^{\ast}$ is non-degenerate and that $\Delta$ is a basis of $\mathcal{V}$.
\end{proof}

\subsection{Extending characters to height functions}\label{subsec:ext-char}

Our next goal is to construct equivariant height functions for the action of $\Gamma$ on $X_S$ .
That is, for a given character $\chi \colon \Gamma \rightarrow \R$, we want to define a continuous function $h \colon X_S \rightarrow \R$ such that $h(g(x)) = \chi(g) + h(x)$ for all $g \in \Gamma$ and $x \in X_S$.
Recall from Notation~\ref{not:chamber} that $\sigma \subseteq \pinf \mathcal{V}$ denotes the chamber at infinity associated to the sector
\[
\mathcal{K} = \Set{v \in \mathcal{V}}{\kappa^{\ast}(v, \alpha) > 0 \text{ for every }  \alpha \in \Delta}.
\]
Using the identification $\Sigma_p \cong \mathcal{V}$ in Subsection~\ref{subsec:a-metric-on-X}, we define the corresponding sector $K_p \subset \Sigma_p$ and its associated chamber $\sigma_p \subseteq \pinf \Sigma_{p}$ for each $p \in S$.
Note that the product $K_S \defeq \prod \limits_{p \in S} K_p$ is a sector in $\Sigma_S = \prod \limits_{p \in S} \Sigma_p$.
Its associated chamber at infinity will be denoted by $\sigma_S \subseteq \pinf \Sigma_S$.

\begin{notation}\label{not:basis-heights-in-V}
Given $p \in S$ and $\alpha \in \Delta$, we consider the function $\kappa_{\alpha,p} \colon \Sigma_p \rightarrow \R,\ v \mapsto \kappa^{\ast}(\alpha,\iota_p(v))$.
By composing $\kappa_{\alpha,p}$ with the canonical projection $\pr_p \colon \Sigma_S \rightarrow \Sigma_p$ and the retraction $\rho_{\sigma_S,\Sigma_S} \colon X_S \rightarrow \Sigma_S$ and inverting the sign, we obtain the height functions
\begin{equation}\label{eq:hgt-alpha-p}
\hgt_{\alpha,p} \defeq - \kappa_{\alpha,p} \circ \pr_p \circ \rho_{\sigma_S,\Sigma_S} \colon X_S \rightarrow \R.
\end{equation}
\end{notation}

Note that each $\hgt_{\alpha,p}$ restricts to a linear function on $\Sigma_S$ and that $\hgt_{\alpha,p}$ is invariant under precomposing with $\rho_{\Sigma_S,\sigma_S}$.
In view of Remark~\ref{rem:another-char-of-retr-pres}, this tells us that $\hgt_{\alpha,p}$ lies in the real vector space $X_S^{\ast} \defeq \Set{\alpha \circ \rho_{\Sigma_S,\sigma_S}}{\alpha \in \Sigma_S^{\ast}}$.
Note also that $B_{\mathcal{G},\mathcal{B}}^{{ht}}(S) \defeq \Set{\hgt_{\alpha,p}}{\alpha \in \Delta,\ p \in S}$ is a basis of $X_S^{\ast}$ since $\Delta$ is a basis of $\mathcal{V}$.
Thus we can extend the bijection
\[
B_{\mathcal{G},\mathcal{B}}(S) \rightarrow B_{\mathcal{G},\mathcal{B}}^{{ht}}(S),\ \chi_{\alpha,p} \mapsto \hgt_{\alpha,p},
\]
where $B_{\mathcal{G},\mathcal{B}}(S)$ is as in Definition~\ref{def:valuation-at-one-coordinate}, to an isomorphism
\begin{equation}\label{eq:associated-heights}
\hgt \colon \Hom(\Gamma,\R) \rightarrow X_S^{\ast},\ \chi \mapsto \hgt_{\chi}.
\end{equation}


\begin{lemma}\label{lem:equivariance-torus-on-stdapp}
Let $\chi \in \Hom(\Gamma,\R)$ be a character.
For every $\gamma \in \mathcal{T}(\mathcal{O}_S)$ and every $x \in \Sigma_S$ we have $\hgt_{\chi}(\gamma(x)) = \hgt_{\chi}(x)+\chi(\gamma)$.
\end{lemma}
\begin{proof}
By the linearity of $\hgt$ it suffices to prove the statement for the basis characters.
Thus let $\chi = \chi_{\alpha,p}$ for some $\alpha \in \Delta$ and some $p \in S$.
Let $\gamma = \prod \limits_{\beta \in \Delta} h_{\beta}(t_{\beta}) \in \mathcal{T}(\mathcal{O}_S)$ be an arbitrary element and let $x \in \Sigma_S$.
Using the $\mathcal{N}$-equivariant identification $\iota_p \colon \Sigma_p \rightarrow \mathcal{V}$ in Subsection~\ref{subsec:a-metric-on-X}, we may apply Lemma~\ref{lem:translation} and Lemma~\ref{lem:the-structure-of-H} to deduce
\begin{align*}
\hgt_{\chi}(\gamma \cdot x)
&=\hgt_{\alpha,p}((\prod \limits_{\beta \in \Delta} h_{\beta}(t_{\beta})) \cdot x)\\
&=\hgt_{\alpha,p}(x - \sum \limits_{\beta \in \Delta} \sum \limits_{q \in S} v_q(t_{\beta})\iota_q^{-1}(\beta^{V}))\\
&=\hgt_{\alpha,p}(x) - \sum \limits_{\beta \in \Delta} \sum \limits_{q \in S} v_q(t_{\beta}) \hgt_{\alpha,p}( \iota_q^{-1}(\beta^{V}))\\
&=\hgt_{\alpha,p}(x) + \sum \limits_{\beta \in \Delta} \sum \limits_{q \in S} v_q(t_{\beta}) \kappa_{\alpha,p} \circ \pr_p \circ \rho_{\Sigma_S,\sigma_S}( \iota_q^{-1}(\beta^{V}))\\
&=\hgt_{\alpha,p}(x) + \sum \limits_{\beta \in \Delta} v_p(t_{\beta}) \kappa_{\alpha,p}(\iota_p^{-1}(\beta^{V}))\\
&=\hgt_{\alpha,p}(x) + \sum \limits_{\beta \in \Delta} v_p(t_{\beta}) \kappa^{\ast}(\alpha,\beta^{V})\\
&=\hgt_{\alpha,p}(x) + \sum \limits_{\beta \in \Delta} v_p(t_{\beta}) \langle \beta , \alpha \rangle\\
&=\hgt_{\alpha,p}(x) + \sum \limits_{\beta \in \Delta} \chi_{\alpha,p}(h_{\beta}(t_{\beta}))\\
&=\hgt_{\chi}(x)+\chi(\gamma),
\end{align*}
which proves the claim.
\end{proof}

The following lemma summarizes how characters and their associated height functions behave under the maps $\rho \defeq \rho_{\Sigma_S,\sigma_S}$ and $\delta$.

\begin{lemma}\label{lem:relations-char-height-delta}
Let $\chi \in \Hom(\Gamma,\R)$ and let $x \in X_S$.
Consider an element $\gamma \in \Gamma$ and its decomposition $\gamma = t_{\gamma}u_{\gamma}$ into its torus part $t_{\gamma} \in \mathcal{T}(\mathcal{O}_S)$ and its unipotent part $u_{\gamma} \in \mathcal{U}(\mathcal{O}_S)$.
Then
\begin{enumerate}
\item[$(a)$] $\hgt_{\chi}(\rho(x)) = \hgt_{\chi}(x)$,
\item[$(b)$] $\chi(\delta(\gamma))=\chi(\gamma)$,
\item[$(c)$] $\rho(u_{\gamma}(x)) = \rho(x)$,
\item[$(d)$] $\rho(t_{\gamma}(x)) = t_{\gamma}(\rho(x))$, and
\item[$(e)$] $\rho(\gamma(x)) = \delta(\gamma)(\rho(x))$.
\end{enumerate}
\end{lemma}
\begin{proof}
Property $(a)$ follows from our observation that $B_{\mathcal{G},\mathcal{B}}^{{ht}}(S)$ is a basis of $X_S^{\ast}$.
Claim $(b)$ is covered by Proposition~\ref{prop:the-character-sphere}.
To prove $(c)$ and $(d)$, let $x \in X_S$ and let $\Sigma \subset X_S$ be an apartment with $x \in \Sigma$ and $\sigma \subseteq \pinf \Sigma$.
Let $\mathcal{K}_1$ be a common sector of $\Sigma_S$ and $\Sigma$ that corresponds to $\sigma$.
From Citation~\ref{cit:B-is-the-stab-of-ch-at-infty} we know that $\mathcal{U}(\mathcal{O}_S)$ fixes $\sigma_S$.
Thus there is a sector $\mathcal{K}_2$ in $u_{\gamma} \cdot \mathcal{K}_1 \cap \mathcal{K}_1$.
We claim that $u_{\gamma}$ fixes $\mathcal{K}_2$ pointwise.
Suppose the opposite.
Then there is some $t \in \mathcal{T}(\mathcal{O}_S)$ acting non-trivially on $\Sigma_S$ such that $t u_{\gamma}$ fixes a point in $\Sigma_S$.
In this case it follows from Lemma~\ref{lem:stab-of-center} that there is some $b \in \mathcal{B}$ with $b t u_{\gamma} b^{-1} \in \mathcal{B}(A_S)$.
But then we have $t = \delta(b t u_{\gamma} b^{-1}) \in \delta(\mathcal{B}(A_S)) = \mathcal{T}(A_S)$, which is a contradiction since Lemma~\ref{lem:the-structure-of-H} implies that $\mathcal{T}(A_S)$ fixes $\Sigma_S$ pointwise.
Thus $u_{\gamma}$ fixes $\mathcal{K}_2$ pointwise from which we deduce that the isomorphisms $\rho \circ u_{\gamma} \colon \Sigma \rightarrow\Sigma_S$ and $\rho \colon \Sigma \rightarrow\Sigma_S$ fix $\mathcal{K}_2$ pointwise and hence coincide.
In particular this shows $\rho(u_{\gamma}(x)) = \rho(x)$ and we obtain $(c)$.
Similarly, we see that $\rho \circ t_{\gamma}$ and $t_{\gamma} \circ \rho$ coincide on $\mathcal{K}_1$ from which we obtain $(d)$.
By applying the rules $(a)-(d)$ we obtain
\[
\rho(\gamma(x)) = \rho(t_{\gamma}(u_{\gamma}(x))) = t_{\gamma}(\rho(u_{\gamma}(x))) = t_{\gamma}(\rho(x)) = \delta(\gamma)(\rho(x)),
\]
which proves $(e)$.
\end{proof}

We are now ready to prove the desired equivariance of $h_{\chi}$.

\begin{corollary}\label{cor:equivariance-of-height}
Let $\chi \colon \Gamma \rightarrow \R$ be a character.
With the notation above we have $h_{\chi}(\gamma(x)) = \chi(\gamma) + h_{\chi}(x)$ for every $x \in X_S$ and every $\gamma \in \Gamma$.
\end{corollary}
\begin{proof}
In view of Lemma~\ref{lem:equivariance-torus-on-stdapp} and the properties in Lemma~\ref{lem:relations-char-height-delta} we have
\begin{align*}
h_{\chi}(\gamma(x))
&=h_{\chi}(\rho(\gamma(x)))\\
&=h_{\chi}(\delta(\gamma)(\rho(x)))\\
&=h_{\chi}(\rho(x))+\chi(\delta(\gamma))\\
&=h_{\chi}(x)+\chi(\gamma),
\end{align*}
which proves the claim.
\end{proof}

\subsection{Sigma invariants of S-arithmetic Borel groups}
We are now ready to prove the main result of this paper.
To formulate it, we consider the image of $B_{\mathcal{G},\mathcal{B}}(S)$ in $S(\Gamma)$, which will be denote by $\Delta_{\mathcal{G},\mathcal{B}}(S)^{(0)}$.
As notation suggests, we think of $\Delta_{\mathcal{G},\mathcal{B}}(S)^{(0)}$ as the vertex set of a spherical simplex in $S(\Gamma)$.
Accordingly, we write $\Delta_{\mathcal{G},\mathcal{B}}(S)^{(k)}$ to denote the $(k+1)$-convex hull of $\Delta_{\mathcal{G},\mathcal{B}}(S)^{(0)}$, which was introduced in Definition~\ref{def:conv-hull}.
Since $B_{\mathcal{G},\mathcal{B}}(S)$ is a basis of $\Hom(\Gamma,\R)$ by Proposition~\ref{prop:basis-of-charactersphere}, we see that this picture indeed makes sense.

\begin{theorem}\label{thm:B}
Let $\mathcal{G} = \mathcal{G}(\Phi,\rho,\Q)$ be a Chevalley group, let $\Delta \subseteq \Phi$ be a system of simple roots and let $\mathcal{B} \subseteq \mathcal{G}$ be the corresponding Borel subgroup.
Let $S \subseteq \N$ be a finite set of primes. Then the $\Sigma$-invariants of $\Gamma = \mathcal{B}(\mathcal{O}_S)$ satisfy the following.
\begin{enumerate}
\item $\Sigma^{\infty}(\Gamma) = S(\Gamma) \backslash \Delta_{\mathcal{G},\mathcal{B}}(S)$.
\item $\Sigma^k(\Gamma) \subseteq S(\Gamma) \backslash \Delta_{\mathcal{G},\mathcal{B}}(S)^{(k-1)}$ for every $k \in \N$.
\item If $\Phi$ is of type $A_{n+1}$, $C_{n+1}$, or $D_{n+1}$ and every $p \in S$ satisfies $p \geq 2^{n}$ in the $A_{n+1}$-case, respectively $p \geq 2^{2n+1}$ in the other two cases, then $\Sigma^k(\Gamma) = S(\Gamma) \backslash \Delta_{\mathcal{G},\mathcal{B}}(S)^{(k-1)}$ for every $k \in \N$.
\end{enumerate}
\end{theorem}
\begin{proof}
Let $\chi \colon \Gamma \rightarrow \R$ be a character and let $\hgt_{\chi} \colon X_S \rightarrow \R$ be its associated height function as given in~\eqref{eq:associated-heights}.
From Corollary~\ref{cor:equivariance-of-height} we know that $\hgt_{\chi}$ is an equivariant extension of $\chi$.
Since $\Gamma$ acts cocompactly on $X_S$ by Proposition~\ref{prop:cocomp-of-borel-over-R} and has cell stabilizers of type $F_{\infty}$ by Proposition~\ref{prop:stabilizer-of-type-F-infty}, we can use Theorem~\ref{thm:sigmas-and-stabilizers} to determine the $\Sigma$-invariants of $\Gamma$.
For each $k \in \N_{0}$ we therefore have $\chi \in \Sigma^{k+1}(\Gamma)$ if and only if $((X_S)_{\hgt_{\chi} \geq r})_{r \in \R}$ is essentially $k$-connected.
We want to apply Theorem~\ref{thm:A-new}, our geometric main result, to compute the essential connectivity properties of $((X_S)_{\hgt_{\chi} \geq r})_{r \in \R}$.
To this end, recall from Lemma~\ref{lem:relations-char-height-delta}, that $\hgt_{\chi}$ is $\rho_{\Sigma_S,\sigma_S}$-invariant, where $\Sigma_S$ and $\sigma_S \subseteq \pinf \Sigma_S$ are as in Section~\ref{subsec:ext-char}.
Let $X_S^{\ast} = (X_S)_{\Sigma_S,0}^{\ast}$, where $0$ is the origin of $\Sigma_S$, and let $\mathcal{B}(\sigma_S) \subseteq S(X_S^{\ast})$ denote the set of classes $\alpha_P$ of functions that are negative on the sector $K_0(\sigma_S)$ and constant on $K_0(P)$ for some panel $P$ of $\sigma_S$.
Following the notation in Subsection~\ref{subsec:geo-main}, we write $\Delta(\sigma_S)^{(k)}$ to denote the $(k+1)$-convex hull of $\mathcal{B}(\sigma_S)$ in $S(X_S^{\ast})$.
An application of Theorem~\ref{thm:A-new} to the present setting now gives us the following:
\begin{enumerate}
\item If $\hgt_{\chi}$ does not represent an element of $\Delta(\sigma_S)$, then $((X_S)_{\hgt_{\chi} \geq r})_{r \in \R}$ is contractible for every $r \in \R$.
\item Suppose that $X_S$ satisfies $(\SOL)$.
If $\hgt_{\chi}$ does not represent a class in $\Delta(\sigma_S)^{(k)}$, then $((X_S)_{\hgt_{\chi} \geq r})_{r \in \R}$ is essentially $k$-connected.
\item Suppose there is an $\alpha \in \mathcal{B}$ with $\alpha(\Sigma_S) = \Sigma_S$ such that $h(\alpha(x)) = h(x) + a$ for some $a \in \R \setminus \{0\}$ and every $x \in X$.
If $\hgt_{\chi}$ represents a class in $\Delta(\sigma_S)^{(k)}$, then $((X_S)_{\hgt_{\chi} \geq r})_{r \in \R}$ is not essentially $k$-acyclic.
\end{enumerate}
From the construction of the functions $\hgt_{\alpha,p}$, given in~\eqref{eq:hgt-alpha-p}, it follows that their union $B_{\mathcal{G},\mathcal{B}}^{{ht}}(S)$ is a system of representatives of $\mathcal{B}(\sigma_S)$.
On the other hand, $\hgt_{\chi}$ is the image of the isomorphism $\hgt \colon \Hom(\Gamma,\R) \rightarrow X_S^{\ast}$ that is induced by the bijection
\[
B_{\mathcal{G},\mathcal{B}}(S) \rightarrow B_{\mathcal{G},\mathcal{B}}^{{ht}}(S),\ \chi_{\alpha,p} \mapsto \hgt_{\alpha,p}.
\]
As a consequence, we see that $[\hgt_{\chi}] \in \Delta(\sigma_S)^{(k)}$ if and only if $[\chi] \in \Delta_{\mathcal{G},\mathcal{B}}(S)^{(k)}$.
Note that Lemma~\ref{lem:cocompactness-torus} provides us with an element $\alpha \in \mathcal{B}$ as in the third point.
Thus it remains to show that $X_S$ satisfies $(\SOL)$ if $\Phi$
is of type $A_{n+1}$, $C_{n+1}$, or $D_{n+1}$ and every $p \in S$ satisfies $p \geq 2^{n}$ in the $A_{n+1}$-case, respectively $p \geq 2^{2n+1}$ in the other two cases.
Since $\thick(X_S) = \thick(X_p) = p+1$, this is exactly what is needed in Example~\ref{ex:SOL} to deduce that $X_S$ satisfies $(\SOL)$, which completes the proof.
\end{proof}

As mentioned in the introduction, Theorem~\ref{thm:B} partially confirms the following conjecture of Witzel.

\begin{conjecture}\label{conjecture:large-factors}
Let $\mathcal{G} = \mathcal{G}(\Phi,\rho,\Q)$ be a Chevalley group, let $\Delta \subseteq \Phi$ be a system of simple roots and let $\mathcal{B} \subseteq \mathcal{G}$ be the corresponding Borel subgroup.
Let $S \subseteq \N$ be a finite set of primes.
Then the $\Sigma$-invariants of $\Gamma = \mathcal{B}(\mathcal{O}_S)$ satisfy $\Sigma^k(\Gamma) = S(\Gamma) \backslash \Delta_{\mathcal{G},\mathcal{B}}(S)^{(k-1)}$ for every $k \in \N$.
\end{conjecture}

\bibliographystyle{amsalpha}
\bibliography{Bibliography}

\end{document}